\documentclass[a4paper,11pt,twoside]{article}

\usepackage{geometry}
\usepackage[utf8]{inputenc}
\usepackage{lmodern}
\usepackage[T1]{fontenc}

\usepackage{amsmath,amssymb,amsthm,empheq,cases}
\usepackage{hyperref}
\hypersetup{colorlinks,
            citecolor=red, 
            filecolor=black,
            linkcolor=blue,
            urlcolor=blue}

\def \dis {\displaystyle}
\def \ecart {\noalign{\medskip}}

\theoremstyle{definition}
\newtheorem{theorem}{Theorem}[section]
\newtheorem{proposition}[theorem]{Proposition}
\newtheorem{lemma}[theorem]{Lemma}

\newtheorem{definition}[theorem]{Definition}
\newtheorem{remark}[theorem]{Remark}

\title{Elliptic differential-operators with an abstract Robin boundary condition and two spectral parameters}
\author{Angelo Favini, Rabah Labbas, St\'{e}phane Maingot \& Alexandre
Thorel \\
{\scriptsize A. F.: Universit\`{a} degli Studi di Bologna, Dipartimento di Matematica,}\\
{\scriptsize \  Piazza di Porta S. Donato, 5, 40126 Bologna, Italia} \\
{\scriptsize angelo.favini@unibo.it}\\
{\scriptsize R. L., S. M. \& A. T.: Université Le Havre Normandie, Normandie Univ, LMAH UR 3821,}\\
{\scriptsize 76600 Le Havre, France.}\\
{\scriptsize rabah.labbas@univ-lehavre.fr, stephane.maingot@univ-lehavre.fr, alexandre.thorel@orange.fr}}
\date{\empty }

\begin{document}

\maketitle

\begin{abstract}
We study the solvability of some boundary-value problems for differential-operator equations of the second order in $L^{p}(0,1;X)$, with $1<p<+\infty $, $X$ being a UMD complex Banach space. The originality of this work lies in the fact that we consider the case where two spectral complex parameters appear in the equation and in abstract Robin boundary conditions. Here, the unbounded linear operator in the equation is not commuting with the one appearing in the boundary conditions. This represents the strong novelty with respect to the existing literature. Existence, uniqueness, representation formula, maximal regularity of the solution, sharp estimates and generation of strongly continuous analytic semigroup are proved. Many concrete applications are given for which our theory applies. This paper improves, in some sense, results by the authors in \cite{Cheggag 1} and it can be viewed as a continuation of the results in \cite{Aliev and Kurbanova} studied only in Hilbert spaces. \\
\textbf{Key Words and Phrases}: Second order boundary value problems with two spectral parameters, Robin boundary conditions, spectral estimates, functional calculus, generation of analytic semigroups. \\
\textbf{2020 Mathematics Subject Classification}: 35B65, 35C15, 35J25, 47A60, 47D06. 
\end{abstract} 

\section{Introduction}

In this article, we consider a new spectral problem that is given by the equation
\begin{equation}
u^{\prime \prime }(x)+Au(x)-\lambda u(x)=f(x),\quad x\in (0,1),
\label{NewEquation}
\end{equation}%
together with the abstract Robin boundary conditions%
\begin{equation}
u^{\prime }(0)-Hu(0)-\mu u(0)=d_{0},\quad u(1)=u_{1}.
\label{RobinSpectralConditions}
\end{equation}%
Here, $\lambda $, $\mu $ are complex parameters, $A$, $H$ are closed linear
operators in a complex Banach space $X$, $f$ belongs to $L^{p}(0,1;X)$ with $%
1<p<+\infty $, $d_{0}$, $u_{1}$ are given elements of $X$. We develop a completely different approach from the ones used until now. It allows an easier verification of the assumptions and their application to concrete problems.

Many boundary value problems with a spectral parameter in the equation and
in the boundary conditions arise in different concrete problems. We shall cite some interesting papers related to this research. In one of these last works, see \cite{Bruk}, the article considers a class of boundary problems with a spectral parameter in the boundary conditions. In \cite{J. Behrndt}, the author considers some second order elliptic boundary value problems on bounded domains with boundary conditions depending nonlinearly on the spectral parameter. In \cite{Aliev}, we find a study, in a separable Hilbert space, of the following boundary-value second-order elliptic differential-operator equation: 
\begin{equation*}
\left\{ 
\begin{array}{l}
u^{\prime \prime }(x)+Au(x)-\lambda u(x)=f(x),\text{ \ }x\in (0,1) \vspace{0.1cm}\\ 
\lambda u^{\prime }(0)-\alpha u(1)=f_{1},\text{ }u(1)=f_{2},%
\end{array}%
\right.
\end{equation*}%
where $\alpha $ is a complex number with $\text{Re}(\alpha )\geqslant 0$ and $-A$ is a linear self-adjoint
operator garanteeing the ellipticity of the equation. Note that here, the
parameter $\lambda $ appears in the nonlocal boundary condition. Recently, in \cite{Aliev and Kurbanova}, the authors consider the following
boundary-value problem for an elliptic differential-operator equation of
second order 
\begin{equation*}
\left\{ 
\begin{array}{l}
\lambda ^{2}u(x)-u^{\prime \prime }(x)+Au(x)=f(x),\text{ \ }x\in (0,1)  \vspace{0.1cm}\\ 
u^{\prime }(0)+\lambda u(1)=f_{1},\text{ }\beta u^{\prime }(1)+\lambda u(0)=f_{2},
\end{array}\right.
\end{equation*}
where the same spectral parameter appears in the equation quadratically;
here $-A$ is a closed positive linear operator in a separable complex
Hilbert space. In \cite{Cheggag 1}, the authors consider Problem \eqref{NewEquation}-\eqref{RobinSpectralConditions} in a complex Banach space $X$, where $\lambda =\omega$ is a positive spectral parameter and $\mu = 0$. For $\omega $ large enough, under some geometrical assumptions on the space $X$
and hypotheses on operators $A-\omega I$ and $H$, including the fact that
they commute in the resolvent sense, the authors furnish necessary
and sufficient conditions on the data $d_{0},u_{1}$ to obtain the existence
and uniqueness of a solution $u$ of \eqref{NewEquation}-\eqref{RobinSpectralConditions} with maximal regularity. Recently, in \cite{Cheggag 4}, the authors develop an interesting new approach in a non commutative framework, concerning some general Sturm-Liouville problems with the same Robin boundary condition in $0$.

In our study of Problem \eqref{NewEquation}-\eqref{RobinSpectralConditions},
the ellipticity of the equation is guaranteed by hypothesis \eqref{NewH2}
below; this assumption allows us to consider, for suitable $\lambda ,\mu $,
the operators 
\begin{equation*}
\left\{ 
\begin{array}{l}
\Lambda _{\lambda ,\mu }:=\left( Q_{\lambda }-H_{\mu }\right)
+e^{2Q_{\lambda }}\left( Q_{\lambda }+H_{\mu }\right)  \vspace{0.1cm}\\ 
Q_{\lambda }=-\sqrt{-A+\lambda I},\text{ \ }H_{\mu }=H+\mu I.%
\end{array}\right.
\end{equation*}
In all the sequel, for any closed linear operator $T$ on $X$, $D(T)$ denotes the domain of $T$ and $\rho(T)$ the resolvent set of $T$. The key point will be to obtain the invertibility of the determinant $\Lambda _{\lambda ,\mu}$ of system  \eqref{NewEquation}-\eqref{RobinSpectralConditions} with estimates of $\left\Vert \Lambda _{\lambda ,\mu }^{-1}\right\Vert _{\mathcal{L}(X)}$, for appropriate $\lambda ,\mu $. To this end, we consider
two different situations:
\begin{enumerate}
\item $D(H)\subset D(A)$
\item $D(\sqrt{-A})\subset D(H)$,
\end{enumerate}
where in the first case, we say that operator $H$ is principal, while in the second case, it is operator $\sqrt{-A}$ which is principal. Concrete applications will illustrate these two cases at the end of this work; the first one is adapted to related problems concerning some heat equations with dynamical boundary conditions of reaction-diffusion type or with Wentzell boundary conditions, whereas the second one will concern, for instance, problems involving the Caputo derivative in the boundary conditions.

Four new and essential results sum up this work.
\begin{enumerate}
\item We solve the above equation by giving an explicit and simplified representation of the solution adapted to each case and we show that it verifies the optimal regularity, that is 
$$u\in W^{2,p}(0,1;X)\cap L^{p}(0,1;D(A)),$$ 
see Theorem~\ref{Main1} and Theorem~\ref{Main1Bis}.

\item We give sharp estimates of this solution in each case according to the complex spectral parameters $\lambda ,\mu $ belonging to some appropriate precised set,  see Theorem~\ref{Main2} and Theorem~\ref{Main2bis} . 

This part essentially uses the results of \cite{Dore Yakubov}, where some inequalities on resolvent operators are precised.

\item Thanks to these estimates, we obtain the generation of analytic semigroups corresponding to each case, see Theorem~\ref{gen sg first case} and Theorem~\ref{Gen second case}. 

\item Using the same tools, we study the Dirichlet case and obtain
similar results to those obtained with Robin boundary conditions, see Theorem~\ref{Main3}, Theorem~\ref{main Dirichlet} and Theorem~\ref{Gen Dirichlet}.
\end{enumerate} 

This article is organized as follows. Section 2 describes the assumptions,
including two spectral parameters $\lambda ,\mu $, and enunciates the main results of this paper. In Section 3, we deal with our model without spectral parameter so that we retrieve in a simple manner results of previous works, see \cite{Cheggag 4} and \cite{Kaid Ould}. Section 4 is devoted to some precise estimates of Dore-Yakubov type, which will be useful to analyze our model. Sections 5 and 6 concern the study of our model with spectral parameters $\lambda ,\mu $ under two different types of behaviour concerning operators with respect to their domains and to the parameters. Moreover sharp estimates in $\lambda ,\mu $ are furnished for the solution. In Section~7, we furnish results for (\ref{NewEquation}) together with Dirichlet boundary conditions. Then, in Section~8, we apply the results of Sections 5, 6, 7 to generation of semigroups. Finally, Section 9 deals with examples of applications.

\section{Assumptions and statement of main results}

In all this work, we will use the following notation: for $\varphi \in
\left( 0,\pi \right) $, we set 
\begin{equation}
S_{\varphi }:=\left\{ z\in \mathbb{C}\backslash \left\{ 0\right\}
:\left\vert \arg (z)\right\vert \leqslant \varphi \right\} \cup \left\{
0\right\} .  \label{EnsembleSpectral}
\end{equation}
Our goal is to seek for a classical solution to Problem \eqref{NewEquation}-\eqref{RobinSpectralConditions}, that is a function $u$ such that

\begin{equation*}
\begin{array}{cl}
i) & u\in W^{2,p}(0,1;X)\cap L^{p}(0,1;D(A)), \\ 
ii) & u(0)\in D(H), \\ 
iii)& u\text{ satisfies }\eqref{NewEquation} \text{ and }\eqref{RobinSpectralConditions}.
\end{array}
\end{equation*}
We suppose that 
\begin{equation}
X\text{ is a }UMD\text{ space}.  \label{NewH1}
\end{equation}%
Recall that $X$ is a $UMD$ space means that for all $q>1$ the Hilbert transform is
continuous from $L^{q}(\mathbb{R};X)$ into itself, see \cite{D. L. Burkholder}; we also assume that

\begin{equation}
\left \{ 
\begin{array}{l}
\exists \text{ }\varphi _{0}\in \left( 0,\pi \right) :\text{ }S_{\varphi
_{0}}\subset \rho \left( A\right) \text{ and }\exists \,C_{A}>0:  \vspace{0.1cm}\\ 
\forall \lambda \in S_{\varphi _{0}},\quad \left \Vert \left( A-\lambda
I\right) ^{-1}\right \Vert _{\mathcal{L}(X)}\leqslant \dfrac{C_{A}}{1+\left
\vert \lambda \right \vert },%
\end{array}%
\right.  \label{NewH2}
\end{equation}
and
\begin{equation} \label{MoinsABip}
\left \{ \begin{array}{l}
\forall s\in \mathbb{R},\text{ }\left( -A\right) ^{is}\in \mathcal{L}\left( X\right) ,\text{ }%
\exists \, \theta _{A}\in \left( 0,\pi \right) \text{:} \vspace{0.1cm} \\ 
\underset{s\in \mathbb{R}}{\sup }\left \Vert e^{-\theta _{A}\left \vert s\right \vert
}(-A)^{is}\right \Vert _{\mathcal{L}\left( X\right) }<+\infty .%
\end{array}\right. 
\end{equation}%
We now set for $\lambda \in S_{\varphi _{0}},\mu \in \mathbb{C}$%
\begin{equation*}
H_{\mu }=H+\mu I,\quad Q_{\lambda }=-\sqrt{-A+\lambda I}\quad \text{and} \quad Q=-\sqrt{-A}.
\end{equation*}%
The existence of the previous square roots is ensured by subsection \ref{spect 1}
below and for operator $H$ we consider the two following types of
hypotheses:

\subsection{First case} 
\begin{equation}
D\left( H\right) \subset D(A),  \label{D(H)Principal}
\end{equation}%
and%
\begin{equation}
\left \{ 
\begin{array}{l}
\exists \, \varphi _{1}\in \left( 0,\pi \right) ,\exists \,C_{H}>0: \vspace{0.1cm} \\ 
S_{\varphi _{1}}\text{$\subset \rho $}\left( \text{$-H$}\right) \text{ and}%
\underset{\mu \in S_{\varphi _{1}}}{\sup }\left( 1+\left \vert \mu \right
\vert \right) \left \Vert H_{\mu }^{-1}\right \Vert _{\mathcal{L}%
(X)}\leqslant C_{H}.%
\end{array}%
\right.  \label{NewH3}
\end{equation}%
For $r > 0$, we set
\begin{equation*}
\Omega _{\varphi _{0},\varphi _{1},r}=\left \{ \left( \lambda ,\mu \right)
\in S_{\varphi _{0}}\times S_{\varphi _{1}}:\left \vert \lambda \right \vert
\geqslant r\text{ and }\frac{\left \vert \mu \right \vert ^{2}}{\left \vert
\lambda \right \vert }\geqslant r\right \}.
\end{equation*}
Then, we have the following main results:
\begin{theorem}\label{Main1}
Assume (\ref{NewH1})$\sim $(\ref{NewH3}). Let $d_{0},u_{1}\in X$ and $f\in L^{p}\left( 0,1;X\right) $ with $p \in (1,+\infty)$. Then, there exists $r_0 > 0$ such that for all $\left( \lambda ,\mu \right) \in \Omega _{\varphi _{0},\varphi_{1},r_{0}}$, the two following statements are equivalent:

\begin{enumerate}
\item Problem (\ref{NewEquation})-(\ref{RobinSpectralConditions}) has a
classical solution $u$, that is, 
\begin{equation*}
u\in W^{2,p}\left( 0,1;X\right) \cap L^{p}\left( 0,1;D\left( A\right)
\right) ,\text{ }u(0)\in D(H),
\end{equation*}

and $u$ satisfies (\ref{NewEquation})-(\ref{RobinSpectralConditions}).

\item $u_{1}\in (D(A),X)_{\frac{1}{2p},p}$.

Moreover in this case $u$ is unique and given by (\ref{repre}) where $%
Q,\Lambda $ are replaced by $Q_{\lambda },\Lambda _{\lambda ,\mu }.$
\end{enumerate}
\end{theorem}
Here, $(D(A),X)_{\frac{1}{2p},p}$ denotes the classical real interpolation space equipped with the norm
\begin{equation*}
\left\| w \right\|_{(D(A),X)_{\frac{1}{2p},p}} = \left\| w \right\| + \left( \int_{0}^{+\infty }\left\|t^{1-1/2p}A\left( A-tI\right)^{-1} w \right\|^p \frac{dt}{t}\right)^{1/p}.  
\end{equation*}

\begin{theorem}
\label{Main2}Assume (\ref{NewH1})$\sim $(\ref{NewH3}), $d_{0}\in X$ and $%
u_{1}\in (D(A),X)_{\frac{1}{2p},p}$. Then, there exists a constant $M > 0$ such that, for $\left( \lambda ,\mu \right) \in \Omega
_{\varphi _{0},\varphi _{1},r_{0}}$ and $f\in L^{p}\left( 0,1;X\right) $
with $1<p<+\infty $, the unique classical solution $u$ of (\ref{NewEquation}%
)-(\ref{RobinSpectralConditions}) satisfies%
\begin{equation*}
\max \left\{ \left(1+|\lambda|\right)\left\|u\right\|_{L^p(0,1;X)}, \left \|u''\right\|_{L^p(0,1;X)}, \left\|Q_{\lambda }^{2}u\right\|_{L^p(0,1;X)}\right \} \leqslant
M\alpha \left( d_{0},u_{1},\lambda ,\mu ,f\right) ,
\end{equation*}%
where%
$$
\alpha ( d_{0},u_{1},\lambda ,\mu ,f) = \dfrac{1+\left \vert \lambda \right \vert +\left \vert \mu \right \vert }{1+\left \vert \mu \right \vert }\left( \left \Vert d_{0}\right \Vert +\left \Vert f\right
\Vert _{L^{p}(0,1;X)}\right) + \left \Vert u_{1}\right \Vert _{(D(A),X)_{\frac{1}{2p},p}}+\left \vert \lambda \right \vert ^{1-\frac{1}{2p}}\left \Vert u_{1}\right \Vert .$$
\end{theorem}

Now, define in $Y=L^{p}(0,1;X)$ with $p\in \left( 1,+\infty \right) 
$, the following operator 
\begin{equation*}
\begin{array}{llll}
\mathcal{A}: & D\left( \mathcal{A}\right) \subset Y & \longrightarrow & Y \\ 
& u & \longmapsto & A\left( u\left( \cdot \right) \right) ,%
\end{array}%
\end{equation*}%
with domain
\begin{equation*}
D\left( \mathcal{A}\right) =\left\{ u\in Y:u\left( x\right) \in D(A)\text{
a.e. }x\in (0,1)\text{ and }A\left( u\left( \cdot \right) \right) \in Y\right\}.
\end{equation*}
Here, we consider the Banach space $Z:=Y\times X$. For $\mu \in \mathbb{C}$, we build a linear operator $\mathcal{P}_{A,H,\mu }$ on $Z$, by setting%
\begin{equation*}
\begin{array}{llll}
\mathcal{P}_{A,H,\mu }: & D\left( \mathcal{P}_{A,H,\mu }\right) \subset Z & 
\longrightarrow & Z \\ 
& \left( u,v\right) & \longmapsto & \left( u^{\prime \prime }+\mathcal{A}%
u,u^{\prime }(0\right) -Hv-\mu v),%
\end{array}%
\end{equation*}%
where $D\left( \mathcal{P}_{A,H,\mu }\right) =\left\{ \left( u,v\right) \in
W\times D(H):u\left( 1\right) =0,u(0)=v\right\} $ with%
\begin{equation*}
W=W^{2,p}(0,1;X)\cap L^{p}(0,1;D(A))\subset Y.
\end{equation*}
We then obtain:

\begin{theorem}\label{gen sg first case}
Assume \eqref{NewH1}$\sim $\eqref{NewH3}. Set $\varphi _{2}:=\min \left\{ \varphi _{0},\varphi _{1}\right\}$. Then, for each $\mu \in \mathbb{C}$ with $\left\vert \arg \left( \mu \right) \right\vert <\pi -\varphi_{2}$, we have

\begin{enumerate}
\item $\mathcal{P}_{A,H,\mu }$ is the infinitesimal generator of a $C_{0}$-semigroup.

\item If $\varphi_2 \in [\pi /2,\pi )$, then $\mathcal{P}_{A,H,\mu }$ is the infinitesimal generator of an analytic semigroup.
\end{enumerate}
\end{theorem}

\subsection{Second case} 

\begin{equation} \label{D(Q)Principal}
D(Q)\subset D\left( H\right) , 
\end{equation}

\begin{equation}\label{NewH3bis}
\exists \,\varepsilon \in (0,1/2],~\exists\, C_{H,Q}>0,\text{ }\underset{t\in
\lbrack 0,+\infty )}{\sup }\left( 1+t\right) ^{\varepsilon }\parallel
HQ_{t}^{-1}\parallel _{\mathcal{L}\left( X\right) }\leqslant C_{H,Q},
\end{equation}

\begin{equation} \label{inv Lamb and interp}
\left( Q-H\right) ^{-1}\left( \left( D\left( Q\right) ,X\right)
_{1/p,p}\right) \subset Q^{-1}\left( \left( D\left( Q\right) ,X\right)
_{1/p,p}\right) ;
\end{equation}
here, we have not supposed that $\left(Q-H\right)^{-1}$ is an operator but we have used the set-theory notation: 
\begin{equation*}
\left( Q-H\right) ^{-1}\left( \left( D\left( Q\right) ,X\right)
_{1/p,p}\right) =\left \{ \xi \in D\left( Q\right) :\left( Q-H\right) \xi
\in \left( D\left( Q\right) ,X\right) _{1/p,p}\right \} .
\end{equation*}

In order to obtain spectral estimates for the solution of \eqref{NewEquation}-\eqref{RobinSpectralConditions}, we will replace \eqref{inv Lamb and interp} by the new assumption:
\begin{equation}\label{Q-H stability}
\left( Q-H\right) ^{-1}\left( D\left( Q\right) \right) \subset D\left(
Q^{2}\right) ,  
\end{equation}%
where, as above
$$\left( Q-H\right) ^{-1}\left( D\left( Q\right) \right) =\left \{ \xi
\in D\left( Q\right) :\left( Q-H\right) \xi \in D\left( Q\right) \right \}.$$
Now, for $\rho >0$, we set
\begin{equation*}
\Pi _{\varphi _{0},\rho }=\left\{ \left( \lambda ,\mu \right) \in S_{\varphi
_{0}}\times \mathbb{C}:\left\vert \lambda \right\vert \geqslant \rho \text{ \ and \ }\frac{\left\vert \lambda \right\vert }{\left\vert \mu \right\vert ^{1/\varepsilon }%
}\geqslant \rho \right\}.
\end{equation*}
Then, we have the following main results:
\begin{theorem}\label{Main1Bis}
Assume (\ref{NewH1})$\sim $(\ref{MoinsABip}) and (\ref%
{D(Q)Principal})$\sim $(\ref{inv Lamb and interp}). Let $d_{0},u_{1}\in X$ and $f\in L^{p}\left( 0,1;X\right) $ with $1<p<+\infty$. Then, there exists $\rho_0>0$, such that for all $\left(\lambda ,\mu \right) \in \Pi_{\varphi _{0},\rho _{0}}$, the two following statements are equivalent:

\begin{enumerate}
\item Problem (\ref{NewEquation})-(\ref{RobinSpectralConditions}) has a
classical solution $u$, that is, 
\begin{equation*}
u\in W^{2,p}\left( 0,1;X\right) \cap L^{p}\left( 0,1;D\left( A\right)
\right) ,\text{ }u(0)\in D(H),
\end{equation*}

and $u$ satisfies (\ref{NewEquation}), (\ref{RobinSpectralConditions}).

\item $u_{1}\in (D(A),X)_{\frac{1}{2p},p}$ and $\left( Q_{\lambda }-H_{\mu
}\right) ^{-1}d_{0}\in \left( D\left( A\right) ,X\right) _{\frac{1}{2p},p}.$

Moreover in this case $u$ is unique and given by (\ref{repre}), where $%
Q,\Lambda $ are replaced by $Q_{\lambda },\Lambda _{\lambda ,\mu }.$
\end{enumerate}
\end{theorem}
\begin{theorem}
\label{Main2bis}Assume (\ref{NewH1})$\sim $(\ref{MoinsABip}) and (\ref%
{D(Q)Principal}), (\ref{NewH3bis}), (\ref{Q-H stability}). Let $\left(
\lambda ,\mu \right) \in \Pi _{\varphi _{0},\rho _{0}},d_{0}\in X$ with 
$$\left( Q_{\lambda }-H_{\mu }\right) ^{-1}d_{0}\in \left( D\left( A\right)
,X\right) _{\frac{1}{2p},p},\quad u_{1}\in (D(A),X)_{\frac{1}{2p},p} \quad \text{and} \quad f\in L^{p}\left( 0,1;X\right),$$ 
with $1<p<+\infty $. Then, there exists a
constant $M > 0$, which does not depend on $d_{0},u_{1},\left(
\lambda ,\mu \right) $ and $f$, such that the unique classical solution $u$
of (\ref{NewEquation})-(\ref{RobinSpectralConditions}) satisfies%
\begin{equation*}
\max \left\{ \left(1+|\lambda|\right)\left\|u\right\|_{L^p(0,1;X)}, \left\| u''\right\|_{L^p(0,1;X)}, \left\| Q_{\lambda }^{2}u\right\|_{L^p(0,1;X)}\right\} \leqslant M \beta \left( d_{0},u_{1},\lambda ,\mu ,f\right) ,
\end{equation*}
where 
\begin{eqnarray*}
\beta \left( d_{0},u_{1},\lambda ,\mu ,f\right) &=&\left \Vert d_{0}\right \Vert +\left \Vert f\right \Vert
_{L^{p}(0,1;X)}+\left \Vert \left( Q_{\lambda }-H_{\mu }\right)
^{-1}d_{0}\right \Vert _{(D(A),X)_{\frac{1}{2p},p}} \\
&&+\left \vert \lambda \right \vert ^{1-\frac{1}{2p}}\left \Vert \left(
Q_{\lambda }-H_{\mu }\right) ^{-1}d_{0}\right \Vert +\left \Vert
u_{1}\right \Vert _{(D(A),X)_{\frac{1}{2p},p}}+\left \vert \lambda \right
\vert ^{1-\frac{1}{2p}}\left \Vert u_{1}\right \Vert .
\end{eqnarray*}
\end{theorem}
Now, we define for $\mu \in \mathbb{C}$, operators
\begin{equation}
\begin{array}{llll}
\mathcal{L}_{A,H,\mu }: & D\left( \mathcal{L}_{A,H,\mu }\right) \subset Y & 
\longrightarrow & Y \\ 
& u & \longmapsto & u^{\prime \prime }+\mathcal{A}u,%
\end{array}
\label{Def L A H Mu}
\end{equation}%
where $\mathcal{A}$, $Y$ are defined above and $D\left( \mathcal{L}_{A,H,\mu }\right) $ is the space of functions $u$ satisfying%
\begin{equation*}
\left\{ 
\begin{array}{l}
u\in W^{2,p}(0,1;X)\cap L^{p}(0,1;D(A)) \vspace{0.1cm}\\ 
u\left( 0\right) \in D\left( H\right) \vspace{0.1cm}\\ 
u^{\prime }(0)-Hu(0)-\mu u(0)=u(1)=0.
\end{array}%
\right.
\end{equation*}
We then obtain:
\begin{theorem}\label{Gen second case}
Assume (\ref{NewH1})$\sim $(\ref{MoinsABip}), (\ref%
{D(Q)Principal}), (\ref{NewH3bis}) and (\ref{Q-H stability}). Then, for any $%
\mu \in \mathbb{C}$, we have

\begin{enumerate}
\item $\mathcal{L}_{A,H,\mu }$ is the infinitesimal generator of a $C_{0}$%
-semigroup.

\item If $\varphi _{0}\in \lbrack \pi /2,\pi )$, then $\mathcal{L}%
_{A,H,\mu }$ is the infinitesimal generator of an analytic semigroup.
\end{enumerate}
\end{theorem}

\subsection{Dirichlet case}

Now we consider the spectral problem 
\begin{equation}
\left \{ 
\begin{array}{l}
u^{\prime \prime }(x)+Au(x)-\lambda u(x)=f(x),\text{ \ }x\in (0,1)  \vspace{0.1cm} \\ 
u(0)=u_{0},\text{ }u(1)=u_{1}.%
\end{array}%
\right.  \label{Dirichlet spectral1}
\end{equation}
We first state the following result on existence and uniqueness of the solution.
\begin{theorem}\label{Main3}
Assume \eqref{NewH1}$\sim $\eqref{MoinsABip}. Let $u_{0}$, $u_{1}\in X$, $\lambda \in S_{\varphi _{0}}$ and $f\in L^{p}\left( 0,1;X\right) $ with $1<p<+\infty 
$. Then the two following statements are equivalent:

\begin{enumerate}
\item Problem (\ref{Dirichlet spectral1}) has a classical solution $u$.

\item $u_{0},u_{1}\in (D(A),X)_{\frac{1}{2p},p}$.

Moreover in this case $u$ is unique and given by (\ref{repre Dirichlet})
where $Q$ is replaced by $Q_{\lambda }.$
\end{enumerate}
\end{theorem}
Now, we state our new results on sharp estimates and generation of semigroup. 
\begin{theorem}\label{main Dirichlet}
Assume (\ref{NewH1})$\sim $(\ref{MoinsABip}) and $u_{0},u_{1}\in (D(A),X)_{\frac{1}{2p},p}$. Then, there exists a constant $M > 0$ such that, for $\lambda \in S_{\varphi _{0}}$ and $f\in L^{p}\left( 0,1;X\right) $ with $1<p<+\infty $, the unique classical solution $u$ of problem (\ref{Dirichlet spectral1}) satisfies
\begin{equation*}
\begin{array}{l}
\max \left\{ \left(1+|\lambda|\right) \left\| u\right\|_{L^{p}(0,1;X)}, \left\| u''\right\|_{L^{p}(0,1;X)}, \left\|Q_{\lambda }^2 u\right\|_{L^{p}(0,1;X)}\right\} \\ \ecart
\leqslant M \left( \left\Vert f\right\Vert_{L^{p}(0,1;X)}+\left\Vert u_{0}\right\Vert _{(D(A),X)_{\frac{1}{2p},p}} + \left\Vert u_{1}\right\Vert _{(D(A),X)_{\frac{1}{2p},p}} + |\lambda|^{1-\frac{1}{2p}} \left(\left\| u_{0}\right\| + \left\| u_{1}\right\| \right)\right).
\end{array}
\end{equation*}
\end{theorem}
Now, we define operator 
\begin{equation*}
\begin{array}{llll}
\mathcal{L}_{A}: & D\left( \mathcal{L}_{A}\right) \subset Y & \longrightarrow
& Y \\ 
& u & \longmapsto & u^{\prime \prime }+\mathcal{A}u.%
\end{array}%
\end{equation*}%
where $D\left( \mathcal{L}_{A}\right) =\left\{ u\in W^{2,p}(0,1;X)\cap
L^{p}(0,1;D(A)):u(0)=u(1)=0\right\} $. We then obtain
\begin{theorem}\label{Gen Dirichlet}
Assume \eqref{NewH1}$\sim$\eqref{MoinsABip}. Then, we have

\begin{enumerate}
\item $\mathcal{L}_{A}$ is the infinitesimal generator of a $C_{0}$%
-semigroup.

\item If $\varphi _{0}\in \lbrack \pi /2,\pi )$, then $\mathcal{L}%
_{A}$ is the infinitesimal generator of an analytic semigroup.
\end{enumerate}
\end{theorem}

\subsection{Remarks}

\begin{remark}\label{Rem equivalence hyp}
Assume (\ref{NewH2}) and (\ref{D(Q)Principal}). Let $\left( \lambda ,\mu \right) \in S_{\varphi _{0}}\times \mathbb{C}$. Then assumption \eqref{Q-H stability} is equivalent to
\begin{equation} \label{Q-H stability bis}
\left( Q_{\lambda }-H_{\mu }\right) ^{-1}\left( D\left( Q\right) \right)
\subset D\left( Q^{2}\right) . 
\end{equation}
In fact, if $\xi \in D\left( Q\right) $, there exists $\zeta \in X$ such that $\xi=Q^{-1}\zeta $, so 
\begin{eqnarray*}
\left( Q-H\right) \xi &=&\left[ Q_{\lambda }-H_{\mu }+\mu I+\left(
Q-Q_{\lambda }\right) \right] \xi \\ \ecart
&=&\left( Q_{\lambda }-H_{\mu }\right) \xi +\mu \xi +\left( Q-Q_{\lambda
}\right) Q^{-1}\zeta ,
\end{eqnarray*}
and it will be seen in Lemma \ref{Compar racine} below that there exists $T_{\lambda }\in \mathcal{L}%
(X)$ such that 
\begin{equation*}
Q=Q_{\lambda }+T_{_{\lambda }}\text{ \ and \ }Q^{-1}T_{_{\lambda}}=T_{_{\lambda }}Q^{-1}.
\end{equation*}
Therefore
\begin{equation*}
\left( Q-H\right) \xi -\left( Q_{\lambda }-H_{\mu }\right) \xi =\mu \xi
+Q^{-1}T_{_{\lambda }}\zeta \in D\left( Q\right) .
\end{equation*}
This proves that 
$$\left( Q-H\right) ^{-1}\left( D\left( Q\right) \right)
=\left( Q_{\lambda }-H_{\mu }\right) ^{-1}\left( D\left( Q\right) \right) .$$
\end{remark}

\begin{remark}\label{second case}
Assume (\ref{NewH2}). If there exists $\omega \in [0,1/2)$ such that $D(\left(-A\right)^{\omega})\subset D\left( H\right) $ then, in virtue of Lemma 2.6 statement a) in \cite{Dore Yakubov}, there exists $C_{\omega }>0$ such that, for $%
t\geqslant 0$%
\begin{equation*}
\begin{array}{lll}
\parallel HQ_{t}^{-1}\parallel _{\mathcal{L}\left( X\right) } & \leqslant&
\parallel H\left( -A\right) ^{-\omega }\parallel _{\mathcal{L}\left(
X\right) }\parallel \left( -A\right) ^{\omega }\left( -A+tI\right)
^{-1/2}\parallel _{\mathcal{L}\left( X\right) } \vspace{0.1cm} \\ 
& \leqslant &\dfrac{C_{\omega }}{\left( 1+t\right) ^{1/2-\omega }};
\end{array}
\end{equation*}
so we have (\ref{D(Q)Principal}) and (\ref{NewH3bis}) with $\varepsilon
=1/2-\omega \in (0,1/2].$
\end{remark}

\begin{remark}
In this paper we have supposed that $0 \in \rho(A)$, but in the theorems above written for $\left\vert \lambda \right\vert $ large enough, and those
concerning generation of semigroups, we can drop this invertibility
assumption; more precisely all the previous Theorems remain true if we replace (\ref{NewH2}) by 
\begin{equation*}
\left\{ \begin{array}{l}
\exists \varphi _{0}\in \left( 0,\pi \right),~\exists \omega _{0}>0:~
S_{\varphi _{0}}\subset \rho \left( A-\omega _{0}I\right) \text{ and } \exists C_{A}>0: \vspace{0.1cm}\\ 
\forall \lambda \in S_{\varphi _{0}},\quad\left\Vert \left( A-\omega
_{0}I-\lambda I\right) ^{-1}\right\Vert _{\mathcal{L}(X)}\leqslant \dfrac{%
C_{A}}{1+\left\vert \lambda \right\vert }.
\end{array}\right. 
\end{equation*}
\end{remark}

\section{Problem without parameters\label{sans parametre}}

In this section we study a problem similar to \eqref{NewEquation}-\eqref{RobinSpectralConditions}, but without the parameters $\lambda $ and $\mu $:
\begin{equation}
\left \{ 
\begin{array}{l}
u^{\prime \prime }(x)+Au(x)=f(x),\quad x\in (0,1)  \vspace{0.1cm}\\ 
u^{\prime }(0)-Hu(0)=d_{0},\text{ }u(1)=u_{1}.
\end{array}%
\right.  \label{Problem}
\end{equation}

\subsection{Hypotheses}

Here our hypotheses are

\begin{description}
\item[$(H_{1})$] $X$ is a $UMD$ space,

\item[$(H_{2})$] $[0,+\infty )\subset \rho \left( A\right) \quad$and$\quad\underset{%
t\in \lbrack 0,+\infty )}{\sup }\left \Vert \left( A-tI\right)
^{-1}\right
\Vert _{\mathcal{L}(X)}\leqslant \dfrac{C}{1+t}$,

\item[$(H_{3})$] $\forall s\in \mathbb{R},$ $\left( -A\right) ^{is}\in \mathcal{L}\left( X\right) $ and
\begin{equation*}
\exists \theta _{A}\in \left( 0,\pi \right) : \quad \underset{
s\in \mathbb{R}}{\sup }\left \Vert e^{-\theta _{A}\left \vert s\right \vert
}(-A)^{is}\right \Vert _{\mathcal{L}\left( X\right) }<+\infty ,
\end{equation*}

\item[$(H_{4})$] $\Lambda :=\left( Q-H\right) +e^{2Q}\left( Q+H\right) $ is
closed and boundedly invertible, where $Q:=-\sqrt{-A}$.

\item[$(H_{5})$] $Q\Lambda ^{-1}\left( \left( D\left( Q\right) ,X\right)
_{1/p,p}\right) \subset \left( D\left( Q\right) ,X\right) _{1/p,p}.$
\end{description}

Note that here we are neither in case 1, nor in case 2. In the following remark we discuss about assumption $(H_{5})$.

\begin{remark}
\label{discuss H5}~

\begin{enumerate}
\item Assume $(H_{2})$ and $(H_{4})$. If we suppose moreover 
\begin{equation}
A^{-1}\Lambda ^{-1}=\Lambda ^{-1}A^{-1},  \label{Commut A inv Lambda}
\end{equation}%
then $(H_{5})$ is satisfied. In fact, the following assertions are
equivalent.

\begin{enumerate}
\item $A^{-1}\Lambda ^{-1}=\Lambda ^{-1}A^{-1},$

\item $\forall \lambda \in \rho \left( A\right) ,\quad \left( A-\lambda I\right)
^{-1}\Lambda ^{-1}=\Lambda ^{-1}\left( A-\lambda I\right) ^{-1},$

\item $Q^{-1}\Lambda ^{-1}=\Lambda ^{-1}Q^{-1}.$
\end{enumerate}

Then, under (\ref{Commut A inv Lambda}), we have $Q\left( Q-tI\right) ^{-1}Q\Lambda ^{-1}=Q\Lambda ^{-1}Q\left( Q-tI\right)^{-1}$ and since 
\begin{equation*}
\left( D\left( Q\right) ,X\right) _{1/p,p}=\left \{ x\in X:t^{1-1/p}Q\left(
Q-tI\right) ^{-1}x\in L_{\ast }^{p}(\mathbb{R}_{+};X)\right \} ,
\end{equation*}%
we get $\left( H_{5}\right)$. Finally we remark that if%
\begin{equation}
\forall \zeta \in D\left( H\right) ,\quad A^{-1}\zeta \in D\left( H\right) \quad\text{and}\quad A^{-1}H\zeta =HA^{-1}\zeta ,  \label{commut A H}
\end{equation}%
then (\ref{Commut A inv Lambda}) and $\left( H_{5}\right) $ are satisfied.

\item Assume $(H_{2})$ and $(H_{4})$; if we use the classical notation
\begin{equation*}
(D(Q),X)_{1+1/p,p}\;:=\;\left\{\psi\in D(Q) : Q \psi\in (D(Q),X)_{1/p,p}\right\},
\end{equation*}
then $(H_{5})$ writes 
\begin{equation*}
\Lambda ^{-1}\left( \left( D\left( Q\right) ,X\right) _{1/p,p}\right)
\subset \left( D\left( Q\right) ,X\right) _{1+1/p,p}.
\end{equation*}

\item In \cite{Cheggag 2}, problem (\ref{Problem}) has been studied under
more restrictive assumptions, that are $(H_{1})\sim (H_{4})$ and the
commutativity hypothesis%
\begin{equation}
\exists \,\lambda _{0}\in \rho (H):\quad A^{-1}\left( H-\lambda _{0}I\right)
^{-1}=\left( H-\lambda _{0}I\right) ^{-1}A^{-1},  \label{K2}
\end{equation}

which, from statement 1., implies $(H_{5})$.

\item Assume $(H_{2})$ and $(H_{4})$. If $Q-H$ is boundedly invertible\
then, due\ to $\Lambda \Lambda ^{-1}=I$ and $\Lambda ^{-1}\Lambda =I$, we
get that 
\begin{equation}
\left \{ 
\begin{array}{l}
\Lambda ^{-1}=\left( Q-H\right) ^{-1}-\left( Q-H\right) ^{-1}e^{2Q}\left(
Q+H\right) \Lambda ^{-1}  \vspace{0.1cm} \\ 
\left( Q-H\right) ^{-1}=\Lambda ^{-1}+\Lambda ^{-1}e^{2Q}\left( Q+H\right)
\left( Q-H\right) ^{-1},%
\end{array}%
\right.  \label{compar Lambda et Q-H}
\end{equation}%
from which we deduce that for any $\xi \in \left( D\left( Q\right) ,X\right)
_{1/p,p}$%
\begin{equation*}
\left[ Q\Lambda ^{-1}\xi \in \left( D\left( Q\right) ,X\right) _{1/p,p}%
\right] \Longleftrightarrow \left[ Q\left( Q-H\right) ^{-1}\xi \in \left(
D\left( Q\right) ,X\right) _{1/p,p}\right];
\end{equation*}
so we can replace in the previous proposition assumption $(H_{5})$ by the
equivalent one%
\begin{equation}\label{Equiv Interp}
(H'_5) \qquad Q\left( Q-H\right) ^{-1}\left( \left( D\left(
Q\right) ,X\right) _{1/p,p}\right) \subset \left( D\left( Q\right) ,X\right)
_{1/p,p},  
\end{equation}

\item Assume $(H_{2})$ and $(H_{4})$. If we suppose that $Q\Lambda ^{-1}\xi \in D(Q)$, then we have $(H_{5})$, see Lemma 5 p. 76 in \cite{Cheggag 4}. Similarly, when $Q-H$ is boundedly invertible, using \eqref{compar Lambda et
Q-H} we have that%
\begin{equation*}
\forall \xi \in D(Q), \quad Q\left( Q-H\right) ^{-1}\xi \in D(Q),
\end{equation*}%
implies $(H_{5}^{\prime })$ and then $(H_{5})$.
\end{enumerate}
\end{remark}

\subsection{Interpolation spaces}

Let us give now some necessary conditions to obtain a classical solution for
our problem (\ref{Problem}) using known properties of interpolation spaces.

\begin{lemma}
Suppose that Problem (\ref{Problem}) has a classical solution $u$. Then:

\begin{enumerate}
\item $u\left( 0\right) ,u\left( 1\right) \in \left( D(Q^{2}),X\right) _{%
\frac{1}{2p},p}=\left( X,D(Q^{2})\right) _{1-\frac{1}{2p},p}$, which implies
that%
\begin{equation*}
u\left( 0\right),~u\left( 1\right) \in D(Q)\quad \text{and}\quad Qu\left( 0\right),~ Qu\left( 1\right) \in \left( D\left( Q\right) ,X\right) _{1/p,p}.
\end{equation*}

\item $u^{\prime }\left( 0\right),~u^{\prime }\left( 1\right) \in \left(
D(Q^{2}),X\right) _{\frac{1}{2}+\frac{1}{2p},p}=\left( D(Q),X\right) _{\frac{%
1}{p},p}.$
\end{enumerate}
\end{lemma}

\begin{proof}
Suppose that Problem (\ref{Problem}) has a classical solution $u$. Then,
from 
\begin{equation*}
u\in W^{2,p}\left( 0,1;X\right) \cap L^{p}\left( 0,1;D(Q^{2})\right)
,\quad 1<p<+\infty ,
\end{equation*}%
we have $u\left( 0\right) ,u\left( 1\right) \in \left( D(Q^{2}),X\right) _{\frac{1}{2p},p}=\left( X,D(Q^{2})\right) _{1-\frac{1}{2p},p}$, see \cite{P. Grisvard}, Teorema 2', p. 678. But, due to \cite{P. Grisvard}, Teorema 6, p. 676, we obtain
$$\left( D\left( Q^{2}\right) ,X\right) _{1/2p,p} = \left( D\left( Q\right) ,X\right) _{1+1/p,p} = \left \{ \varphi \in D\left( Q\right) :Q\varphi \in \left( D\left(Q\right) ,X\right) _{1/p,p}\right \} \subset D\left( Q\right) ,$$
from which it follows that%
\begin{equation*}
u\left( 0\right) ,u\left( 1\right) \in D(Q)\quad \text{and} \quad Qu\left( 0\right)
,Qu\left( 1\right) \in \left( D\left( Q\right) ,X\right) _{1/p,p}.
\end{equation*}%
Similarly, by using Teorema 2', in \cite{P. Grisvard}, p. 678, we have%
\begin{equation*}
u^{\prime }\left( 0\right) ,u^{\prime }\left( 1\right) \in \left(
D(Q^{2}),X\right) _{\frac{1}{2}+\frac{1}{2p},p}=\left( D(Q),X\right) _{\frac{%
1}{p},p}.
\end{equation*}
\end{proof}

\subsection{Representation formula\label{Sect Repr}}

Under $(H_{2})$, if $u$ is a classical solution of (\ref{Problem}) then
there exist $\xi _{0},\xi _{1}\in X$ \ such that 
\begin{equation}
u(x)=e^{xQ}\xi _{0}+e^{\left( 1-x\right) Q}\xi _{1}+I\left( x\right)
+J\left( x\right) ,\quad x\in \left[ 0,1\right] ,  \label{repr initiale}
\end{equation}%
where%
\begin{equation}
I\left( x\right) =\dfrac{1}{2}Q^{-1}\int_{0}^{x}e^{(x-s)Q}f(s)ds\quad \text{and}\quad J\left( x\right) =\dfrac{1}{2}Q^{-1}\int_{x}^{1}e^{(s-x)Q}f(s)ds,
\label{def I et J sans parametre}
\end{equation}%
see \cite{Cheggag 1}, p. 989. Note that here, unlike \cite{Cheggag 1}, we do
not suppose that $A$ and $H$ commute. Now, taking into account the fact
that $I-e^{2Q}$ is invertible, we set $T=\left( I-e^{2Q}\right) ^{-1}\in 
\mathcal{L}\left( X\right) $ and 
\begin{equation*}
S\left( x\right) =T\left( e^{xQ}-e^{\left( 1-x\right) Q}e^{Q}\right) \in 
\mathcal{L}\left( X\right) ,\quad x\in \left[ 0,1\right];
\end{equation*}
then formula (\ref{repr initiale}) takes the following form 
\begin{equation*}
u(x)=S\left( x\right) \mu _{0}+S\left( 1-x\right) \mu _{1}+I\left( x\right)
+J\left( x\right) ,\quad x\in \left[ 0,1\right] ,
\end{equation*}
with $\mu _{0}=\xi _{0}+e^{Q}\xi _{1},\mu _{1}=e^{Q}\xi _{0}+\xi _{1}$ and
we deal with this new writing. We note that%
\begin{equation*}
\left \{ 
\begin{array}{l}
u_{0}=u(0)=\mu _{0}+J\left( 0\right) \\ 
u_{1}=u(1)=\mu _{1}+I\left( 1\right) \\ 
u^{\prime }\left( 0\right) =TQ\left( I+e^{2Q}\right) \mu _{0}-2TQe^{Q}\mu_{1}-QJ\left( 0\right) ,
\end{array}\right.
\end{equation*}%
and we determine $\mu _{0},\mu _{1}$ by using the boundary conditions $u(1)=u_{1}$ and $u^{\prime }(0)-Hu(0)=d_{0}$. So $\mu _{1}=u_{1}-I\left( 1\right) $ and
\begin{equation*}
TQ\left( I+e^{2Q}\right) \mu _{0}-2TQe^{Q}\mu _{1}-QJ\left( 0\right)
-H\left( \mu _{0}+J\left( 0\right) \right) =d_{0};
\end{equation*}%
hence
$$\begin{array}{lll}
TQ\left( I+e^{2Q}\right) \left( \mu _{0}+J\left( 0\right) \right) -H\left(
\mu _{0}+J\left( 0\right) \right) &=& d_{0}+2TQe^{Q}\mu _{1} \\
&& +QJ\left( 0\right) +TQ\left( I+e^{2Q}\right)J\left( 0\right),
\end{array}$$
thus
\begin{eqnarray*}
\left[ Q\left( I+e^{2Q}\right) -\left( I-e^{2Q}\right) H\right] \left( \mu
_{0}+J\left( 0\right) \right) &=& \left( I-e^{2Q}\right) d_{0}+2Qe^{Q}\mu _{1}\\
&& +\left( I-e^{2Q}\right)QJ\left( 0\right) +Q\left( I+e^{2Q}\right) J\left( 0\right);
\end{eqnarray*}%
but $\Lambda =Q\left( I+e^{2Q}\right) -\left( I-e^{2Q}\right) H$, so
\begin{equation*}
\left \{ 
\begin{array}{l}
\mu _{1}=u_{1}-I\left( 1\right) \vspace{0.1cm} \\ 
\mu _{0}=\Lambda ^{-1}\left[ \left( I-e^{2Q}\right) d_{0}+2Qe^{Q}\mu
_{1}+2QJ\left( 0\right) \right] -J\left( 0\right) .%
\end{array}%
\right.
\end{equation*}%
Finally, if $u$ is a classical solution to \eqref{Problem}, then 
\begin{equation}
u(x)=S\left( x\right) \mu _{0}+S\left( 1-x\right) \mu _{1}+I\left( x\right)
+J\left( x\right) ,\text{ \ }x\in \left[ 0,1\right] ,  \label{repre}
\end{equation}%
where%
\begin{equation}
\left \{ 
\begin{array}{l}
\mu _{1}=u_{1}-I\left( 1\right)\vspace{0.1cm} \\ 
\mu _{0}=\Lambda ^{-1}\left[ \left( I-e^{2Q}\right) d_{0}+2Qe^{Q}\mu
_{1}+2QJ\left( 0\right) \right] -J\left( 0\right)\vspace{0.1cm} \\ 
S\left( x\right) =\left( I-e^{2Q}\right) ^{-1}\left( e^{xQ}-e^{\left(
1-x\right) Q}e^{Q}\right) .%
\end{array}%
\right.  \label{repre suite}
\end{equation}

When (\ref{K2}) is satisfied, we can check that this representation formula
coincides with the one given in \cite{Cheggag 2} p. 528. We can also, after
computations, verify that (\ref{repre}) is the same formula as the one in \cite{Cheggag 4}, p. 92 (with $L=M=Q$) and also compare it with (34)$\sim $(38)
pp. 54-55, in \cite{Kaid Ould}.

\subsection{Regularity results}

The following results will be useful to study the regularity of the solution
of (\ref{Problem}).

\begin{lemma}
\label{Trieb}Let $p\in \left( 1,+\infty \right) ,\psi \in X$ and $n\in \mathbb{N} \backslash \left \{ 0\right \} $. Then, under $(H_{2})$, we have

\begin{enumerate}
\item $x\mapsto e^{xQ}\psi \in L^{p}\left( 0,1,X\right) .$

\item $x\mapsto Q^{n}e^{xQ}\psi \in L^{p}\left( 0,1,X\right) $ if and only
if $\psi \in \left( D\left( Q^{n}\right) ,X\right) _{\frac{1}{np},p}.$
\end{enumerate}
\end{lemma}

See for instance \cite{H. Triebel}, Theorem at p. 96.

\begin{lemma}
\label{Dore Venni}For $f\in L^{p}\left( 0,1,X\right) $ with $1<p<+\infty $,
under $(H_{1})\sim (H_{3})$, we have

\begin{enumerate}
\item $\dis x\mapsto Q \int \nolimits_{0}^{x}e^{\left( x-s\right) Q}f\left(
s\right) ds\in L^{p}\left( 0,1,X\right) ~~ \text{and} ~~\dis x\mapsto Q\int \nolimits_{x}^{1}e^{\left( s-x\right) Q}f\left( s\right)
ds\in L^{p}\left( 0,1,X\right) .$

\item $\dis x\mapsto Q\int \nolimits_{0}^{1}e^{\left( x+s\right) Q}f\left(
s\right) ds\in L^{p}\left( 0,1,X\right) .$
\end{enumerate}
\end{lemma}

For statements 1 and 2 which are consequences of the Dore-Venni Theorem, see 
\cite{Favin et al2}, p.~167-168 and also (24), (25) and (26) in \cite{Favini
et al1}.

\begin{lemma}
\label{REG S}Let $\psi ,\chi \in X$ and $1<p<\mathbb{+}\infty $. Then, under 
$(H_{2})$, we have

\begin{enumerate}
\item $x\mapsto Q^{2}S\left( x\right) \psi \in L^{p}(0,1;X)$ $%
\Longleftrightarrow $ $\psi \in \left( D\left( Q^{2}\right) ,X\right) _{%
\frac{1}{2p},p},$

$x\mapsto Q^{2}S\left( 1-x\right) \chi \in L^{p}(0,1;X)$ $%
\Longleftrightarrow $ $\chi \in \left( D\left( Q^{2}\right) ,X\right) _{%
\frac{1}{2p},p}.$

\item $x\mapsto Q^{2}S\left( x\right) \psi +Q^{2}S\left( 1-x\right) \chi \in
L^{p}(0,1;X)\Longleftrightarrow \psi ,\chi \in \left( D\left( Q^{2}\right)
,X\right) _{\frac{1}{2p},p}.$
\end{enumerate}
\end{lemma}

\begin{proof}
~

\begin{enumerate}
\item Since $T=\left( I-e^{2Q}\right) ^{-1}=I+e^{2Q}\left( I-e^{2Q}\right) ^{-1}$,
we have
\begin{equation*}
S\left( x\right) =e^{xQ}+\left( I-e^{2Q}\right) ^{-1}e^{xQ}e^{2Q}-\left(
I+e^{2Q}\left( I-e^{2Q}\right) ^{-1}\right) e^{\left( 1-x\right) Q}e^{Q};
\end{equation*}%
then, by Lemma \ref{Trieb}
\begin{equation*}
\begin{array}{lll}
Q^{2}S\left( \cdot \right) \psi \in L^{p}(0,1;X) & \iff & Q^{2}e^{\cdot
Q}\psi \in L^{p}(0,1;X)\vspace{0.1cm} \\ 
& \iff & \psi \in \left( D\left( Q^{2}\right) ,X\right) _{\frac{1}{2p},p}.%
\end{array}%
\end{equation*}

\item For any $\psi ,\chi \in X,$ we have 
\begin{equation}\label{Q^2S + Q^2S(1-)}
Q^{2}S\left( \cdot \right) \psi +Q^{2}S\left( 1-\cdot \right) \chi \in
L^{p}(0,1;X),
\end{equation}
if and only if $Q^{2}S\left( \cdot \right) \psi \in L^{p}(0,1;X)$ and $Q^{2}S\left( 1-\cdot \right) \chi \in L^{p}(0,1;X)$. 

In fact, assume \eqref{Q^2S + Q^2S(1-)}, then
$$Q^{2}S\left( \cdot \right) \psi +Q^{2}S\left( 1-\cdot \right) \chi \in
L^{p}(1/2,1;X),$$
but $Q^{2}S\left( \cdot \right) \psi \in L^{p}(1/2,1;X)$, therefore $Q^{2}S\left( 1-\cdot \right) \chi \in L^{p}(1/2,1;X)$, which gives 
$$Q^{2}S\left( 1-\cdot \right) \chi \in L^{p}(0,1;X).$$
\end{enumerate}
\end{proof}

\begin{lemma}
\label{mu}Consider $\mu _{0},\mu _{1}$ defined in (\ref{repre suite}). Then%
\begin{equation*}
\left \{ 
\begin{array}{l}
\mu _{0}\in \left( D\left( A\right) ,X\right) _{\frac{1}{2p},p}\iff \Lambda
^{-1}d_{0}\in \left( D\left( A\right) ,X\right) _{\frac{1}{2p},p} \\ 
\mu _{1}\in \left( D\left( A\right) ,X\right) _{\frac{1}{2p},p}\iff u_{1}\in
\left( D\left( A\right) ,X\right) _{\frac{1}{2p},p}.%
\end{array}%
\right.
\end{equation*}
\end{lemma}

\begin{proof}
From \cite{Hammou}, Proposition 3.5, p. 1676, we have $J\left( 0\right)
,I\left( 1\right) \in \left( D\left( A\right) ,X\right) _{\frac{1}{2p},p}$,
thus: 
$$\mu _{1}\in \left( D\left( A\right) ,X\right) _{\frac{1}{2p},p}\iff
u_{1}\in \left( D\left( A\right) ,X\right) _{\frac{1}{2p},p}.$$
Moreover $\mu _{0}=\Lambda ^{-1}d_{0}-\Lambda ^{-1}\left[
e^{2Q}d_{0}+2Qe^{Q}\mu _{1}+2QJ\left( 0\right) \right] -J\left( 0\right) $,
with 
\begin{equation*}
e^{2Q}d_{0}+2Qe^{Q}\mu _{1}+2QJ\left( 0\right) \in \left( D\left( Q\right)
,X\right) _{1/p,p},
\end{equation*}%
and from $(H_{5})$%
\begin{equation*}
Q\Lambda ^{-1}\left[ e^{2Q}d_{0}+2Qe^{Q}\mu _{1}+2QJ\left( 0\right) \right]
\in \left( D\left( Q\right) ,X\right) _{1/p,p},
\end{equation*}%
which means that 
\begin{equation*}
\Lambda ^{-1}\left[ e^{2Q}d_{0}+2Qe^{Q}\mu _{1}+2QJ\left( 0\right) \right]
\in \left( D\left( Q\right) ,X\right) _{1+1/p,p}=\left( D\left( A\right)
,X\right) _{\frac{1}{2p},p}.
\end{equation*}%
Finally: $\mu _{0}\in \left( D\left( A\right) ,X\right) _{\frac{1}{2p}%
,p}\iff \Lambda ^{-1}d_{0}\in \left( D\left( A\right) ,X\right) _{\frac{1}{2p%
},p}.$
\end{proof}

\subsection{Resolution of Problem (\protect\ref{Problem})}

\begin{proposition}
\label{Prop without lambda}Let $f\in L^{p}(0,1;X)$ with $1<p<\mathbb{+}%
\infty $ and assume that $(H_{1})\sim (H_{5})$ are satisfied. Then the
following assertions are equivalent:

\begin{enumerate}
\item Problem (\ref{Problem}) admits a classical solution $u$.

\item $u_{1},\Lambda ^{-1}d_{0}\in \left( D\left( A\right) ,X\right) _{\frac{%
1}{2p},p}$

Moreover in this case $u$ is unique and it is given by (\ref{repre}).
\end{enumerate}
\end{proposition}

\begin{proof}
From the previous study we know that if Problem (\ref{Problem}) admits a
classical solution $u$ then $u$ is unique and given by (\ref{repre}).
Moreover $u$ defined by (\ref{repre}) satisfies%
\begin{equation*}
u(0)=\mu _{0}+J\left( 0\right) =\Lambda ^{-1}\left[ \left( I-e^{2Q}\right)
d_{0}+2Qe^{Q}\mu _{1}+2QJ\left( 0\right) \right] \in D\left( H\right) ,
\end{equation*}%
and then $u$ is a classical solution of\ (\ref{Problem}) if and only if $%
Q^{2}u\left( \cdot \right) \in L^{p}(0,1;X)$.

But from Lemma \ref{Dore Venni}, $Q^{2}I\left( \cdot \right) ,Q^{2}J\left(
\cdot \right) \in L^{p}(0,1;X)$, so Lemma \ref{REG S} and Lemma \ref{mu}
imply that%
\begin{equation*}
\begin{array}{lll}
Q^{2}u\left( \cdot \right) \in L^{p}(0,1;X) & \iff & Q^{2}S\left( x\right)
\mu _{0}+Q^{2}S\left( 1-x\right) \mu _{1}\in L^{p}(0,1;X) \vspace{0.1cm}\\ 
& \iff & \mu _{0},\mu _{1}\in \left( D\left( A\right) ,X\right) _{\frac{1}{2p%
},p} \vspace{0.1cm}\\ 
& \iff & u_{1},\Lambda ^{-1}d_{0}\in \left( D\left( A\right) ,X\right) _{%
\frac{1}{2p},p};
\end{array}
\end{equation*}
this proves that statement 1. is equivalent to statement 2.
\end{proof}

\begin{remark}
\label{simplif}In the previous Proposition, if moreover, $Q-H$ is boundedly
invertible, then using (\ref{compar Lambda et Q-H}) we can replace the
condition $\Lambda ^{-1}d_{0}\in \left( D\left( A\right) ,X\right) _{\frac{1%
}{2p},p}$ by the simpler one $\left( Q-H\right) ^{-1}d_{0}\in \left(
D\left( A\right) ,X\right) _{\frac{1}{2p},p}$.
\end{remark}

\section{Dore-Yakubov estimates\label{Dore Yaku}}
This section is devoted to Dore-Yakubov Estimates and applications. The
results are based on those given in \cite{Dore Yakubov} and we have used the
definitions and notations from this paper. We consider here a complex Banach
space $E$, which is not necessarily a UMD space.

\begin{definition}
$W$ is an operator of type $\phi \in (0,\pi )$ with bound $C$, if $W:D\left(
W\right) \subset E\longrightarrow E$ is a closed linear operator such that $%
S_{\phi }\subset \rho \left( -W\right) $ and%
\begin{equation*}
\forall \lambda \in S_{\phi },\text{ }\left\Vert \left( W+\lambda I\right)
^{-1}\right\Vert _{\mathcal{L}(E)}\leqslant \dfrac{C}{1+\left\vert \lambda
\right\vert },
\end{equation*}%
where $S_{\phi }$ is defined by (\ref{EnsembleSpectral}).
\end{definition}

In all this section,\ we fix $\varphi $\ in $(0,\pi )$, $\lambda $ in $%
S_{\varphi }$ and $L:D\left( L\right) \subset E\longrightarrow E$ an
operator of type $\varphi $ with bound $C_{L}$, satisfying $\overline{%
D\left( L\right) }=E$.

We set%
\begin{equation*}
D_{\lambda }:=L+\lambda I\text{ and }\varepsilon \left( \varphi \right)
:=\min \left\{ \varphi ,\pi -\varphi \right\} \in (0,\pi /2).
\end{equation*}%
Note that $\varepsilon \left( \varphi \right) =\varphi $ if $\varphi \in
(0,\pi /2]$, and $\varepsilon \left( \varphi \right) =\pi -\varphi $ if $%
\varphi \in \lbrack \pi /2,\pi )$.

\begin{lemma}
\label{estim racine L}~

\begin{enumerate}
\item Let $\theta \in \left( 0,\varepsilon \left( \varphi \right) \right) $;
then $D_{\lambda }$ is an operator of type $\theta $ with bound $\dis C_{\theta
}:=C_{L}/\cos \left( \dfrac{\varphi +\theta }{2}\right)$.

Moreover for $\nu \in S_{\theta }$ we have : $\left\Vert \left( D_{\lambda
}+\nu I\right) ^{-1}\right\Vert _{\mathcal{L}(E)}\leqslant \dfrac{C_{\theta }%
}{\left\vert \lambda \right\vert +\left\vert \nu \right\vert +1}.$

\item Let $\overline{\theta }\in \mathbb{R}$ and $\nu \in \mathbb{C}$ such that $\theta =\left\vert \overline{\theta }\right\vert \in \left(0,\varepsilon \left( \varphi \right) \right) $ and $\text{Re}\left( \nu e^{-i\overline{\theta }/2}\right) >0$. Then $D_{\lambda }^{1/2}+\nu I$ is
boundedly invertible and $\left\Vert \left( D_{\lambda }^{1/2}+\nu I\right) ^{-1}\right\Vert _{\mathcal{L}(E)}\leqslant \dfrac{C_{\nu ,\overline{\theta }}}{\left\vert \nu \right\vert +\sqrt{\left\vert \lambda \right\vert +1}},$ with $C_{\nu ,\overline{\theta }}:=C_{L}/\left[ \cos \left( \arg \left( \nu
\right) -\frac{\overline{\theta }}{2}\right) \cos \left( \dfrac{\varphi
+\theta }{2}\right) \right] .$

\item Let $\psi \in \left( \dfrac{\pi }{2},\dfrac{\pi }{2}+\dfrac{%
\varepsilon \left( \varphi \right) }{2}\right) $; then $D_{\lambda }^{1/2}$
is a linear operator of type $\psi$ with bound \\
$K_{\psi }:=C_{L}/\cos ^{2}\left(\beta _{\psi }\right)$ where $\beta _{\psi }=\dfrac{\pi }{4}+\dfrac{\psi
-\varepsilon \left( \varphi \right) }{2}\in \left( 0,\pi /2\right) .$

Moreover for $\nu \in S_{\psi }$ we have%
\begin{equation}
\left\Vert \left( D_{\lambda }^{1/2}+\nu I\right) ^{-1}\right\Vert _{%
\mathcal{L}(E)}\leqslant \dfrac{K_{\psi }}{\left\vert \nu \right\vert +\sqrt{%
\left\vert \lambda \right\vert +1}}.  \label{est gen}
\end{equation}
\end{enumerate}
\end{lemma}

\begin{proof}
This Lemma is essentially based on Lemmas 2.3 and 2.4 in \cite{Dore Yakubov}%
. The novelty is in some precisions given on the estimates in statement 2.
and 3., which integrate the behaviour with respect to\ the complex
parameters $\lambda $ and $\nu .$

\begin{enumerate}
\item See \cite{Dore Yakubov}, (2.1) in Lemma 2.4, p. 99.

\item The idea is to use the calculus given in \cite{Dore Yakubov}, in Lemma
2.4, at the end of p. 99:%
\begin{equation*}
\left\Vert \left( D_{\lambda }^{1/2}+\nu I\right) ^{-1}\right\Vert _{\mathcal{L}(E)}=\left\Vert \frac{1}{\pi }\int\nolimits_{0}^{\infty }\dfrac{
r^{1/2}e^{i\overline{\theta }/2}}{re^{i\overline{\theta }}+\nu ^{2}}\left(
D_{\lambda }+re^{i\overline{\theta }}I\right) ^{-1}e^{i\overline{\theta }%
}dr\right\Vert _{\mathcal{L}(E)},
\end{equation*}%
but using $\left\Vert \left( D_{\lambda }+re^{i\overline{\theta }}I\right)
^{-1}\right\Vert _{\mathcal{L}(E)}\leqslant \dfrac{C_{\theta }}{\left\vert
\lambda \right\vert +r+1}$, we get that%
\begin{eqnarray*}
\left\Vert \left( D_{\lambda }^{1/2}+\nu I\right) ^{-1}\right\Vert _{\mathcal{L}(E)} &\leqslant &\frac{C_{\theta }}{\pi }\int_{0}^{+\infty }\frac{r^{1/2}}{\left\vert re^{i\overline{\theta }}+\nu
^{2}\right\vert }\dfrac{1}{\left\vert \lambda \right\vert +r+1}dr \\
&\leqslant &\frac{C_{\theta }}{\pi }\int_{0}^{+\infty }\frac{r^{1/2}}{\left\vert \cos \left( \frac{\arg \left( \nu ^{2}\right) -\overline{\theta }}{2}\right) \right\vert \left( r+\left\vert \nu \right\vert^{2}\right) }\frac{1}{\left\vert \lambda \right\vert +r+1}dr \\
&\leqslant &\dfrac{C_{\theta }}{\pi \cos \left( \arg \left( \nu \right) - \frac{\overline{\theta }}{2}\right) }\int_{0}^{+\infty }\frac{r^{1/2}}{\left( r+\left\vert \nu \right\vert ^{2}\right) \left( r+\left\vert
\lambda \right\vert +1\right) }dr \\
&=&\dfrac{C_{L}}{\cos \left( \arg \left( \nu \right) -\frac{\overline{\theta 
}}{2}\right) \cos \left( \left( \varphi +\theta \right) /2\right) }\frac{1}{%
\left\vert \nu \right\vert +\sqrt{\left\vert \lambda \right\vert +1}}.
\end{eqnarray*}
The last equality, follows from
$$
\int \nolimits_{0}^{+\infty }\dfrac{r^{1/2}}{\left( r+a\right) \left(
r+b\right)}~dr = \frac{\pi }{\sqrt{b}+\sqrt{a}}, \quad a,b > 0.
$$

\item Estimate (\ref{est gen}) is deduced from Statement 3. as in \cite{Dore
Yakubov}, Lemma 2.4, p. 100.
\end{enumerate}
\end{proof}
From the previous Lemma, we deduce that
\begin{equation}\label{estim C0}
\forall\nu \geqslant 0, \quad \left\Vert \left( D_{\lambda }+\nu I\right) ^{-1}\right\Vert _{\mathcal{L}(E)}\leqslant \dfrac{C_{0}}{\left\vert \lambda \right\vert +\nu +1},
\end{equation}
and
\begin{equation}\label{est part}
\forall \nu \in S_{\pi /2}, \quad \left\Vert \left( D_{\lambda }^{1/2}+\nu I\right) ^{-1}\right\Vert _{\mathcal{L}(E)}\leqslant \dfrac{K_{L}}{\left\vert \nu \right\vert +\sqrt{\left\vert \lambda \right\vert +1}}, 
\end{equation}
where $C_{0}=C_{L}/\cos \left( \varphi /2\right) $ and $K_{L}:=C_{L}/\cos ^{2}\left( \varphi /2\right)$.

Here, since $D_{\lambda }~$is boundedly invertible, we have also that $%
D_{\lambda }^{1/2}$ is boundedly invertible and then $\rho \left( D_{\lambda
}^{1/2}\right) $ contains a ball centered in $0$. The following Lemma
specifies the size of this ball with respect to $\lambda \in S_{\varphi }$.

\begin{lemma}\label{Lemme boule}
We have

\begin{enumerate}
\item $\left\Vert D_{\lambda }^{-1/2}\right\Vert _{\mathcal{L}(E)}\leqslant 
\dfrac{K_{L}}{\sqrt{\left\vert \lambda \right\vert +1}}.$

\item For $z\in \overline{B\left( 0,\tfrac{\sqrt{\left\vert \lambda
\right\vert +1}}{2K_{L}}\right) }:z\in \rho (D_{\lambda }^{1/2})$ and $\Vert
(D_{\lambda }^{1/2}-zI)^{-1}\Vert _{\mathcal{L}(E)}\leqslant \dfrac{2K_{L}}{%
\sqrt{\left\vert \lambda \right\vert +1}}.$
\end{enumerate}
\end{lemma}

\begin{proof}
For statement 1., it is enough to consider (\ref{est part}) with $\nu =0$.

For statement 2, we consider $z\in \overline{B\left( 0,\tfrac{\sqrt{%
\left\vert \lambda \right\vert +1}}{2K_L}\right) }$, then%
\begin{equation*}
0\leqslant \Vert zD_{\lambda }^{-1/2}\Vert _{\mathcal{L}(E)}=\left\vert
z\right\vert \Vert D_{\lambda }^{-1/2}\Vert _{\mathcal{L}(E)}\leqslant 
\dfrac{\sqrt{\left\vert \lambda \right\vert +1}}{2K_{L}}\dfrac{K_{L}}{\sqrt{%
\left\vert \lambda \right\vert +1}}=1/2<1,
\end{equation*}
so $D_{\lambda }^{1/2}-zI=D_{\lambda }^{1/2}\left( I-zD_{\lambda
}^{-1/2}\right)$ is boundedly invertible with 
$$\Vert (D_{\lambda }^{1/2}-zI)^{-1}\Vert _{\mathcal{L}(E)} \leqslant \left\Vert D_{\lambda }^{-1/2}\right\Vert _{\mathcal{L}(E)}\times \left\Vert \left( I-zD_{\lambda }^{-1/2}\right) ^{-1}\right\Vert _{\mathcal{L}(E)} \leqslant \frac{%
2K_{L}}{\sqrt{\left\vert \lambda \right\vert +1}}.$$
\end{proof}

Now we will compare $D_{\lambda }^{1/2}$ and $D_{0}^{1/2}$. This has been
already done for $\lambda >0$ in \cite{Haase}, Proposition 3.1.7 p. 65. Here 
$\lambda $ is a complex parameter: we furnish a precise estimate for the
bounded operator $T_{\lambda }$ which extends $D_{\lambda
}^{1/2}-D_{0}^{1/2} $.

\begin{lemma}
\label{Compar racine}~

\begin{enumerate}
\item There exists a unique $T_{\lambda }\in \mathcal{L}(E)$ such that 
\begin{equation}
D_{\lambda }^{1/2}=D_{0}^{1/2}+T_{\lambda },  \label{Square diff}
\end{equation}

\item $T_{\lambda }D_{\lambda ^{\prime }}^{-1/2}=D_{\lambda ^{\prime
}}^{-1/2}T_{\lambda }$ for any $\lambda ^{\prime }\in S_{\varphi }.$

\item $\left\Vert T_{\lambda }\right\Vert _{\mathcal{L}(E)}\leqslant
C_{0}C_{L}\sqrt{\left\vert \lambda \right\vert }$.
\end{enumerate}
\end{lemma}

\begin{proof}
First, notice that $L=D_{0}$ and $D\left( L\right) \subset D\left(
D_{\lambda }^{1/2}\right) \cap D\left( D_{0}^{1/2}\right) $. Thus, if $%
T_{\lambda }\in \mathcal{L}(E)$ satisfies (\ref{Square diff}) then $%
T_{\lambda }$ is unique since $\overline{D\left( L\right) }=E$.

We have (see for example \cite{Dore Yakubov}, p. 100) 
\begin{eqnarray*}
D_{\lambda }^{-1/2}-D_{0}^{-1/2} &=&\frac{1}{\pi }\int_{0}^{+\infty }\frac{1}{
\sqrt{t}}\left( D_{\lambda }+tI\right) ^{-1}dt-\frac{1}{\pi }\int_{0}^{+\infty }\frac{1}{\sqrt{t}}\left( D_{0}+tI\right) ^{-1}dt \\
&=&\frac{1}{\pi }\int_{0}^{+\infty }\frac{1}{\sqrt{t}}\left[ \left(
D_{\lambda }+tI\right) ^{-1}-\left( D_{0}+tI\right) ^{-1}\right] dt \\
&=&\frac{-\lambda }{\pi }L^{-1}\int_{0}^{+\infty }\frac{1}{\sqrt{t}}\left(
L+tI-tI\right) \left( D_{\lambda }+tI\right) ^{-1}\left( D_{0}+tI\right)
^{-1}dt \\
&=&\frac{-\lambda }{\pi }L^{-1}\int_{0}^{+\infty }\frac{1}{\sqrt{t}}\left(
D_{\lambda }+tI\right) ^{-1}dt \\
&&+\frac{\lambda }{\pi }L^{-1}\int_{0}^{+\infty }%
\sqrt{t}\left( D_{\lambda }+tI\right) ^{-1}\left( D_{0}+tI\right) ^{-1}dt \\
&=&-\lambda L^{-1}D_{\lambda }^{-1/2}+L^{-1}T_{\lambda }
\end{eqnarray*}%
where $\dis T_{\lambda }:=\frac{\lambda }{\pi }\int_{0}^{+\infty }\sqrt{t}\left(
D_{\lambda }+tI\right) ^{-1}\left( D_{0}+tI\right) ^{-1}dt\in \mathcal{L}(E).
$

This proves that $D\left( D_{\lambda }^{1/2}\right) =D\left(
D_{0}^{1/2}\right) =D\left( L^{1/2}\right) $. We then deduce (\ref{Square
diff}) by writing, for $\zeta \in D\left( L^{1/2}\right) $%
\begin{equation*}
D_{\lambda }^{1/2}\zeta -D_{0}^{1/2}\zeta =\left( L+\lambda I\right)
D_{\lambda }^{-1/2}\zeta -LD_{0}^{-1/2}\zeta =T_{\lambda }\zeta .
\end{equation*}

Statement 2., is an easy consequence of the definition of $T_{\lambda }$.
Statement 3. follows from (\ref{estim C0}), since 
\begin{equation*}
\left\Vert T_{\lambda }\right\Vert _{\mathcal{L}(E)}\leqslant \frac{\left\vert \lambda \right\vert }{\pi }\int_{0}^{+\infty }\sqrt{t}\frac{C_{0}}{\left\vert \lambda \right\vert +t+1}\frac{C_{L}}{t+1}dt = \frac{\left\vert \lambda \right\vert C_{0}C_{L}}{\sqrt{\left\vert \lambda \right\vert +1}+1}.
\end{equation*}
\end{proof}
From Lemma 2.6, p. 103, in \cite{Dore Yakubov}, for $\lambda \in S_{\varphi }$, we have that $G_{\lambda }:=-D_{\lambda }^{1/2}$ generates a semigroup $\left( e^{tG_{\lambda }}\right) _{t\geqslant 0}$ bounded, analytic for $t>0$ and strongly continuous for $t\geqslant 0$. Moreover, it satisfies
\begin{equation*}
\left \{ 
\begin{array}{l}
\exists K_{0}>0,~\exists c_{0}>0,~\forall t \geqslant 1/2,~\forall \lambda \in
S_{\varphi }: \vspace{0.1cm}\\ 
\max \left \{ \parallel e^{tG_{\lambda }}\parallel _{\mathcal{L}\left(
E\right) },\parallel G_{\lambda }e^{tG_{\lambda }}\parallel _{\mathcal{L}%
\left( E\right) }\right \} \leqslant K_{0}e^{-tc_{0}\left \vert \lambda
\right \vert ^{1/2}}.%
\end{array}%
\right.
\end{equation*}

\begin{lemma}\label{estim U V}
Let $-\infty <a<b<+\infty $. For $x\in \lbrack a,b]$, $\lambda \in S_{\varphi }$ and $f\in
L^{p}(a,b;E)$, we set
\begin{equation}\label{def U V lambda f}
U_{\lambda ,f}\left( x\right) =\int_{a}^{x}e^{(x-s)G_{\lambda }}f(s)ds \quad\text{and} \quad V_{\lambda ,f}\left( x\right) =\int_{x}^{b}e^{(s-x)G_{\lambda }}f(s)ds.
\end{equation}%
There exists $M_{L}>0$ such that for any $f\in L^{p}(a,b;E)$ and any $%
\lambda \in S_{\varphi }$ 
\begin{equation*}
\left \Vert U_{\lambda ,f}\right \Vert _{L^{p}(a,b;E)}\leqslant \dfrac{M_{L}%
}{\sqrt{|\lambda |+1}}\left \Vert f\right \Vert _{L^{p}(a,b;E)} \quad\text{and} \quad \left \Vert V_{\lambda ,f}\right \Vert _{L^{p}(a,b;E)}\leqslant \dfrac{M_{L}
}{\sqrt{|\lambda |+1}}\left \Vert f\right \Vert _{L^{p}(a,b;E)}.
\end{equation*}
\end{lemma}

\begin{proof}
We fix $\psi \in \left( \dfrac{\pi }{2},\dfrac{\pi }{2}+\dfrac{\varepsilon
\left( \varphi \right) }{2}\right) $ and use notations and estimates of
Lemma \ref{estim racine L}. We first focus on $U_{\lambda ,f}$. Let $x\in \lbrack a,b]$. We apply
the Dunford-Riesz Calculus to define $e^{\cdot G_{\lambda }}$, and obtain 
\begin{eqnarray*}
U_{\lambda ,f}(x) &=&\frac{1}{2i\pi }\int_{a}^{x}\int_{\gamma
}e^{(x-s)z}\left( zI-G_{\lambda }\right) ^{-1}f(s)~dz~ds \\ \ecart
&=&\frac{1}{2i\pi }\int_{a}^{x}\int_{\gamma }e^{(x-s)z}(zI+D_{\lambda
}^{1/2})^{-1}f(s)~dz~ds,
\end{eqnarray*}%
where the path $\gamma $ is the boundary positively oriented of $S_{\psi
}\cup B(0,\varepsilon )$ with $\varepsilon =\dfrac{\sqrt{\left \vert \lambda
\right \vert +1}}{2K_{L}}$.

Then%
\begin{eqnarray*}
U_{\lambda ,f}(x) &=&\frac{1}{2i\pi }\int_{a}^{x}\int_{\varepsilon
}^{+\infty }e^{(x-s)re^{i\psi }}(re^{i\psi }I+D_{\lambda
}^{1/2})^{-1}f(s)e^{i\psi }drds \\ \ecart
&&+\frac{1}{2i\pi }\int_{a}^{x}\int_{2\pi -\psi }^{\psi }e^{(x-s)\varepsilon
e^{i\theta }}(\varepsilon e^{i\theta }I+D_{\lambda
}^{1/2})^{-1}f(s)\varepsilon ie^{i\theta }d\theta ds \\ \ecart
&&-\frac{1}{2i\pi }\int_{a}^{x}\int_{\varepsilon }^{+\infty}e^{(x-s)re^{-i\psi }}(re^{-i\psi }I+D_{\lambda }^{1/2})^{-1}f(s)e^{-i\psi}drds,
\end{eqnarray*}%
hence
\begin{eqnarray*}
\Vert U_{\lambda ,f}(x)\Vert &\leqslant &\dfrac{1}{2\pi }\int_{a}^{x}\int_{\varepsilon }^{+\infty }\left
\Vert e^{(x-s)re^{i\psi }}f(s)\right \Vert \left \Vert (re^{i\psi
}I+D_{\lambda }^{1/2})^{-1}\right \Vert _{\mathcal{L}(E)}drds \\ \ecart
&&+\dfrac{\varepsilon }{2\pi }\int_{a}^{x}\int_{\psi }^{2\pi -\psi }\left
\Vert e^{(x-s)\varepsilon e^{i\theta }}f(s)\right \Vert \left \Vert
(\varepsilon e^{i\theta }I+D_{\lambda }^{1/2})^{-1}\right \Vert _{\mathcal{L}%
(E)}d\theta ds \\ \ecart
&&+\dfrac{1}{2\pi }\int_{a}^{x}\int_{\varepsilon }^{+\infty }\left \Vert
e^{(x-s)re^{-i\psi }}f(s)\right \Vert \left \Vert (re^{-i\psi }I+D_{\lambda
}^{1/2})^{-1}\right \Vert _{\mathcal{L}(E)}drds.
\end{eqnarray*}
We deduce, from Lemma \ref{estim racine L}, statement 4. and Lemma \ref%
{Lemme boule}, statement 3., that 
\begin{eqnarray*}
\Vert U_{\lambda ,f}(x)\Vert &\leqslant &\frac{1}{2\pi }\int_{a}^{x}\int_{%
\varepsilon }^{+\infty }\left \Vert e^{(x-s)re^{i\psi }}f(s)\right \Vert 
\frac{K_{\psi }}{r+\sqrt{|\lambda |+1}}drds \\
&&+\frac{\varepsilon }{2\pi }\int_{a}^{x}\int_{\psi }^{2\pi -\psi }\left
\Vert e^{(x-s)\varepsilon e^{i\theta }}f(s)\right \Vert \frac{2K_{L}}{\sqrt{%
|\lambda |+1}}d\theta ds \\
&&+\frac{1}{2\pi }\int_{a}^{x}\int_{\varepsilon }^{+\infty }\left \Vert
e^{(x-s)re^{-i\psi }}f(s)\right \Vert \frac{K_{\psi }}{r+\sqrt{|\lambda |+1}}%
drds,
\end{eqnarray*}
hence
\begin{eqnarray*}
\Vert U_{\lambda ,f}(x)\Vert &\leqslant &\frac{K_{\psi }}{\pi }\int_{a}^{x}\int_{\varepsilon }^{+\infty
}e^{(x-s)r\cos \left( \psi \right) }\left \Vert f(s)\right \Vert \frac{1}{r+%
\sqrt{|\lambda |+1}}drds \\
&&+\frac{\varepsilon K_{L}}{\pi }\frac{1}{\sqrt{|\lambda |+1}}%
\int_{a}^{x}\int_{\psi }^{2\pi -\psi }e^{(x-s)\varepsilon \cos \left( \theta
\right) }\left \Vert f(s)\right \Vert d\theta ds \\
&\leqslant &\frac{K_{\psi }}{\pi }%
\int_{a}^{x}\left( \int_{\varepsilon }^{+\infty }\frac{e^{(x-s)r\cos \left(
\psi \right) }}{r+\sqrt{|\lambda |+1}}dr\right) \left \Vert f(s)\right \Vert
ds \\
&&+\frac{\varepsilon K_{L}}{\pi }\frac{1}{\sqrt{|\lambda |+1}}\int_{\psi
}^{2\pi -\psi }\int_{a}^{x}e^{(x-s)\varepsilon \cos \left( \psi \right)
}\left \Vert f(s)\right \Vert dsd\theta \\
&\leqslant &\frac{K_{\psi }}{\pi }\int_{a}^{x}\left( \int_{\varepsilon
}^{+\infty }\frac{e^{(x-s)r\cos \left( \psi \right) }}{r+\sqrt{|\lambda |+1}}%
dr\right) \left \Vert f(s)\right \Vert ds \\
&&+\frac{2\varepsilon K_{L}}{\sqrt{|\lambda |+1}}\int_{a}^{x}e^{(x-s)\varepsilon \cos \left( \psi \right) }\left \Vert f(s)\right \Vert ds.
\end{eqnarray*}
So, setting%
\begin{equation*}
\left \{ 
\begin{array}{l}
\dis U_{\lambda ,f}^{1}(x)=\dfrac{K_{\psi }}{\pi }\int_{a}^{x}\left(\int_{\varepsilon }^{+\infty }\dfrac{e^{(x-s)r\cos \left( \psi \right) }}{r+\sqrt{|\lambda |+1}}dr\right) \left \Vert f(s)\right \Vert ds \\ \ecart
\dis U_{\lambda ,f}^{2}(x)=\dfrac{2\varepsilon K_{L}}{\sqrt{|\lambda |+1}}\int_{a}^{x}e^{(x-s)\varepsilon \cos \left( \psi \right) }\left \Vert
f(s)\right \Vert ds,%
\end{array}%
\right.
\end{equation*}%
we have 
\begin{equation}
\left \Vert U_{\lambda ,f}\right \Vert _{L^{p}(a,b;E)}\leqslant \left \Vert
U_{\lambda ,f}^{1}\right \Vert _{L^{p}(a,b)}+\left \Vert U_{\lambda
,f}^{2}\right \Vert _{L^{p}(a,b)}.  \label{decomp}
\end{equation}

\begin{description}
\item[Estimate of $\left \Vert U_{\protect\lambda ,f}^{1}\right \Vert
_{L^{p}(a,b)}.$] Define $g\in L^{1}\left( \mathbb{R}\right) ,F\in L^{p}\left( \mathbb{R}\right) $ by%
\begin{equation*}
g(t):=\left \{ 
\begin{array}{ll}
\dis \int_{\varepsilon }^{+\infty }\dfrac{e^{tr\cos \left( \psi \right) }}{r+%
\sqrt{|\lambda |+1}}dr & \text{if }t>0 \vspace{0.1cm}\\ 
0 & \text{elsewhere,}%
\end{array}%
\right. ~~ \text{and} ~~ F(t):=\left \{ 
\begin{array}{ll}
\Vert f(t)\Vert & \text{if }t\in (a,b)\vspace{0.1cm} \\ 
0 & \text{elsewhere,}%
\end{array}%
\right.
\end{equation*}%
and thus
\begin{eqnarray*}
\dis U_{\lambda ,f}^{1}(x) &=&\dis \dfrac{K_{\psi }}{\pi }\int_{a}^{x}g(x-s)\left \Vert f(s)\right \Vert ds+\dfrac{K_{\psi }}{\pi }\int_{x}^{b}g(x-s)\left \Vert f(s)\right \Vert ds \\
&=&\dfrac{K_{\psi }}{\pi }\int_{-\infty }^{+\infty }g(x-s)F(s)~ds = \dfrac{K_{\psi }}{\pi }\left( g\ast F\right) \left( x\right) .
\end{eqnarray*}%
Then, from Young's inequality, we obtain%
$$
\left \Vert U_{\lambda ,f}^{1}\right \Vert _{L^{p}(a,b)} \leqslant \frac{K_{\psi }}{\pi }\left \Vert g\ast F\right \Vert _{L^{p}\left( \mathbb{R}\right) } \leqslant \frac{K_{\psi }}{\pi }\left \Vert g\right \Vert _{L^{1}\left( \mathbb{R}\right) }\left \Vert F\right \Vert _{L^{p}\left( \mathbb{R}\right) }.$$
Setting $\ell =\sqrt{|\lambda |+1}$ and noting that $\varepsilon /\ell =1/2K$, we have
\begin{eqnarray*}
\left \Vert g\right \Vert _{L^{1}\left( \mathbb{R}\right) } &=&\int_{0}^{+\infty }\int_{\varepsilon }^{+\infty }\dfrac{e^{tr\cos \left( \psi \right) }}{r+\ell }drdt\\ \ecart
&=&\int_{0}^{+\infty }\left( \int_{\varepsilon /\ell }^{+\infty }\dfrac{e^{t\rho \ell \cos \left( \psi \right) }}{\rho +1}d\rho \right) dt \\ \ecart
&=&\frac{1}{\ell \cos \left( \psi \right) }\int_{1/2K_{L}}^{+\infty }\frac{d\rho }{\rho \left( \rho +1\right) } \\ \ecart
&=&\frac{\ln \left( 2K_{L}+1\right) /\cos \left( \psi \right) }{\sqrt{|\lambda |+1}},
\end{eqnarray*}
and finally
\begin{equation*}
\left \Vert U_{\lambda ,f}^{1}\right \Vert _{L^{p}(a,b)}\leqslant \frac{%
K_{\psi }\ln \left( 2K_{L}+1\right) /\pi \cos \left( \psi \right) }{\sqrt{%
|\lambda |+1}}\left \Vert f\right \Vert _{L^{p}(a,b;E)}.
\end{equation*}

\item[Estimate of $\left \Vert U_{\protect\lambda ,f}^{2}\right \Vert
_{L^{p}(a,b)}.$] Define $F\in L^{p}\left( \mathbb{R}\right) $ as above and $h \in L^{1}\left( \mathbb{R}\right)$ as follows
\begin{equation*}
h(t):=\left \{ 
\begin{array}{ll}
e^{t\varepsilon \cos \left( \psi \right) } & \text{if }t>0 \\ 
0 & \text{elsewhere};
\end{array}\right.
\end{equation*}%
then, as previously
$$U_{\lambda ,f}^{2}(x) = \frac{2\varepsilon K_{L}}{\sqrt{|\lambda |+1}}\left( h\ast F\right)\left( x\right);$$
therefore from Young's inequality, we get
$$\left \Vert U_{\lambda ,f}^{2}\right \Vert _{L^{p}(a,b)} \leqslant \frac{2K_{L}/\left \vert \cos \left( \psi \right) \right \vert }{\sqrt{
|\lambda |+1}}\left \Vert f\right \Vert _{L^{p}(a,b;E)}.$$
\end{description}

From (\ref{decomp}) and the two previous estimates, we obtain the expected result on $U_\lambda$.

Setting $\widetilde{f}(\cdot ):=f(\cdot -a-b)$, we note that
\begin{equation}\label{change}
V_{\lambda ,f}\left( x\right) =U_{\lambda ,\widetilde{f}}\left( b+a-x\right); 
\end{equation}
then, there exists $M_L > 0$ such that
\begin{equation*}
\left \Vert V_{\lambda ,f}\right \Vert _{L^{p}(a,b;E)}\leqslant \dfrac{M_{L}%
}{\sqrt{|\lambda |+1}}\left \Vert \widetilde{f}\right \Vert _{L^{p}(a,b;E)}=%
\dfrac{M_{L}}{\sqrt{|\lambda |+1}}\left \Vert f\right \Vert _{L^{p}(a,b;E)}.
\end{equation*}
\end{proof}

\begin{definition}
\label{Def Lp Reg}We say that a closed linear operator $\mathcal{A}$ on $E$,
has the $L^{p}$ regularity property on $[a,b]$, if the Cauchy problem%
\begin{equation*}
\left \{ 
\begin{array}{l}
u^{\prime }(t)=\mathcal{A}u(t)+f(t),\quad t\in \left( a,b\right) \vspace{0.1cm}\\ 
u(a)=0,%
\end{array}%
\right.
\end{equation*}%
admits, for any $f\in L^{p}(a,b;E)$, a unique solution $u_{f}\in W^{1,p}\left( a,b;E\right) \cap L^{p}\left( a,b;D(\mathcal{A}%
)\right)$. 

In this case, there exists $K>0$ such that for any $f\in L^{p}\left(
a,b;E\right) $%
\begin{equation*}
\left \Vert u_{f}\right \Vert _{L^{p}(a,b;E)}+\left \Vert u_{f}^{\prime
}\right \Vert _{L^{p}(a,b;E)}+\left \Vert \mathcal{A}u_{f}\right \Vert
_{L^{p}(a,b;E)}\leqslant K\left \Vert f\right \Vert _{L^{p}(a,b;E)}.
\end{equation*}
\end{definition}

For details on the $L^{p}$ regularity property we refer to \cite{Dore1} and \cite{Dore2}.

\begin{lemma}
\label{Reg max}Assume that $G=-L^{1/2}$ has the $L^{p}$ regularity property
on $[a,b]$, and consider $U_{\lambda ,f},V_{\lambda ,f}$ defined in (\ref%
{def U V lambda f}). Let $\lambda \in S_{\varphi }$, then:

\begin{enumerate}
\item The linear operator $G_{\lambda }=-\left( -L+\lambda I\right) ^{1/2}$
has the $L^{p}$ regularity property on $[a,b].$

\item For any $f\in L^{p}(a,b;E)$, $U_{\lambda ,f},V_{\lambda ,f}\in
W^{1,p}\left( a,b;E\right) \cap L^{p}\left( a,b;D(G)\right) $, $U_{\lambda
,f}$ is the unique solution to 
\begin{equation}
\left \{ 
\begin{array}{l}
v^{\prime }(t)=G_{\lambda }v(t)+f(t),\text{ \ }t\in \left( a,b\right) \vspace{0.1cm}\\ 
v(a)=0,%
\end{array}%
\right.  \label{Cauchy Lambda}
\end{equation}%
and $V_{\lambda ,f}$~is the unique solution to
\begin{equation*}
\left \{ 
\begin{array}{l}
v^{\prime }(t)=-G_{\lambda }v(t)+f(t),\text{ \ }t\in \left( a,b\right) \vspace{0.1cm}\\ 
v(b)=0.%
\end{array}%
\right.
\end{equation*}

\item There exists $\widetilde{M}_{L}>0$ (which does not depend on $\lambda $) such that for any $f\in L^{p}(a,b;E)$ we have
\begin{equation*}
\left \{ \begin{array}{l}
\sqrt{|\lambda |+1}\left \Vert U_{\lambda ,f}\right \Vert
_{L^{p}(a,b;E)}+\left \Vert U_{\lambda ,f}^{\prime }\right \Vert
_{L^{p}(a,b;E)}+\left \Vert G_{\lambda }U_{\lambda ,f}\right \Vert
_{L^{p}(a,b;E)} \leqslant \widetilde{M}_{L}\left \Vert f\right \Vert _{L^{p}(a,b;E)} \\ 
\text{~} \vspace{0.1cm}\\ 
\sqrt{|\lambda |+1}\left \Vert V_{\lambda ,f}\right \Vert
_{L^{p}(a,b;E)}+\left \Vert V_{\lambda ,f}^{\prime }\right \Vert
_{L^{p}(a,b;E)}+\left \Vert G_{\lambda }V_{\lambda ,f}\right \Vert
_{L^{p}(a,b;E)} \leqslant \widetilde{M}_{L}\left \Vert f\right \Vert _{L^{p}(a,b;E)}.%
\end{array}
\right.
\end{equation*}
\end{enumerate}
\end{lemma}

\begin{proof}
Let $\lambda \in S_{\varphi }$. We consider $T_{\lambda }$, defined in Lemma %
\ref{Compar racine}, statement 1. and due to (\ref{Square diff}), we have $%
G_{\lambda }=G-T_{\lambda }$.

\begin{enumerate}
\item Let $f\in L^{p}\left( a,b;E\right) $. Here, we want to show that (\ref{Cauchy Lambda}) admits a unique solution\ in $W^{1,p}\left( a,b;E\right)
\cap L^{p}\left( a,b;D(G)\right)$.

\begin{itemize}
\item First, we set $g(.) = e^{\left( .-a\right) T_{\lambda }}f(.)\in L^{p}\left( a,b;E\right)$.

\item Then we consider $U_{0,g}$ defined by (\ref{def U V lambda f}) which
is the solution to
\begin{equation}\label{Cauchy G}
\left\{ \begin{array}{l}
u^{\prime }(t)=Gu(t)+g(t),\quad t\in (a,b) \\ 
u(a)=0;
\end{array}\right.   
\end{equation}
but $G$ has the $L^{p}$ regularity property on $[a,b]$, so 
$$U_{0,g} \in W^{1,p}\left( a,b;E\right) \cap L^{p}\left( a,b;D(G)\right).$$

\item Since $T_{\lambda }\in \mathcal{L}(E)$ and $U_{0,g}\in W^{1,p}\left(
a,b;E\right) \cap L^{p}\left( a,b;D(G)\right) $ we get that%
\begin{equation}
v:=e^{-\left( \cdot -a\right) T_{\lambda }}U_{0,g},  \label{def v lambda}
\end{equation}%
is also in $W^{1,p}\left( a,b;E\right) \cap L^{p}\left( a,b;D(G)\right) $
with%
\begin{equation*}
v^{\prime }=-T_{\lambda }e^{-\left( \cdot -a\right) T_{\lambda
}}U_{0,g}+e^{-\left( \cdot -a\right) T_{\lambda }}U_{0,g}^{\prime }.
\end{equation*}%
So using (\ref{Cauchy G}) and the fact that $T_{\lambda }G=GT_{\lambda }$ on 
$D\left( G\right) $ (see Lemma \ref{Compar racine}, statement 2.) we deduce
that%
\begin{eqnarray*}
v^{\prime } &=&-T_{\lambda }e^{-\left( \cdot -a\right) T_{\lambda
}}U_{0,g}+e^{-\left( \cdot -a\right) T_{\lambda }}\left( GU_{0,g}+g\right) 
\\
&=&\left( G-T_{\lambda }\right) e^{-\left( \cdot -a\right) T_{\lambda
}}U_{0,g}+e^{-\left( \cdot -a\right) T_{\lambda }}g.
\end{eqnarray*}%
Finally $v$ satisfies%
\begin{equation*}
\left\{ 
\begin{array}{l}
v^{\prime }(t)=\left( G-T_{\lambda }\right) v(t)+f(t),\text{ \ }t\in \left(
a,b\right)  \\ 
v(a)=0.%
\end{array}%
\right. 
\end{equation*}

\item From Lemma \ref{estim racine L}, statement 5, we have $G_{\lambda
}=G-T_{\lambda }$ so $v=e^{-\left( \cdot -a\right) T_{\lambda }}U_{0,g}$ is
a solution of (\ref{Cauchy Lambda}) with the expected regularity. Moreover
if $w$ is another solution of (\ref{Cauchy Lambda}) then $e^{-\left( \cdot
-a\right) T_{\lambda }}w$ satisfies (\ref{Cauchy G}), so $e^{-\left( \cdot
-a\right) T_{\lambda }}w=U_{0,g}$ and $w=v$; this proves the uniqueness of
the solution of (\ref{Cauchy Lambda}).
\end{itemize}

\item From \eqref{def U V lambda f} we have that $U_{\lambda ,f}$ is a
formal solution of \eqref{Cauchy Lambda}; then $\dis U_{\lambda ,f}=e^{-\left( \cdot -a\right) T_{\lambda }}U_{0,g}$ and has the expected regularity. We use (\ref{change}) to study $V_{\lambda
,f}.$

\item Since $G$ has the $L^{p}$ regularity property on $[a,b]$, there exists 
$K > 0$ such for any $h\in L^{p}(a,b;E)$ 
\begin{equation*}
\left \Vert U_{0,h}^{\prime }\right \Vert _{L^{p}(a,b;E)}+\left \Vert
GU_{0,h}\right \Vert _{L^{p}(a,b;E)}\leqslant K\left \Vert h\right \Vert
_{L^{p}(a,b;E)}.
\end{equation*}

Now let $\lambda \in S_{\varphi }$. $U_{\lambda ,f}$ satisfies (\ref{Cauchy
Lambda}) so%
\begin{equation*}
\left \{ 
\begin{array}{l}
U_{\lambda ,f}^{\prime }(t)=\left( G-T_{\lambda }\right) U_{\lambda
,f}(t)+f(t),\text{ \ }t\in \left( a,b\right) \vspace{0.1cm}\\ 
U_{\lambda ,f}(a)=0;
\end{array}\right.
\end{equation*}%
thus setting $h_{\lambda }=-T_{\lambda }U_{\lambda ,f}+f$%
\begin{equation*}
\left \{ 
\begin{array}{l}
U_{\lambda ,f}^{\prime }(t)=GU_{\lambda ,f}(t)+h_{\lambda }\left( t\right) ,%
\text{ \ }t\in \left( a,b\right) \vspace{0.1cm}\\ 
U_{\lambda ,f}(a)=0,%
\end{array}%
\right.
\end{equation*}%
then $U_{\lambda ,f}=U_{0,h_{\lambda }}$ and
\begin{eqnarray*}
\left \Vert U_{\lambda ,f}^{\prime }\right \Vert _{L^{p}(a,b;E)}+\left
\Vert G_{\lambda }U_{\lambda ,f}\right \Vert _{L^{p}(a,b;E)} &=&\left \Vert U_{0,h_{\lambda }}^{\prime }\right \Vert_{L^{p}(a,b;E)}+\left \Vert GU_{0,h_{\lambda }}\right \Vert _{L^{p}(a,b;E)} \\ \ecart
&\leqslant &\left \Vert T_{\lambda }\right \Vert _{\mathcal{L}(E)}\left
\Vert U_{\lambda ,f}\right \Vert _{L^{p}(a,b;E)}+\left \Vert f\right \Vert
_{L^{p}(a,b;E)} \\ \ecart
&\leqslant & C_{0}C_{L}\sqrt{\left \vert \lambda \right \vert }\dfrac{M_{L}}{\sqrt{|\lambda |+1}}\left \Vert f\right \Vert _{L^{p}(a,b;E)} +\left \Vert f\right \Vert _{L^{p}(a,b;E)} \\ \ecart
&\leqslant &\widetilde{M}_{L}\left \Vert f\right \Vert _{L^{p}(a,b;E)}.
\end{eqnarray*}%
For the estimate of $T_\lambda$, we have used Lemma \ref{Compar racine}, statement 3. 

Moreover, using again (\ref{change}) to study $V_{\lambda ,f}$, we obtain the expected result.
\end{enumerate}
\end{proof}

\section{Spectral problem (\protect\ref{NewEquation})-(\protect\ref{RobinSpectralConditions}): first case\label{Section 1er Cas}}

\subsection{Preliminary estimates\label{spect 1}}

In this subsection we suppose that $X,A,H$ satisfy (\ref{NewH1})$\sim $(\ref{MoinsABip}). Note that the results of Section~\ref{Dore Yaku}, can be
applied to our operator $-A$, since due to (\ref{NewH1}), (\ref{NewH2}), $-A$
is densely defined\ and from (\ref{NewH2}) we have that $-A$ is an operator
of type $\varphi _{0}$ with bound $C_{A}$. For $\lambda \in S_{\varphi _{0}}$%
, $-A+\lambda I$ is an operator of type $\theta $ (for any $\theta \in
\left( 0,\varepsilon \left( \varphi _{0}\right) \right) $; in particular if
we set $Q_{\lambda }=-(-A+\lambda I)^{1/2}$, then from Lemma~\ref{estim racine L}, statement 2., $Q_{\lambda }$ generates a semigroup $\left( e^{-tQ_{\lambda }}\right) _{t\geqslant 0} $which is bounded, analytic for $t>0$ and strongly continuous for $t\geqslant 0$. Moreover, there exists $K>0$, such that
\begin{equation}
\forall \lambda \in S_{\varphi _{0}},\quad \parallel Q_{\lambda
}^{-1}\parallel _{\mathcal{L}\left( X\right) }\leqslant \dfrac{K}{\left(
1+\left \vert \lambda \right \vert \right) ^{1/2}};
\label{EstimationQmoins1}
\end{equation}%
furthermore, from Lemma \ref{Lemme boule}, statement 3., we have $\overline{B\left( 0,1/2K\right) }\subset \rho \left( Q_{\lambda }\right) $, so there exists $\delta >0$, which does not depend on $\lambda $ such that $%
Q_{\lambda }+\delta I$ generates a bounded analytic semigroup; thus, for some 
$K_{1}\geqslant 1$%
\begin{equation}
\forall \lambda \in S_{\varphi _{0}},~\forall t\geqslant 0, \quad \parallel
e^{tQ_{\lambda }}\parallel _{\mathcal{L}\left( X\right) }\leqslant
K_{1}e^{-\delta t}.  \label{estim exp xqlambda}
\end{equation}%
There exist also $K_{0}, c_{0}> 0$ such that%
\begin{equation}
\left \{ 
\begin{array}{l}
\forall \lambda \in S_{\varphi _{0}},~\forall t\geqslant 1/2,~\forall j\in
\left \{ 0,1,2\right \} : \vspace{0.1cm}\\ 
\parallel Q_{\lambda }^{j}e^{tQ_{\lambda }}\parallel _{\mathcal{L}\left(
X\right) }\leqslant K_{0}e^{-2c_{0}\left \vert \lambda \right \vert ^{1/2}}.%
\end{array}%
\right.  \label{estim exp Q}
\end{equation}

\begin{lemma}
\label{Inv (I-exp2Q)}There exists a constant $M\geqslant 0$ independent of $%
\lambda \in S_{\varphi _{0}},$ such that for any $\lambda \in S_{\varphi
_{0}}$, operators $I\pm e^{2Q_{\lambda }}$ are invertible in $\mathcal{L}%
\left( X\right) $ and 
\begin{equation}
\forall \lambda \in S_{\varphi _{0}},\quad \left \Vert \left( I\pm
e^{2Q_{\lambda }}\right) ^{-1}\right \Vert _{\mathcal{L}\left( X\right)
}\leqslant M,  \label{estim (I - expQ)}
\end{equation}
\end{lemma}

\begin{proof}
Let $\lambda \in S_{\varphi _{0}}$. For $t\geqslant 0$, we have $\left \Vert e^{tQ_{\lambda }}\right \Vert _{\mathcal{L}\left( X\right)
}\leqslant K_{1}e^{-t\delta };$ we choose $k\in \mathbb{N}\backslash \{0\}$ such that $K_{1}e^{-2k\delta }\leqslant 1/2<1$. Then $I-e^{2kQ_{\lambda }}$ is invertible with%
\begin{equation*}
\left \Vert \left( I-e^{2kQ_{\lambda }}\right) ^{-1}\right \Vert _{\mathcal{L%
}\left( X\right) }\leqslant \frac{1}{1-1/2}=2,
\end{equation*}%
thus $0\in \rho (I-e^{2Q_{\lambda }})$ since%
\begin{eqnarray*}
I &=&(I-e^{2Q_{\lambda }})\left( I+e^{2Q_{\lambda }}+\cdots +e^{2\left(
k-1\right) Q_{\lambda }}\right) (I-e^{2kQ_{\lambda }})^{-1} \\
&=&(I-e^{2kQ_{\lambda }})^{-1}\left( I+e^{2Q_{\lambda }}+\cdots +e^{2\left(
k-1\right) Q_{\lambda }}\right) (I-e^{2Q_{\lambda }}).
\end{eqnarray*}%
Moreover%
\begin{eqnarray*}
\left \Vert \left( I-e^{2Q_{\lambda }}\right) ^{-1}\right \Vert _{%
\mathcal{L}\left( X\right) } &\leqslant &\left \Vert \left( I+e^{2Q_{\lambda }}+\cdots +e^{2\left(k-1\right) Q_{\lambda }}\right) (I-e^{2kQ_{\lambda }})^{-1}\right \Vert _{\mathcal{L}\left( X\right) } \\
&\leqslant &\left( 1+\left \Vert e^{2Q_{\lambda }}\right \Vert _{\mathcal{L}%
\left( X\right) }+\cdots +\left \Vert e^{2Q_{\lambda }}\right \Vert _{%
\mathcal{L}\left( X\right) }^{k-1}\right) \left \Vert (I-e^{2kQ_{\lambda
}})^{-1}\right \Vert _{\mathcal{L}\left( X\right) } \\
&\leqslant &2K_{1}^{k}.
\end{eqnarray*}

We obtain the same result for $I+e^{2Q_{\lambda }}$.
\end{proof}

\subsection{Spectral estimates}

In this subsection we assume that $X,A,H$ satisfy (\ref{NewH1})$\sim $(\ref{NewH3}).

Let $\lambda \in S_{\varphi _{0}},\mu \in S_{\varphi _{1}}$. We recall that $%
H_{\mu }=H+\mu I$ and furnish estimates concerning operators $Q_{\lambda
},H_{\mu }$ which are easy consequences of our assumptions.

In the following, $M$ denotes various constants, independent of $\lambda ,\mu 
$, which can vary from line to line.

\begin{lemma}
Let $\lambda \in S_{\varphi _{0}},\mu \in S_{\varphi _{1}}$. Then $\left(
-A+\lambda I\right) H_{\mu }^{-1}$ $\in \mathcal{L}\left( X\right)$;
moreover there exists a constant $M>0$ independent of $\lambda \in
S_{\varphi _{0}}$ and $\mu \in S_{\varphi _{1}}$ such that%
\begin{equation}
\max \left \{ \left \Vert HH_{\mu }^{-1}\right \Vert _{\mathcal{L}\left(
X\right) },\left \Vert AH_{\mu }^{-1}\right \Vert _{\mathcal{L}\left(
X\right) }\right \} \leqslant M,  \label{Consequence1}
\end{equation}%
\begin{equation}
\left \Vert Q_{\lambda }^{2}H_{\mu }^{-1}\right \Vert _{\mathcal{L}\left(
X\right) }\leqslant M\frac{1+\left \vert \lambda \right \vert +\left \vert
\mu \right \vert }{1+\left \vert \mu \right \vert },  \label{Consequence1bis}
\end{equation}%
and%
\begin{equation}
\left \Vert Q_{\lambda }H_{\mu }^{-1}\right \Vert _{\mathcal{L}(X)}\leqslant
M\frac{1+\left \vert \lambda \right \vert +\left \vert \mu \right \vert }{%
\left( 1+\left \vert \mu \right \vert \right) \left( 1+\left \vert \lambda
\right \vert \right) ^{1/2}}.  \label{Consequence2}
\end{equation}
\end{lemma}

\begin{proof}
Note that $\left( -A+\lambda I\right) $ is closed, so due to (\ref%
{D(H)Principal}), $\left( -A+\lambda I\right) H_{\mu }^{-1}$ is bounded.
Then 
$$\left \Vert HH_{\mu }^{-1}\right \Vert _{\mathcal{L}\left( X\right) }
=\left \Vert \left( H+\mu I\right) H_{\mu }^{-1}-\mu H_{\mu }^{-1}\right
\Vert \leqslant \left \Vert I\right \Vert _{\mathcal{L}\left( X\right) }+\left
\Vert \mu \left( H+\mu I\right) ^{-1}\right \Vert _{\mathcal{L}\left(
X\right) }\leqslant M;$$
moreover%
$$\left \Vert AH_{\mu }^{-1}\right \Vert _{\mathcal{L}\left( X\right) } \leqslant \left \Vert AH^{-1}HH_{\mu }^{-1}\right \Vert _{\mathcal{L}%
\left( X\right) } \leqslant \left \Vert AH^{-1}\right \Vert _{\mathcal{L}\left( X\right)}\left \Vert HH_{\mu }^{-1}\right \Vert _{\mathcal{L}\left( X\right) } \leqslant M,$$
and
\begin{eqnarray*}
\left \Vert Q_{\lambda }^{2}H_{\mu }^{-1}\right \Vert _{\mathcal{L}\left(
X\right) } &=&\left \Vert \left( -A+\lambda I\right) H_{\mu }^{-1}\right
\Vert _{\mathcal{L}\left( X\right) } \\
&\leqslant &\left \Vert AH_{\mu }^{-1}\right \Vert _{\mathcal{L}\left(
X\right) }+\left \Vert \lambda \left( H+\mu I\right) ^{-1}\right \Vert _{%
\mathcal{L}\left( X\right) }\leqslant M\left(1 +\frac{\left \vert
\lambda \right \vert }{1+\left \vert \mu \right \vert }\right).
\end{eqnarray*}

Finally, since 
\begin{equation*}
\left \Vert Q_{\lambda }H_{\mu }^{-1}\right \Vert _{\mathcal{L}(X)}=\left
\Vert Q_{\lambda }^{-1}Q_{\lambda }^{2}H_{\mu }^{-1}\right \Vert _{\mathcal{L%
}(X)}\leqslant \left \Vert Q_{\lambda }^{-1}\right \Vert _{\mathcal{L}%
(X)}\left \Vert Q_{\lambda }^{2}H_{\mu }^{-1}\right \Vert _{\mathcal{L}(X)},
\end{equation*}%
we deduce (\ref{Consequence2}) from (\ref{Consequence1}) and (\ref%
{EstimationQmoins1}).
\end{proof}

For $\lambda \in S_{\varphi _{0}},\mu \in S_{\varphi _{1}}$, let us recall
that%
\begin{equation*}
\Lambda _{\lambda ,\mu }:=\left( Q_{\lambda }-H_{\mu }\right)
+e^{2Q_{\lambda }}\left( Q_{\lambda }+H_{\mu }\right) .
\end{equation*}%
Note that, since $D\left( H_{\mu }\right) \subset D(Q_{\lambda }^{2})$, we
have $D\left( \Lambda _{\lambda ,\mu }\right) =D\left( H_{\mu }\right) =D(H)$%
. We now introduce, for $r>0$, the notation

\begin{equation*}
\Omega _{\varphi _{0},\varphi _{1},r}=\left \{ \left( \lambda ,\mu \right)
\in S_{\varphi _{0}}\times S_{\varphi _{1}}:\left \vert \lambda \right \vert
\geqslant r\text{ and }\frac{\left \vert \mu \right \vert ^{2}}{\left \vert
\lambda \right \vert }\geqslant r\right \},
\end{equation*}%
and furnish results on $\Lambda _{\lambda ,\mu }$.

\begin{lemma}
\label{Grand Lamba inversible}There exist $r_{0}>0$ and $M>0$ such that for
all $\left( \lambda ,\mu \right) \in \Omega _{\varphi _{0},\varphi
_{1},r_{0}}$ we have%
\begin{equation}
\left \{ 
\begin{array}{l}
0 \in \rho\left(\left( I-e^{2Q_{\lambda }}\right) ^{-1}\left( I+e^{2Q_{\lambda }}\right)
Q_{\lambda }H_{\mu }^{-1}-I\right)\vspace{0.1cm} \\ 
\left \Vert \left[ \left( I-e^{2Q_{\lambda }}\right) ^{-1}\left(
I+e^{2Q_{\lambda }}\right) Q_{\lambda }H_{\mu }^{-1}-I\right] ^{-1}\right
\Vert _{\mathcal{L}(X)}\leqslant 2,%
\end{array}%
\right.  \label{first inv}
\end{equation}%
\begin{equation}
0 \in \rho\left(\Lambda _{\lambda ,\mu }\right) \quad \text{and} \quad \left \Vert
\Lambda _{\lambda ,\mu }^{-1}\right \Vert _{\mathcal{L}(X)}\leqslant \dfrac{M}{1+\left \vert \mu \right \vert },  \label{inversibilty det}
\end{equation}%
and%
\begin{equation}
\left \Vert Q_{\lambda }^{2}\Lambda _{\lambda ,\mu }^{-1}\right \Vert _{%
\mathcal{L}(X)}\leqslant M\frac{1+\left \vert \lambda \right \vert +\left
\vert \mu \right \vert }{1+\left \vert \mu \right \vert }.
\label{Q2InvLambda}
\end{equation}
\end{lemma}

Note that $Q_{\lambda }^{2}\Lambda _{\lambda ,\mu }^{-1}$ has the same
behaviour as $Q_{\lambda }H_{\mu }^{-1}$, see (\ref{Consequence2}) and (\ref%
{Est Q2 SLambda}).

\begin{proof}
Let $\left( \lambda ,\mu \right) \in \Omega _{\varphi _{0},\varphi _{1},r}$
for some $r>0$. From (\ref{D(H)Principal}), we have $Q_{\lambda }H_{\mu }^{-1}\in 
\mathcal{L}(X)$, hence $\left( I-e^{2Q_{\lambda }}\right) ^{-1}\left(
I+e^{2Q_{\lambda }}\right) Q_{\lambda }H_{\mu }^{-1}-I$ $\in \mathcal{L}(X)$%
; moreover, from (\ref{estim (I - expQ)}) and (\ref%
{Consequence2}), we obtain
$$\begin{array}{lll}
\dis \left \Vert \left( I-e^{2Q_{\lambda }}\right) ^{-1}\left( I+e^{2Q_{\lambda
}}\right) Q_{\lambda }H_{\mu }^{-1}\right \Vert _{\mathcal{L}(X)} \\ 
\leqslant \dis M\left \Vert Q_{\lambda }H_{\mu }^{-1}\right \Vert _{\mathcal{L}(X)} \\
\leqslant \dis M \left(\frac{1 +\left \vert \mu \right \vert }{\left( 1+\left \vert \mu \right \vert \right) \left( 1+\left \vert \lambda\right \vert \right) ^{1/2}} + \frac{\left \vert \lambda \right \vert}{\left( 1+\left \vert \mu \right \vert \right) \left( 1+\left \vert \lambda\right \vert \right) ^{1/2}}\right) \\ 
\leqslant \dis M\left( \frac{1}{\left \vert \lambda \right \vert ^{1/2}}+\frac{%
\left \vert \lambda \right \vert ^{1/2}}{\left \vert \mu \right \vert }%
\right) \\
\leqslant \dis \frac{2M}{r^{1/2}}.
\end{array}$$
So there exists $r_{0}>0$ such that for all $\left( \lambda ,\mu \right) \in
\Omega _{\varphi _{0},\varphi _{1},r_{0}}$ we have%
\begin{equation}
\left \Vert \left( I-e^{2Q_{\lambda }}\right) ^{-1}\left( I+e^{2Q_{\lambda
}}\right) Q_{\lambda }H_{\mu }^{-1}\right \Vert _{\mathcal{L}(X)}\leqslant
1/2.  \label{est}
\end{equation}%
Let $\left( \lambda ,\mu \right) \in \Omega _{\varphi _{0},\varphi
_{1},r_{0}}$. Then (\ref{est}) proves (\ref{first inv}). We deduce that 
\begin{equation*}
L_{\lambda ,\mu }:=\left( I-e^{2Q_{\lambda }}\right) \left[ \left(
I-e^{2Q_{\lambda }}\right) ^{-1}\left( I+e^{2Q_{\lambda }}\right) Q_{\lambda
}H_{\mu }^{-1}-I\right] \in \mathcal{L}(X),
\end{equation*}%
is boundedly invertible. Moreover 
\begin{equation*}
L_{\lambda ,\mu }^{-1}=\left[ \left( I-e^{2Q_{\lambda }}\right) ^{-1}\left(
I+e^{2Q_{\lambda }}\right) Q_{\lambda }H_{\mu }^{-1}-I\right] ^{-1}\left(
I-e^{2Q_{\lambda }}\right) ^{-1},
\end{equation*}%
satisfies%
\begin{equation*}
\left \Vert L_{\lambda ,\mu }^{-1}\right \Vert _{\mathcal{L}(X)}\leqslant 2M.
\end{equation*}%
Now, we write $\Lambda _{\lambda ,\mu }=\left( I+e^{2Q_{\lambda }}\right)
Q_{\lambda }-\left( I-e^{2Q_{\lambda }}\right) H_{\mu }=L_{\lambda ,\mu
}H_{\mu }$, so $\Lambda _{\lambda ,\mu }$ is boundedly invertible with 
\begin{equation*}
\Lambda _{\lambda ,\mu }^{-1}=H_{\mu }^{-1}L_{\lambda ,\mu }^{-1};
\end{equation*}
this furnishes (\ref{inversibilty det}). Finally, $\left \Vert L_{\lambda ,\mu
}^{-1}\right \Vert _{\mathcal{L}(X)}\leqslant 2M$ and (\ref{Consequence1bis}%
) gives%
\begin{equation*}
\left \Vert Q_{\lambda }^{2}\Lambda _{\lambda ,\mu }^{-1}\right \Vert _{%
\mathcal{L}(X)}=\left \Vert Q_{\lambda }^{2}H_{\mu }^{-1}\right \Vert _{%
\mathcal{L}(X)}\left \Vert L_{\lambda ,\mu }^{-1}\right \Vert _{\mathcal{L}%
(X)}\leqslant M\frac{1+\left \vert \lambda \right \vert +\left \vert \mu
\right \vert }{1+\left \vert \mu \right \vert }.
\end{equation*}
\end{proof}

\begin{lemma}
\label{estim omega}Assume (\ref{NewH1})$\sim $(\ref{MoinsABip}), let $f\in
L^{p}\left( 0,1;X\right) $ with $1<p<+\infty $ and set for $x\in \left[ 0,1%
\right] $%
\begin{equation}\label{def I et J}
\dis I_{\lambda ,f}\left( x\right) = \displaystyle\dfrac{1}{2}Q_{\lambda}^{-1}\int_{0}^{x}e^{(x-s)Q_{\lambda }}f(s)ds \quad \text{and} \quad \dis J_{\lambda ,f}\left( x\right) = \displaystyle\dfrac{1}{2}Q_{\lambda}^{-1}\int_{x}^{1}e^{(s-x)Q_{\lambda }}f(s)ds;
\end{equation}%
then, there exists $M\geqslant 0$ (independent of $\lambda $ and $f$) such
that%
\begin{equation}
\left \Vert Q_{\lambda }I_{\lambda ,f}\left( 1\right) \right \Vert \leqslant
M\left \Vert f\right \Vert _{_{L^{p}(0,1;X)}}\quad\text{and}\quad\left \Vert
Q_{\lambda }J_{\lambda ,f}\left( 0\right) \right \Vert \leqslant M\left
\Vert f\right \Vert _{_{L^{p}(0,1;X)}},  \label{Est QI et QJ}
\end{equation}%
moreover $I_{\lambda ,f},J_{\lambda ,f}\in W^{2,p}\left( 0,1;X\right) \cap
L^{p}\left( 0,1;D\left( A\right) \right) $ with 
\begin{equation*}
\left \Vert Q_{\lambda }^{2}I_{\lambda ,f}\right \Vert
_{L^{p}(0,1;X)}\leqslant M\left \Vert f\right \Vert _{_{L^{p}(0,1;X)}}\quad\text{and}\quad\left \Vert Q_{\lambda }^{2}J_{\lambda ,f}\right \Vert
_{L^{p}(0,1;X)}\leqslant M\left \Vert f\right \Vert _{_{L^{p}(0,1;X)}}.
\end{equation*}
\end{lemma}

\begin{proof}
From (\ref{estim exp xqlambda}), we have%
\begin{equation*}
\left \{ 
\begin{array}{l}
\dis \left \Vert Q_{\lambda }I_{\lambda ,f}\left( 1\right) \right \Vert \leqslant
\dis \int \nolimits_{0}^{1}\left \Vert e^{\left( 1-s\right) Q_{\lambda }}f\left(s\right) \right \Vert ds\leqslant M\int \nolimits_{0}^{1}\left \Vert
f\left( s\right) \right \Vert ds\leqslant M\left \Vert f\right \Vert_{_{L^{p}(0,1;X)}} \\ \ecart 
\dis \left \Vert Q_{\lambda }J_{\lambda ,f}\left( 0\right) \right \Vert \leqslant
\int \nolimits_{0}^{1}\left \Vert e^{sQ_{\lambda }}f\left( s\right) \right
\Vert ds\leqslant M\int \nolimits_{0}^{1}\left \Vert f\left( s\right)
\right \Vert ds\leqslant M\left \Vert f\right \Vert _{_{L^{p}(0,1;X)}}.%
\end{array}%
\right.
\end{equation*}

We apply Lemma \ref{Reg max} with $E=X,L=-A$, $G_{\lambda }=Q_{\lambda }$, $%
a=0$, $b=1$ so that 
\begin{equation*}
I_{\lambda ,f}=\dfrac{1}{2}Q_{\lambda }^{-1}U_{\lambda ,f}\quad \text{and}\quad
J_{\lambda ,f}=\dfrac{1}{2}Q_{\lambda }^{-1}V_{\lambda ,f};
\end{equation*}%
then $Q_{\lambda }^{2}I_{\lambda ,f}=\dfrac{1}{2}Q_{\lambda }U_{\lambda ,f}$
and $Q_{\lambda }^{2}J_{\lambda ,f}=\dfrac{1}{2}Q_{\lambda }V_{\lambda ,f}$
have the desired estimates.
\end{proof}

\begin{lemma}
\label{estim vj}Assume (\ref{NewH1})$\sim $(\ref{MoinsABip}) and let $f\in
L^{p}\left( 0,1;X\right) $ with $1<p<+\infty $. Set for $x\in \left[ 0,1\right] $ 
\begin{equation*}
v_{0}\left( x\right) =e^{xQ_{\lambda }}J_{\lambda ,f}\left( 0\right) \quad\text{and}\quad v_{1}\left( x\right) =e^{xQ_{\lambda }}I_{\lambda ,f}\left(
1\right) .
\end{equation*}
where $I_{\lambda,f}$ and $J_{\lambda,f}$ are given by \eqref{def I et J}. Moreover, there exists $M > 0$ (independent of $\lambda $ and $f$)
such that 
\begin{equation}
\left \Vert Q_{\lambda }^{2}v_{j}\right \Vert _{L^{p}(0,1;X)}\leqslant
M\left \Vert f\right \Vert _{_{L^{p}(0,1;X)}},~j=0,1.  \label{est vj}
\end{equation}
\end{lemma}

\begin{proof}
From \cite{Hammou}, Proposition 3.5, p. 1676, we have
$$J_{\lambda,f}(0) \in (D(Q_\lambda^2),X)_{\frac{1}{2p},p} = (D(Q_\lambda),X)_{1+\frac{1}{p},p};$$ 
so for $x\in (0,1]$, we can write
\begin{eqnarray*}
Q_{\lambda }^{2}v_{0}\left( x\right) &=& \dis e^{xQ_{\lambda }}Q_{\lambda }\int
\nolimits_{0}^{1}e^{sQ_{\lambda }}f\left( s\right) ds \vspace{0.1cm} \\
&=&\dis Q_{\lambda }\int \nolimits_{0}^{x}e^{\left( x-s\right)Q_{\lambda }} e^{2sQ_{\lambda }} f\left( s\right) ds+ e^{2xQ_{\lambda }} Q_{\lambda }\int
\nolimits_{x}^{1}e^{\left( s-x\right) Q_{\lambda }} f\left(
s\right) ds \vspace{0.1cm}\\
&=&2Q_{\lambda }^{2}I_{\lambda ,g}\left( x\right)
+2e^{2xQ_{\lambda }}Q_{\lambda }^{2}J_{\lambda ,f}\left( x\right) ,
\end{eqnarray*}%
with $g=e^{2\cdot Q_{\lambda }}f\left( \cdot \right) $. From Lemma \ref%
{estim omega}, statement 1. and (\ref{estim exp xqlambda}), we have 
\begin{equation*}
\left \{ 
\begin{array}{l}
\left \Vert e^{2\cdot Q_{\lambda }}Q_{\lambda }^{2}J_{\lambda ,f}\left(
\cdot \right) \right \Vert _{L^{p}(0,1;X)}\leqslant M\left \Vert Q_{\lambda
}^{2}J_{\lambda ,f}\left( \cdot \right) \right \Vert_{L^{p}(0,1;X)}\leqslant M\left \Vert f\right \Vert _{_{L^{p}(0,1;X)}} \vspace{0.1cm}\\ 
\left \Vert Q_{\lambda }^{2}I_{\lambda ,g}\right \Vert
_{L^{p}(0,1;X)}\leqslant M\left \Vert g\right \Vert
_{_{L^{p}(0,1;X)}}\leqslant M\left \Vert f\right \Vert _{_{L^{p}(0,1;X)}},%
\end{array}%
\right.
\end{equation*}%
from which we deduce $\left \Vert Q_{\lambda }^{2}v_{0}\right \Vert_{L^{p}(0,1;X)}\leqslant M\left \Vert f\right \Vert _{_{L^{p}(0,1;X)}}$.

The same estimate runs for $v_{1}$ since 
\begin{equation*}
v_{1}=e^{\cdot Q_{\lambda }}J_{\lambda ,f\left( 1-\cdot \right) }\left(
0\right) \quad\text{and}\quad\left \Vert f\left( 1-\cdot \right) \right \Vert
_{_{L^{p}(0,1;X)}}=\left \Vert f\right \Vert _{_{L^{p}(0,1;X)}}.
\end{equation*}
\end{proof}

\subsection{Proofs to Theorem \ref{Main1} and Theorem \ref{Main2}}

Let $r_{0}$ be fixed as in Lemma \ref{Grand Lamba inversible}.

\subsubsection{Proof to Theorem \ref{Main1}}

We apply Proposition \ref{Prop without lambda}, with $A,H,Q$ and $\Lambda $
replaced by 
\begin{equation*}
A-\lambda I,H+\mu I,Q_{\lambda }\text{ and }\Lambda _{\lambda ,\mu },
\end{equation*}
since in this case Problem (\ref{Problem}) becomes Problem (\ref{NewEquation}%
)-(\ref{RobinSpectralConditions}). So, it is enough to verify that (\ref%
{NewH1})$\sim $(\ref{NewH3}) imply $(H_{1})\sim (H_{5}).$

It is clear that (\ref{NewH1}), (\ref{NewH2}), (\ref{MoinsABip}) imply $%
(H_{1}),\left( H_{2}\right) ,(H_{3})$ mentioned in section \ref{sans
parametre}. Moreover, due to (\ref{inversibilty det}), assumptions (\ref%
{NewH1})$\sim $(\ref{D(H)Principal}) imply $(H_{4})$. Finally, under (\ref%
{NewH1})$\sim $(\ref{D(H)Principal}) 
\begin{equation*}
\Lambda _{\lambda ,\mu }^{-1}\left( X\right) \subset D\left( Q\right) \cap
D(H)\subset D\left( Q^{2}\right) ,
\end{equation*}%
so that $Q\Lambda _{\lambda ,\mu }^{-1}\left( X\right) \subset D\left(
Q\right) $ and then $(H_{5})$ is satisfied.

Note that here, the condition $\Lambda _{\lambda ,\mu }^{-1}d_{0}\in \left(
D\left( A\right) ,X\right) _{\frac{1}{2p},p}$ is automatically realized
since for any $d_{0} \in X$, we have $\Lambda _{\lambda ,\mu }^{-1}d_{0}\in D\left(
Q\right) \cap D(H)\subset D\left( Q^{2}\right) .$

\subsubsection{Proof to Theorem \ref{Main2}}

Let $\left( \lambda ,\mu \right) \in \Omega _{\varphi _{0},\varphi
_{1},r_{0}}$ and $f\in L^{p}\left( 0,1;X\right) $. We recall that, taking
into account the notations (\ref{def I et J}), we have, for $x\in \left[ 0,1%
\right] $%
\begin{equation*}
u(x)=S_{\lambda }\left( x\right) \mu _{0}+S_{\lambda }\left( 1-x\right) \mu
_{1}+I_{\lambda ,f}\left( x\right) +J_{\lambda ,f}\left( x\right) ,
\end{equation*}%
where%
\begin{equation*}
\left \{ 
\begin{array}{l}
\mu _{1}=u_{1}-I_{\lambda ,f}\left( 1\right) \vspace{0.1cm}\\ 
\mu _{0}=\Lambda _{\lambda ,\mu }^{-1}\left[ \left( I-e^{2Q_{\lambda
}}\right) d_{0}+2Q_{\lambda }e^{Q_{\lambda }}\mu _{1}+2Q_{\lambda
}J_{\lambda ,f}\left( 0\right) \right] -J_{\lambda ,f}\left( 0\right) \vspace{0.1cm}\\ 
S_{\lambda }\left( x\right) =\left( I-e^{2Q_{\lambda }}\right) ^{-1}\left(
e^{xQ_{\lambda }}-e^{\left( 1-x\right) Q_{\lambda }}e^{Q_{\lambda }}\right)
\in \mathcal{L}\left( X\right) .%
\end{array}%
\right.
\end{equation*}%
So we can write $u=h_{0}+h_{1}-h_{2}+h_{3}+h_{4}$ with 
\begin{equation}
\left \{ 
\begin{array}{l}
h_{0}\left( x\right) =S_{\lambda }\left( x\right) \Lambda _{\lambda ,\mu
}^{-1}\left[ \left( I-e^{2Q_{\lambda }}\right) d_{0}+2Q_{\lambda }\left(
J_{\lambda ,f}\left( 0\right) -e^{Q_{\lambda }}I_{\lambda ,f}\left( 1\right)
\right) \right] \vspace{0.1cm}\\ 
h_{1}\left( x\right) =2S_{\lambda }\left( x\right) \Lambda _{\lambda ,\mu
}^{-1}Q_{\lambda }e^{Q_{\lambda }}u_{1} \vspace{0.1cm}\\ 
h_{2}\left( x\right) =S_{\lambda }\left( x\right) J_{\lambda ,f}\left(
0\right) +S_{\lambda }\left( 1-x\right) I_{\lambda ,f}\left( 1\right)\vspace{0.1cm} \\ 
h_{3}\left( x\right) =S_{\lambda }\left( 1-x\right) u_{1} \vspace{0.1cm}\\ 
h_{4}\left( x\right) =I_{\lambda ,f}\left( x\right) +J_{\lambda ,f}\left(
x\right) .%
\end{array}%
\right.  \label{def h}
\end{equation}

\begin{description}
\item[Estimate of $Q_{\protect\lambda }^{2}h_{0}.$] For $\xi \in X$ and $%
x\in \left( 0,1\right) $, we have%
\begin{eqnarray*}
\left \Vert Q_{\lambda }^{2}S_{\lambda }\left( x\right) \Lambda _{\lambda
,\mu }^{-1}\xi \right \Vert &\hspace{-0.27cm}=&\hspace{-0.27cm}\left \Vert \left( I-e^{2Q_{\lambda }}\right) ^{-1}\left( I-e^{2\left(1-x\right) Q_{\lambda }}\right)e^{xQ_{\lambda }}Q_{\lambda }^{2}\Lambda_{\lambda ,\mu }^{-1}\xi \right \Vert \\
&\hspace{-0.27cm}\leqslant &\hspace{-0.27cm}\left \Vert \left( I-e^{2Q_{\lambda }}\right) ^{-1}\left(I-e^{2\left( 1-x\right) Q_{\lambda }}\right) \right \Vert _{\mathcal{L}%
(X)}\left \Vert e^{xQ_{\lambda }}\right \Vert _{\mathcal{L}(X)}\left \Vert
Q_{\lambda }^{2}\Lambda _{\lambda ,\mu }^{-1}\right \Vert _{\mathcal{L}%
(X)}\left \Vert \xi \right \Vert \\
&\hspace{-0.27cm}\leqslant &\hspace{-0.27cm}M\dfrac{1+\left \vert \lambda \right \vert +\left \vert \mu
\right \vert }{1+\left \vert \mu \right \vert }\left \Vert \xi \right \Vert ,
\end{eqnarray*}%
so, from (\ref{Q2InvLambda}) and (\ref{Est QI et QJ}), we deduce 
\begin{eqnarray*}
\left \Vert Q_{\lambda }^{2}h_{0}\left( x\right) \right \Vert &\leqslant &M\dfrac{1+\left \vert \lambda \right \vert +\left \vert \mu
\right \vert }{1+\left \vert \mu \right \vert }\left( \left \Vert
d_{0}\right \Vert +2\left \Vert Q_{\lambda }J_{\lambda ,f}\left( 0\right)
\right \Vert +2\left \Vert e^{Q_{\lambda }}\right \Vert _{\mathcal{L}%
(X)}\left \Vert Q_{\lambda }I_{\lambda ,f}\left( 1\right) \right \Vert
\right) \\
&\leqslant &M\dfrac{1+\left \vert \lambda \right \vert +\left \vert \mu
\right \vert }{1+\left \vert \mu \right \vert }\left( \left \Vert
d_{0}\right \Vert +\left \Vert f\right \Vert _{L^{p}(0,1;X)}\right) .
\end{eqnarray*}%
Then%
\begin{equation*}
\left \Vert Q_{\lambda }^{2}h_{0}\right \Vert _{L^{p}(0,1;X)}\leqslant M%
\dfrac{1+\left \vert \lambda \right \vert +\left \vert \mu \right \vert }{%
1+\left \vert \mu \right \vert }\left( \left \Vert d_{0}\right \Vert +\left
\Vert f\right \Vert _{L^{p}(0,1;X)}\right) .
\end{equation*}

\item[Estimate of $Q_{\protect\lambda }^{2}h_{1}.$] As above, we have for $%
\xi \in X$ and $x\in \left( 0,1\right) $%
\begin{equation*}
\left \Vert Q_{\lambda }^{2}h_{1}\left( x\right) \right \Vert \leqslant M%
\dfrac{1+\left \vert \lambda \right \vert +\left \vert \mu \right \vert }{%
1+\left \vert \mu \right \vert }\left( \left \Vert Q_{\lambda }e^{Q_{\lambda
}}\right \Vert _{\mathcal{L}(X)}\left \Vert u_{1}\right \Vert \right),
\end{equation*}%
and from (\ref{estim exp Q}), we deduce that $\left \Vert Q_{\lambda }^{2}h_{1}\left( x\right) \right \Vert \leqslant
M\left \Vert u_{1}\right \Vert$, hence
\begin{equation*}
\left \Vert Q_{\lambda }^{2}h_{1}\right \Vert _{L^{p}(0,1;X)}\leqslant
M\left \Vert u_{1}\right \Vert .
\end{equation*}

\item[Estimate of $Q_{\protect\lambda }^{2}h_{2}.$] For $\xi \in X$ and $%
x\in (0,1]$, we have%
\begin{eqnarray}
\left \Vert Q_{\lambda }^{2}S_{\lambda }\left( x\right) \xi \right \Vert
&=&\left \Vert \left( I-e^{2Q_{\lambda }}\right) ^{-1}\left( I-e^{2\left(
1-x\right) Q_{\lambda }}\right) Q_{\lambda }^{2}e^{xQ_{\lambda }}\xi \right
\Vert  \notag \\
&\leqslant &\left \Vert \left( I-e^{2Q_{\lambda }}\right) ^{-1}\left(
I-e^{2\left( 1-x\right) Q_{\lambda }}\right) \right \Vert _{\mathcal{L}%
(X)}\left \Vert Q_{\lambda }^{2}e^{xQ_{\lambda }}\xi \right \Vert  \notag \\
&\leqslant &M\left \Vert Q_{\lambda }^{2}e^{xQ_{\lambda }}\xi \right \Vert ,
\label{Est Q2 SLambda}
\end{eqnarray}%
so%
\begin{equation*}
\left \Vert Q_{\lambda }^{2}h_{2}\left( x\right) \right \Vert \leqslant
M\left( \left \Vert Q_{\lambda }^{2}e^{xQ_{\lambda }}J_{\lambda ,f}\left(
0\right) \right \Vert +\left \Vert Q_{\lambda }^{2}e^{\left( 1-x\right)
Q_{\lambda }}I_{\lambda ,f}\left( 1\right) \right \Vert \right) ,
\end{equation*}%
and\ then, from (\ref{est vj})%
$$\left \Vert Q_{\lambda }^{2}h_{2}\right \Vert _{L^{p}(0,1;X)} \leqslant
M\left( \left \Vert Q^{2}v_{0}\right \Vert _{L^{p}(0,1;X)}+\left \Vert
Q^{2}v_{1}\left( 1-\cdot \right) \right \Vert _{L^{p}(0,1;X)}\right) \leqslant M\left \Vert f\right \Vert _{L^{p}(0,1;X)}.$$

\item[Estimate of $Q_{\protect\lambda }^{2}h_{3}.$] Due to (\ref{Est Q2
SLambda}), we have, for $x\in [ 0,1)$
\begin{equation*}
\left \Vert Q_{\lambda }^{2}h_{3}\left( x\right) \right \Vert \leqslant
M\left \Vert Q_{\lambda }^{2}e^{\left( 1-x\right) Q_{\lambda }}u_{1}\right
\Vert .
\end{equation*}%
From Theorem 2.1 in \cite{Dore Yakubov}, since $u_{1}\in (D(A),X)_{\frac{1}{2p},p}=(X,D(A))_{1-\frac{1}{2p},p},$ we get%
$$\left \Vert Q_{\lambda }^{2}h_{3}\right \Vert _{_{L^{p}(0,1;X)}} \leqslant M\left \Vert Q_{\lambda }^{2}e^{\left( 1-x\right) Q_{\lambda }}u_{1}\right
\Vert _{_{L^{p}(0,1;X)}} \leqslant M\left( \left \Vert u_{1}\right \Vert _{(D(A),X)_{\frac{1}{2p},p}}+\left \vert \lambda \right \vert ^{1-\frac{1}{2p}}\left \Vert u_{1}\right \Vert \right) .$$

\item[Estimate of $Q_{\protect\lambda }^{2}h_{4}.$] From Lemma \ref{estim
omega}, we get%
\begin{equation*}
\left \Vert Q_{\lambda }^{2}h_{4}\right \Vert _{L^{p}(0,1;X)}\leqslant
M\left \Vert f\right \Vert _{_{L^{p}(0,1;X)}}.
\end{equation*}
\end{description}
Summarizing the previous study we obtain that 
\begin{equation}
\left \Vert Q_{\lambda }^{2}u\right \Vert _{_{L^{p}(0,1;X)}}\leqslant
M\alpha \left( d_{0},u_{1},\lambda ,\mu ,f\right) .  \label{estim final}
\end{equation}%
Moreover, since $u$ satisfies (\ref{NewEquation}), that is,
\begin{equation*}
u^{\prime \prime }(x)-Q_{\lambda }^{2}u(x)=f(x),\text{ \ a.e. }x\in (0,1),
\end{equation*}%
we deduce that%
\begin{equation*}
\left \Vert u^{\prime \prime }\right \Vert _{_{L^{p}(0,1;X)}}\leqslant
M\alpha \left( d_{0},u_{1},\lambda ,\mu ,f\right) .
\end{equation*}%
Writing $u=Q_{\lambda }^{-2}Q_{\lambda }^{2}u$ and $Q_{\lambda }u=Q_{\lambda
}^{-1}Q_{\lambda }^{2}u$, we obtain the estimates concerning $u$ and $%
Q_{\lambda }u$. Setting, for $x\in \left[ 0,1\right] $%
\begin{equation*}
\widetilde{S}_{\lambda }\left( x\right) =\left( I-e^{2Q_{\lambda }}\right)
^{-1}\left( e^{xQ_{\lambda }}+e^{\left( 1-x\right) Q_{\lambda
}}e^{Q_{\lambda }}\right) \in \mathcal{L}\left( X\right) ,
\end{equation*}%
we have 
\begin{equation*}
u^{\prime }(x)=Q_{\lambda }\widetilde{S}_{\lambda }\left( x\right) \mu
_{0}-Q_{\lambda }\widetilde{S}_{\lambda }\left( 1-x\right) \mu
_{1}+QI_{\lambda ,f}\left( x\right) -QJ_{\lambda ,f}\left( x\right)
=Q_{\lambda }^{-1}Q_{\lambda }^{2}\omega \left( x\right) ,
\end{equation*}%
the terms in $\omega \left( x\right) =\widetilde{S}_{\lambda }\left(
x\right) \mu _{0}-\widetilde{S}_{\lambda }\left( 1-x\right) \mu
_{1}+I_{\lambda ,f}\left( x\right) -J_{\lambda ,f}\left( x\right) $ are (in
absolute value) those of $u(x),$ so (\ref{estim final}) runs when we replace 
$u$ by $\omega $ , this furnishes the estimate for $u^{\prime }.$

From Lemma 2.6 a) p. 103 in \cite{Dore Yakubov} we have
\begin{equation}\label{est Q invQLambda}
\left \Vert QQ_{\lambda }^{-1}\right \Vert _{\mathcal{L}(X)}=\left \Vert
(-A)^{1/2}(-A+\lambda I)^{-1/2}\right \Vert _{\mathcal{L}(X)}\leqslant M;
\end{equation}
so writing $Qu=QQ_{\lambda }^{-1}Q_{\lambda }u,Q^{2}u=\left( QQ_{\lambda
}^{-1}\right) ^{2}Q_{\lambda }^{2}u$, we deduce the estimates of $\left
\Vert Qu\right \Vert _{_{L^{p}(0,1;X)}}$ and $\left \Vert Q^{2}u\right\Vert_{_{L^{p}(0,1;X)}}$ from those of $\left \Vert Q_{\lambda }u\right\Vert _{L^{p}(0,1;X)}$ and $\left \Vert Q_{\lambda}^{2}u\right \Vert _{L^{p}(0,1;X)}$.

\begin{remark}
Under the assumptions of the previous theorem, we obtain moreover that%
\begin{equation}
\left\Vert u(0)\right\Vert \leqslant \dfrac{M}{1+\left\vert \mu \right\vert }%
\left( \left\Vert d_{0}\right\Vert +Me^{-2c_{0}\left\vert \lambda
\right\vert ^{1/2}}\left\Vert u_{1}\right\Vert +\left\Vert f\right\Vert
_{_{L^{p}(0,1;X)}}\right) .  \label{estim u(0)}
\end{equation}%
Indeed%
\begin{equation*}
u(0)=\Lambda _{\lambda ,\mu }^{-1}\left[ \left( I-e^{2Q_{\lambda }}\right)
d_{0}+2Q_{\lambda }e^{Q_{\lambda }}\left( u_{1}-I_{\lambda ,f}\left(
1\right) \right) +2Q_{\lambda }J_{\lambda ,f}\left( 0\right) \right] ,
\end{equation*}%
so%
\begin{eqnarray*}
\left\Vert u(0)\right\Vert &\leqslant &\left\Vert \Lambda _{\lambda ,\mu
}^{-1}\right\Vert _{\mathcal{L}(X)}\left\Vert I-e^{2Q_{\lambda }}\right\Vert
_{\mathcal{L}(X)}\left\Vert d_{0}\right\Vert +\left\Vert \Lambda _{\lambda
,\mu }^{-1}\right\Vert _{\mathcal{L}(X)}\left\Vert 2Q_{\lambda
}e^{Q_{\lambda }}\right\Vert _{\mathcal{L}(X)}\left\Vert u_{1}\right\Vert \\
&&+2\left\Vert \Lambda _{\lambda ,\mu }^{-1}\right\Vert _{\mathcal{L}%
(X)}\left( \left\Vert e^{Q_{\lambda }}\right\Vert _{\mathcal{L}%
(X)}\left\Vert Q_{\lambda }I_{\lambda ,f}\left( 1\right) \right\Vert
+\left\Vert Q_{\lambda }J_{\lambda ,f}\left( 0\right) \right\Vert \right) \\
&\leqslant &\dfrac{M}{1+\left\vert \mu \right\vert }\left( \left\Vert
d_{0}\right\Vert +Me^{-2c_{0}\left\vert \lambda \right\vert
^{1/2}}\left\Vert u_{1}\right\Vert +\left\Vert f\right\Vert
_{_{L^{p}(0,1;X)}}\right) .
\end{eqnarray*}
\end{remark}

\subsubsection{Proof to Theorem \ref{gen sg first case}}\label{espace produit}

\begin{remark}
Let $\left( f,\tau \right) \in Z$. We consider the following problem%
\begin{equation}\label{Pb l=m}
\left\{ \begin{array}{l}
u^{\prime \prime }+\mathcal{A}u-\lambda u=f \vspace{0.1cm} \\ 
u^{\prime }(0)-Hu(0)-\left( \lambda +\mu \right) u(0)=\tau \vspace{0.1cm}\\ 
u\left( 1\right) =0;
\end{array}\right.  
\end{equation}%
then the two following assertions are equivalent:

\begin{enumerate}
\item $\left( u,v\right) \in D\left( \mathcal{P}_{A,H,\mu }\right) $ and $%
\left( \mathcal{P}_{A,H,\mu }-\lambda I\right) \left( u,v\right) =\left(
f,\tau \right) $.

\item $u\in W^{2,p}(0,1;X)\cap L^{p}(0,1;D(A))$ is a classical solution of (%
\ref{Pb l=m}) together with $v=u(0)$.
\end{enumerate}
So to study $\mathcal{P}_{A,H,\mu }$, it remains to solve (\ref{Pb l=m}).
\end{remark}

We set $\varphi _{2}:=\min \left\{ \varphi _{0},\varphi _{1}\right\} $ and
define for $\varphi _{3}\in \left( 0,\pi -\varphi _{2}\right) $, $r_{\varphi
_{3}}\in \left( r_{0},+\infty \right) $ by 
\begin{equation*}
r_{\varphi _{3}}:=\dfrac{r_{0}}{\cos ^{2}\left( \frac{\varphi _{2}+\varphi
_{3}}{2}\right) },
\end{equation*}
for $r_0 > 0$, see Theorem \ref{Main1}. 
\begin{proposition}
\label{Est L first case}Let $\varphi _{3}\in \left( 0,\pi -\varphi
_{2}\right) .$

\begin{enumerate}
\item If $\lambda \in S_{\varphi _{2}},\mu \in S_{\varphi _{3}}$ with $%
\left\vert \lambda \right\vert \geqslant r_{\varphi _{3}}$, then $\left(
\lambda ,\lambda +\mu \right) \in \Omega _{\varphi _{0},\varphi _{1},r_{0}}$.

\item Let $\mu \in S_{\varphi _{3}}$. Then $\mathcal{P}_{A,H,\mu }$ is a
closed linear operator on $Z$ with%
\begin{equation*}
S_{\varphi _{2}}\backslash B\left( 0,r_{\varphi _{3}}\right) \subset \rho
\left( \mathcal{P}_{A,H,\mu }\right) .
\end{equation*}

Moreover, let $\lambda \in S_{\varphi _{2}}\backslash B\left( 0,r_{\varphi
_{3}}\right) $ and $\left( f,\tau \right) \in Z$; then 
\begin{equation*}
\left( u,v\right) =(\mathcal{P}_{A,H,\mu }-\lambda I)^{-1}\left( f,\tau
\right) ,
\end{equation*}%
satisfies, for $x\in \lbrack 0,1]$%
\begin{equation*}
\left\{ 
\begin{array}{lll}
u(x) & =& S_{\lambda }\left( x\right) \Lambda _{\lambda ,\lambda +\mu
}^{-1}\left( I-e^{2Q_{\lambda }}\right) \tau \vspace{0.1cm} \\ 
&& +S_{\lambda }\left( x\right) \left[ 2\Lambda _{\lambda ,\lambda +\mu}^{-1}Q_{\lambda }\left[ J_{\lambda ,f}\left( 0\right) -e^{Q_{\lambda}}I_{\lambda ,f}\left( 1\right) \right] -J_{\lambda ,f}\left( 0\right) \right] \vspace{0.1cm} \\ 
&& -S_{\lambda }\left( 1-x\right) I_{\lambda ,f}\left( 1\right) +I_{\lambda
,f}\left( x\right) +J_{\lambda ,f}\left( x\right) \\ \ecart 
v(x) & = & u\left( 0\right) ,%
\end{array}%
\right.
\end{equation*}%
where $S_{\lambda }\left( x\right) =\left( I-e^{2Q_{\lambda }}\right)
^{-1}\left( e^{xQ_{\lambda }}-e^{\left( 1-x\right) Q_{\lambda
}}e^{Q_{\lambda }}\right) \in \mathcal{L}\left( X\right) $.

\item There exists $M_{A,H,\varphi _{3}}>0$ such that for $\lambda \in
S_{\varphi _{2}}\backslash B\left( 0,r_{\varphi _{3}}\right) $ and $\mu \in
S_{\varphi _{3}}$ we have%
\begin{equation*}
\left\Vert (\mathcal{P}_{A,H,\mu }-\lambda I)^{-1}\right\Vert _{\mathcal{L}%
\left( Z\right) }\leqslant \dfrac{M_{A,H,\varphi _{3}}}{1+|\lambda |}.
\end{equation*}
\end{enumerate}
\end{proposition}

\begin{proof}\hfill

\begin{enumerate}
\item We have $\left( \lambda ,\lambda +\mu \right) \in S_{\varphi
_{2}}\times S_{\varphi _{2}}\subset S_{\varphi _{0}}\times S_{\varphi _{1}}$%
, moreover $\left\vert \lambda \right\vert \geqslant r_{\varphi _{3}}>r_{0}$
and, due to \cite{Dore Yakubov}, Lemma 2.3, p. 98, we have
\begin{equation*}
\frac{\left\vert \lambda +\mu \right\vert ^{2}}{\left\vert \lambda
\right\vert }\geqslant \cos ^{2}\left( \tfrac{\varphi _{2}+\varphi _{3}}{2}%
\right) \frac{\left( \left\vert \lambda \right\vert +\left\vert \mu
\right\vert \right) ^{2}}{\left\vert \lambda \right\vert }\geqslant \cos
^{2}\left( \tfrac{\varphi _{2}+\varphi _{3}}{2}\right) \times \left\vert
\lambda \right\vert \geqslant r_{0}.
\end{equation*}

\item It is a consequence of statement 1. and Theorem \ref{Main1}.

\item As in statement 1., we have, $\left\vert \lambda +\mu \right\vert
\geqslant \cos \left( \tfrac{\varphi _{2}+\varphi _{3}}{2}\right) \times
\left\vert \lambda \right\vert $, so setting 
\begin{equation*}
C_{\varphi _{3}}:=\frac{1}{\cos \left( \frac{\varphi _{2}+\varphi _{3}}{2}%
\right) }+1,
\end{equation*}%
we have 
\begin{equation*}
\dfrac{1+\left\vert \lambda \right\vert +\left\vert \lambda +\mu \right\vert 
}{1+\left\vert \lambda +\mu \right\vert }\leqslant \dfrac{\left\vert \lambda
\right\vert }{1+\left\vert \lambda +\mu \right\vert }+\dfrac{1+\left\vert
\lambda +\mu \right\vert }{1+\left\vert \lambda +\mu \right\vert }\leqslant
C_{\varphi _{3}}.
\end{equation*}%
Let $\left( f,\tau \right) \in Z$, then Theorem \ref{Main2} and (\ref{estim u(0)}) imply that $\left( u,v\right) =(\mathcal{P}_{A,H,\mu }-\lambda I)^{-1}\left( f,\tau\right) ,$ satisfies
\begin{equation*}
\left\Vert u\right\Vert _{L^{p}(0,1;X)}\leqslant \dfrac{MC_{\varphi _{3}}}{%
1+|\lambda |}\left( \left\Vert \tau \right\Vert +\left\Vert f\right\Vert
_{Y}\right)  \quad \text{and} \quad v=\left\Vert u(0)\right\Vert \leqslant \dfrac{MC_{\varphi _{3}}}{1+|\lambda |}\left( \left\Vert \tau \right\Vert +\left\Vert f\right\Vert _{Y}\right) ,
\end{equation*}%
that is%
\begin{equation*}
\left\Vert (\mathcal{P}_{A,H,\mu }-\lambda I)^{-1}\left( f,\tau \right)
\right\Vert _{Z}\leqslant \dfrac{MC_{\varphi _{3}}}{1+|\lambda |}\left\Vert
\left( f,\tau \right) \right\Vert _{Z}.
\end{equation*}
\end{enumerate}
\end{proof}
The proof of Theorem \ref{gen sg first case} is given by Statement $3$ of the previous proposition.

\section{Spectral problem (\protect\ref{NewEquation})-(\protect\ref{RobinSpectralConditions}): second case\label{Section 2eme Cas}}

In all this section we suppose that $X,A,H$ satisfy (\ref{NewH1})$\sim $(\ref%
{MoinsABip}) and (\ref{D(Q)Principal})$\sim $(\ref{inv Lamb and interp}).

Note that the results of the previous section obtained under assumption (\ref%
{NewH1})$\sim $(\ref{MoinsABip}) can be used here, in particular results of
subsection \ref{spect 1}, Lemma \ref{estim omega}, Lemma \ref{estim vj} and
also estimate (\ref{est Q invQLambda}).

\subsection{Spectral estimates}

Let $\lambda \in S_{\varphi _{0}},\mu \in \mathbb{C}$. Recall that $H_{\mu }=H+\mu I$ and $Q_{\lambda }=-(-A+\lambda I)^{1/2}$.
We first furnish estimates concerning operators $Q_{\lambda },H_{\mu }$
which are easy consequences of our assumptions.

Again, in the following $M$ denotes various constants, independent of $%
\lambda ,\mu $, which can vary from one line to another.

\begin{lemma}
Let $\lambda \in S_{\varphi _{0}},\mu \in \mathbb{C}$. Then $H_{\mu }Q_{\lambda }^{-1}$ $\in \mathcal{L}\left( X\right) $,
moreover there exists a constant $M > 0$ independent of $\lambda$ and $\mu$ such that%
\begin{equation}
\left \Vert H_{\mu }Q_{\lambda }^{-1}\right \Vert _{\mathcal{L}(X)}\leqslant
M\frac{1+\left \vert \mu \right \vert }{\left( 1+\left \vert \lambda \right
\vert \right) ^{\varepsilon }}.  \label{Consequence3}
\end{equation}
\end{lemma}

\begin{proof}
From (\ref{est Q invQLambda}), we deduce%
\begin{eqnarray*}
\left \Vert Q_{\left \vert \lambda \right \vert }Q_{\lambda }^{-1}\right
\Vert _{\mathcal{L}(X)} &=&\left \Vert \left( -A+\left \vert \lambda \right
\vert I\right) Q^{-1}Q_{\left \vert \lambda \right \vert }^{-1}QQ_{\lambda
}^{-1}\right \Vert _{\mathcal{L}(X)} \vspace{0.1cm} \\
&\leqslant &\left \Vert AQ^{-1}Q_{\left \vert \lambda \right \vert
}^{-1}QQ_{\lambda }^{-1}\right \Vert _{\mathcal{L}(X)}+\left \Vert \left
\vert \lambda \right \vert Q^{-1}Q_{\left \vert \lambda \right \vert
}^{-1}QQ_{\lambda }^{-1}\right \Vert _{\mathcal{L}(X)} \vspace{0.1cm} \\
&\leqslant &\left \Vert QQ_{\left \vert \lambda \right \vert }^{-1}\right
\Vert _{\mathcal{L}(X)}\left \Vert QQ_{\lambda }^{-1}\right \Vert _{\mathcal{%
L}(X)}+\left \vert \lambda \right \vert \left \Vert Q_{\left \vert \lambda
\right \vert }^{-1}\right \Vert _{\mathcal{L}(X)}\left \Vert Q_{\lambda
}^{-1}\right \Vert _{\mathcal{L}(X)} \vspace{0.1cm} \\
&\leqslant &M,
\end{eqnarray*}%
and, from (\ref{NewH3bis})%
\begin{equation*}
\parallel HQ_{\lambda }^{-1}\parallel _{\mathcal{L}\left( X\right)
}\leqslant \parallel HQ_{\left \vert \lambda \right \vert }^{-1}\parallel _{%
\mathcal{L}\left( X\right) }\parallel Q_{\left \vert \lambda \right \vert
}Q_{\lambda }^{-1}\parallel _{\mathcal{L}\left( X\right) }\leqslant \dfrac{M%
}{\left( 1+\left \vert \lambda \right \vert \right) ^{\varepsilon }}.
\end{equation*}%
Moreover%
\begin{eqnarray*}
\left \Vert H_{\mu }Q_{\lambda }^{-1}\right \Vert _{\mathcal{L}(X)}
&\leqslant &\left \Vert HQ_{\lambda }^{-1}\right \Vert _{\mathcal{L}%
(X)}+\left \vert \mu \right \vert \left \Vert Q_{\lambda }^{-1}\right \Vert
_{\mathcal{L}(X)} \\
&\leqslant &\dfrac{M}{\left( 1+\left \vert \lambda \right \vert \right)
^{\varepsilon }}+\frac{1+\left \vert \mu \right \vert }{\left( 1+\left \vert
\lambda \right \vert \right) ^{1/2}}\leqslant M\frac{1+\left \vert \mu
\right \vert }{\left( 1+\left \vert \lambda \right \vert \right)
^{\varepsilon }}.
\end{eqnarray*}
\end{proof}

\begin{remark}
From \eqref{Consequence3}, assumption \eqref{NewH3bis} can be written as
\begin{equation*}
\exists \,\varepsilon \in (0,1/2],~\exists \,C_{H,Q}>0,~\forall \lambda \in
S_{\varphi _{0}} \,: ~\left\| HQ_{\lambda }^{-1}\right\|_{\mathcal{L}\left(
X\right) }\leqslant \dfrac{C_{H,Q}}{\left( 1+\left \vert \lambda \right
\vert \right) ^{\varepsilon }}.
\end{equation*}
\end{remark}

We now introduce the notation : for $\rho >0$%
\begin{equation*}
\Pi _{\varphi _{0},\rho }=\left\{ \left( \lambda ,\mu \right) \in S_{\varphi
_{0}}\times \mathbb{C}:\left\vert \lambda \right\vert \geqslant \rho \text{ \ and \ }\frac{\left\vert \lambda \right\vert }{\left\vert \mu \right\vert ^{1/\varepsilon }%
}\geqslant \rho \right\} ,
\end{equation*}%
where we have set $\dfrac{\left\vert \lambda \right\vert }{\left\vert \mu
\right\vert ^{1/\varepsilon }}=+\infty $ for $\mu =0$ and furnished results (see below) on 
\begin{equation*}
\Lambda _{\lambda ,\mu }=\left( Q_{\lambda }-H_{\mu }\right) +e^{2Q_{\lambda
}}\left( Q_{\lambda }+H_{\mu }\right) ,
\end{equation*}%
where $\lambda \in S_{\varphi _{0}},\mu \in \mathbb{C}.$

\begin{lemma}
\label{Grand Lamba inversible 2}There exist $\rho _{0}>0$ and $M > 0$
such that for all $\left( \lambda ,\mu \right) \in \Pi _{\varphi _{0},\rho
_{0}}:$%
\begin{equation}
\max \left \{ \left \Vert H_{\mu }Q_{\lambda }^{-1}\right \Vert _{\mathcal{L}%
(X)},\left \Vert \left( I+e^{2Q_{\lambda }}\right) ^{-1}\left(
I-e^{2Q_{\lambda }}\right) H_{\mu }Q_{\lambda }^{-1}\right \Vert _{\mathcal{L%
}(X)}\right \} \leqslant 1/2,  \label{est2}
\end{equation}%
\begin{equation}
\left \{ 
\begin{array}{l}
0 \in \rho\left(I-\left( I+e^{2Q_{\lambda }}\right) ^{-1}\left( I-e^{2Q_{\lambda }}\right)
H_{\mu }Q_{\lambda }^{-1}\right) \vspace{0.1cm} \\ 
\left \Vert \left[ I-\left( I+e^{2Q_{\lambda }}\right) ^{-1}\left(
I-e^{2Q_{\lambda }}\right) H_{\mu }Q_{\lambda }^{-1}\right] ^{-1}\right
\Vert _{\mathcal{L}(X)}\leqslant 2,%
\end{array}
\right.  \label{first inv 2}
\end{equation}
\begin{equation}
0\in \rho\left(\Lambda _{\lambda ,\mu }\right) \quad \text{and} \quad \left \Vert
\Lambda _{\lambda ,\mu }^{-1}\right \Vert _{\mathcal{L}(X)}\leqslant \dfrac{M%
}{\left( 1+\left \vert \lambda \right \vert \right) ^{1/2}},
\label{inversibilty det 2}
\end{equation}%
\begin{equation}
\left \Vert Q_{\lambda }\Lambda _{\lambda ,\mu }^{-1}\right \Vert _{\mathcal{%
L}(X)}\leqslant M,  \label{Q2InvLambda 2}
\end{equation}%
and%
\begin{equation}
\left \{ 
\begin{array}{l}
0 \in \rho\left(Q_{\lambda }-H_{\mu }\right) \vspace{0.1cm} \\ 
\left \Vert \left( Q_{\lambda }-H_{\mu }\right) ^{-1}\right \Vert _{\mathcal{%
L}(X)}\leqslant \dfrac{M}{\left( 1+\left \vert \lambda \right \vert \right)
^{1/2}} \vspace{0.1cm} \\ 
\left \Vert Q_{\lambda }\left( Q_{\lambda }-H_{\mu }\right) ^{-1}\right
\Vert _{\mathcal{L}(X)}\leqslant M \vspace{0.1cm} \\ 
\left \Vert \left( Q_{\lambda }+H_{\mu }\right) \left( Q_{\lambda }-H_{\mu
}\right) ^{-1}\right \Vert _{\mathcal{L}(X)}\leqslant M \vspace{0.1cm} \\ 
\left \Vert e^{2Q_{\lambda }}\left( Q_{\lambda }+H_{\mu }\right) \left(
Q_{\lambda }-H_{\mu }\right) ^{-1}\right \Vert _{\mathcal{L}(X)}\leqslant
1/2.%
\end{array}%
\right.  \label{est Q-H}
\end{equation}
\end{lemma}

\begin{proof}
Let $\rho >0$ and $\left( \lambda ,\mu \right) \in \Pi _{\varphi _{0},\rho }$%
. Then 
\begin{equation*}
I-\left( I+e^{2Q_{\lambda }}\right) ^{-1}\left( I-e^{2Q_{\lambda }}\right)
H_{\mu }Q_{\lambda }^{-1}\in \mathcal{L}(X),
\end{equation*}
and from (\ref{Consequence3}) together with Lemma \ref{Inv (I-exp2Q)}%
\begin{eqnarray*}
\max \left \{ \left \Vert H_{\mu }Q_{\lambda }^{-1}\right \Vert _{\mathcal{%
L}(X)},\left \Vert \left( I+e^{2Q_{\lambda }}\right) ^{-1}\left(
I-e^{2Q_{\lambda }}\right) H_{\mu }Q_{\lambda }^{-1}\right \Vert _{\mathcal{L%
}(X)}\right \} &\hspace{-0.27cm}\leqslant &\hspace{-0.27cm} M\left \Vert H_{\mu }Q_{\lambda }^{-1}\right \Vert _{\mathcal{L}(X)} \\
&\hspace{-0.27cm}\leqslant & \hspace{-0.27cm} M\frac{1+\left \vert \mu \right \vert }{\left( 1+\left \vert
\lambda \right \vert \right) ^{\varepsilon }} \\
&\hspace{-0.27cm}\leqslant & \hspace{-0.27cm} M\left( \left( \frac{1}{\left \vert \lambda \right \vert }\right) ^{\varepsilon }\hspace{-0.1cm}+\hspace{-0.05cm}\left( \frac{\left \vert \mu \right \vert
^{1/\varepsilon }}{\left \vert \lambda \right \vert }\right) ^{\varepsilon
}\right) .
\end{eqnarray*}%
So there exists $\rho _{0}>0$ such that for all $\left( \lambda ,\mu \right)
\in \Pi _{\varphi _{0},\rho _{0}}$: (\ref{est2}) and (\ref{first inv 2})
hold. Now, let $\left( \lambda ,\mu \right) \in \Pi _{\varphi _{0},\rho
_{0}} $. We deduce that 
\begin{eqnarray*}
\Lambda _{\lambda ,\mu } &=&\left( I+e^{2Q_{\lambda }}\right) Q_{\lambda
}-\left( I-e^{2Q_{\lambda }}\right) H_{\mu } \\
&=&\left( I+e^{2Q_{\lambda }}\right) \left[ I-\left( I+e^{2Q_{\lambda
}}\right) ^{-1}\left( I-e^{2Q_{\lambda }}\right) H_{\mu }Q_{\lambda }^{-1}%
\right] Q_{\lambda },
\end{eqnarray*}%
is boundedly invertible with%
\begin{equation*}
\left \{ 
\begin{array}{c}
\Lambda _{\lambda ,\mu }^{-1}=Q_{\lambda }^{-1}\left[ I-\left(
I+e^{2Q_{\lambda }}\right) ^{-1}\left( I-e^{2Q_{\lambda }}\right) H_{\mu
}Q_{\lambda }^{-1}\right] ^{-1}\left( I+e^{2Q_{\lambda }}\right) ^{-1} \vspace{0.1cm} \\ 
Q_{\lambda }\Lambda _{\lambda ,\mu }^{-1}=\left[ I-\left( I+e^{2Q_{\lambda
}}\right) ^{-1}\left( I-e^{2Q_{\lambda }}\right) H_{\mu }Q_{\lambda }^{-1}%
\right] ^{-1}\left( I+e^{2Q_{\lambda }}\right) ^{-1},%
\end{array}%
\right.
\end{equation*}%
so%
\begin{eqnarray*}
\left \Vert \Lambda _{\lambda ,\mu }^{-1}\right \Vert _{\mathcal{L}(X)} &\leqslant &M\left \Vert Q_{\lambda }^{-1}\right \Vert _{\mathcal{L}%
(X)}\left \Vert I-\left( I+e^{2Q_{\lambda }}\right) ^{-1}\left(
I-e^{2Q_{\lambda }}\right) H_{\mu }Q_{\lambda }^{-1}\right \Vert _{\mathcal{L%
}(X)}\left \Vert I+e^{2Q_{\lambda }}\right \Vert _{\mathcal{L}(X)} \\
&\leqslant &\dfrac{M}{\left( 1+\left \vert \lambda \right \vert \right)
^{1/2}},
\end{eqnarray*}%
and $\left \Vert Q_{\lambda }\Lambda _{\lambda ,\mu }^{-1}\right \Vert _{%
\mathcal{L}(X)}\leqslant M$. Moreover, from (\ref{est2}),{\Large \ }$Q_{\lambda }-H_{\mu }=\left(
I-H_{\mu }Q_{\lambda }^{-1}\right) Q_{\lambda }$ is boundedly invertible with%
\begin{equation*}
\left( Q_{\lambda }-H_{\mu }\right) ^{-1}=Q_{\lambda }^{-1}\left( I-H_{\mu
}Q_{\lambda }^{-1}\right) ^{-1},
\end{equation*}%
so%
\begin{equation*}
\left \{ 
\begin{array}{l}
\left \Vert \left( Q_{\lambda }-H_{\mu }\right) ^{-1}\right \Vert _{\mathcal{%
L}(X)}\leqslant \left \Vert Q_{\lambda }^{-1}\right \Vert _{\mathcal{L}%
(X)}\left \Vert \left( I-H_{\mu }Q_{\lambda }^{-1}\right) ^{-1}\right \Vert
_{\mathcal{L}(X)}\leqslant \dfrac{M}{\left( 1+\left \vert \lambda \right
\vert \right) ^{1/2}} \\ 
\left \Vert Q_{\lambda }\left( Q_{\lambda }-H_{\mu }\right) ^{-1}\right
\Vert _{\mathcal{L}(X)}=\left \Vert \left( I-H_{\mu }Q_{\lambda
}^{-1}\right) ^{-1}\right \Vert _{\mathcal{L}(X)}\leqslant M,%
\end{array}%
\right.
\end{equation*}%
and%
\begin{eqnarray*}
\left \Vert \left( Q_{\lambda }+H_{\mu }\right) \left( Q_{\lambda }-H_{\mu
}\right) ^{-1}\right \Vert _{\mathcal{L}(X)} &=&\left \Vert \left(
-Q_{\lambda }+H_{\mu }+2Q_{\lambda }\right) \left( Q_{\lambda }-H_{\mu
}\right) ^{-1}\right \Vert _{\mathcal{L}(X)} \\
&\leqslant &1+2\left \Vert Q_{\lambda }\left( Q_{\lambda }-H_{\mu }\right)
^{-1}\right \Vert _{\mathcal{L}(X)} \\
&\leqslant &M.
\end{eqnarray*}%
Finally 
\begin{eqnarray*}
\left \Vert e^{2Q_{\lambda }}\left( Q_{\lambda }+H_{\mu }\right) \left(
Q_{\lambda }-H_{\mu }\right) ^{-1}\right \Vert _{\mathcal{L}(X)} &\leqslant &\left \Vert e^{2Q_{\lambda }}\right \Vert _{\mathcal{L}(X)}\left
\Vert \left( Q_{\lambda }+H_{\mu }\right) \left( Q_{\lambda }-H_{\mu
}\right) ^{-1}\right \Vert _{\mathcal{L}(X)} \\
&\leqslant &M\left \Vert e^{2Q_{\lambda }}\right \Vert _{\mathcal{L}(X)},
\end{eqnarray*}%
and due to (\ref{estim exp Q}), we can eventually increase $\rho _{0}$, for $\left( \lambda ,\mu \right) \in \Pi _{\varphi _{0},\rho_{0}}$, which implies that $\left \vert \lambda \right \vert \geqslant \rho _{0}$, in order to have $M\left \Vert e^{2Q_{\lambda }}\right \Vert _{\mathcal{L}(X)}\leqslant 1/2.$
\end{proof}
The proof of the following Lemma will use (\ref{Q-H stability bis}), wich is equivalent to (\ref{Q-H stability}) from Remark~\ref{Rem equivalence hyp}, to study, for a given 
$\left( \lambda ,\mu \right) \in \Pi _{\varphi _{0},\rho _{0}}$, operator $Q_{\lambda }^{2}\left( Q_{\lambda }-H_{\mu }\right)
^{-1}Q_{\lambda }^{-1}$.
\begin{lemma}\label{Stab Op}
Assume (\ref{NewH2}), (\ref{D(Q)Principal}), (\ref{NewH3bis})
and (\ref{Q-H stability}). Fix $\left( \lambda _{1},\mu _{1}\right) \in \Pi
_{\varphi _{0},\rho _{0}}$. Then, there exists $M > 0$ such that for
any $\left( \lambda ,\mu \right) \in \Pi _{\varphi _{0},\rho _{0}}$, we have

\begin{enumerate}
\item $\left( Q_{\lambda }-H_{\mu }\right) ^{-1}Q_{\lambda }^{-1}=\left(
Q_{\lambda _{1}}-H_{\mu _{1}}\right) ^{-1}Q^{-1}P_{\lambda ,\mu }$, where $%
P_{\lambda ,\mu }\in \mathcal{L}(X)$ with 
$$\left \Vert P_{\lambda ,\mu}\right \Vert _{\mathcal{L}(X)}\leqslant M.$$

\item $Q_{\lambda }^{2}\left( Q_{\lambda }-H_{\mu }\right) ^{-1}Q_{\lambda
}^{-1}\in \mathcal{L}(X)$ with 
\begin{equation}
\left \Vert Q_{\lambda }^{2}\left( Q_{\lambda }-H_{\mu }\right)
^{-1}Q_{\lambda }^{-1}\right \Vert _{\mathcal{L}(X)}\leqslant M.
\label{est Q2Q-HQmoins1}
\end{equation}

\item There exists $W_{\lambda ,\mu }\in \mathcal{L}(X)$ such that%
\begin{equation}
\Lambda _{\lambda ,\mu }^{-1}=\left( Q_{\lambda }-H_{\mu }\right)
^{-1}\left( I+e^{2Q_{\lambda }}W_{\lambda ,\mu }\right) ,
\label{Rel Lambda Q-H}
\end{equation}%
with
\begin{equation}
\left \Vert W_{\lambda ,\mu }\right \Vert _{\mathcal{L}(X)}\leqslant M\text{
\ and \ }\left \Vert Q_{\lambda }^{2}\Lambda _{\lambda ,\mu }^{-1}Q_{\lambda
}^{-1}\right \Vert _{\mathcal{L}(X)}\leqslant M.  \label{est Lambda}
\end{equation}
\end{enumerate}
\end{lemma}

\begin{proof}
Let $\left( \lambda ,\mu \right) \in \Pi _{\varphi _{0},\rho _{0}}$.

\begin{enumerate}
\item We have%
\begin{eqnarray*}
\left( Q_{\lambda }-H_{\mu }\right) ^{-1}Q_{\lambda }^{-1} &=&\left( Q_{\lambda _{1}}-H_{\mu _{1}}\right) ^{-1}\left( Q_{\lambda
_{1}}-H_{\mu _{1}}\right) \left( Q_{\lambda }-H_{\mu }\right)
^{-1}Q_{\lambda }^{-1} \\ \ecart
&=&\left( Q_{\lambda _{1}}-H_{\mu _{1}}\right) ^{-1}\left[ Q_{\lambda
}-H_{\mu }+\left( \mu -\mu _{1}\right) I\right]\left( Q_{\lambda }-H_{\mu }\right) ^{-1}Q_{\lambda }^{-1} \vspace{0.1cm}\\
&&+\left( Q_{\lambda _{1}}-H_{\mu _{1}}\right) ^{-1}\left( Q_{\lambda _{1}}-Q_{\lambda
}\right) \left( Q_{\lambda }-H_{\mu }\right) ^{-1}Q_{\lambda }^{-1}
\\ \ecart
&=&\left( Q_{\lambda _{1}}-H_{\mu _{1}}\right) ^{-1}\left[ Q_{\lambda
}^{-1}+\left( \mu -\mu _{1}\right) \left( Q_{\lambda }-H_{\mu }\right)
^{-1}Q_{\lambda }^{-1}\right] \vspace{0.1cm}\\
&& + \left( Q_{\lambda _{1}}-H_{\mu _{1}}\right) ^{-1}\left( Q_{\lambda _{1}}-Q_{\lambda }\right) \left( Q_{\lambda}-H_{\mu }\right) ^{-1}Q_{\lambda }^{-1} ,
\end{eqnarray*}%
but, $Q_{\lambda _{1}}-Q_{\lambda }=\left( Q-Q_{\lambda }\right) -\left(
Q-Q_{\lambda _{1}}\right) $ and from Lemma \ref{Compar racine}, there exists 
$T_{\lambda ,\lambda _{1}}\in \mathcal{L}(X)$ such that $Q_{\lambda
_{1}}=Q_{\lambda }+T_{_{\lambda,\lambda_1 }}$ 
\begin{equation}\label{est T}
\left \Vert T_{\lambda ,\lambda _{1}}\right \Vert _{\mathcal{L}(X)}\leqslant
M\left( 1+\sqrt{\left \vert \lambda \right \vert }\right) \quad \text{and} \quad Q^{-1}T_{_{\lambda ,\lambda _{1}}}=T_{_{\lambda ,\lambda _{1}}}Q^{-1},
\end{equation}%
so%
\begin{equation*}
\left( Q_{\lambda }-H_{\mu }\right) ^{-1}Q_{\lambda }^{-1}=\left( Q_{\lambda
_{1}}-H_{\mu _{1}}\right) ^{-1}Q^{-1}P_{\lambda ,\mu },
\end{equation*}%
where $P_{\lambda ,\mu }\in \mathcal{L}(X)$ is defined by%
\begin{eqnarray*}
P_{\lambda ,\mu } &=&QQ_{\lambda }^{-1}\left[ I+\left( \mu _{1}-\mu \right) Q_{\lambda }\left(
Q_{\lambda }-H_{\mu }\right) ^{-1}Q_{\lambda }^{-1}+T_{_{\lambda ,\lambda
_{1}}}Q_{\lambda }\left( Q_{\lambda }-H_{\mu }\right) ^{-1}Q_{\lambda }^{-1}%
\right].
\end{eqnarray*}%
Moreover, using (\ref{EstimationQmoins1}), (\ref{est Q invQLambda}), (\ref%
{est T}) and (\ref{est Q-H})%
\begin{eqnarray*}
\left \Vert P_{\lambda ,\mu }\right \Vert _{\mathcal{L}(X)} &\leqslant &M\left[ 1+\left( \left \vert \mu -\mu _{1}\right \vert \right)\left \Vert Q_{\lambda }\left( Q_{\lambda }-H_{\mu }\right) ^{-1}\right\Vert _{\mathcal{L}(X)}\left \Vert Q_{\lambda }^{-1}\right \Vert _{\mathcal{L}(X)}\right] \\
&& + M \left[\left \Vert T_{\lambda ,\lambda _{1}}\right \Vert _{\mathcal{L}%
(X)}\left \Vert Q_{\lambda }\left( Q_{\lambda }-H_{\mu }\right) ^{-1}\right
\Vert _{\mathcal{L}(X)}\left \Vert Q_{\lambda }^{-1}\right \Vert _{\mathcal{L%
}(X)}\right] \\
&\leqslant &M\left[ 1+\dfrac{\left \vert \mu -\mu _{1}\right \vert }{\left(
1+\left \vert \lambda \right \vert \right) ^{1/2}}+\dfrac{1+\sqrt{\left
\vert \lambda \right \vert }}{\left( 1+\left \vert \lambda \right \vert
\right) ^{1/2}}\right];
\end{eqnarray*}%
but since $\left( \lambda ,\mu \right) \in \Pi _{\varphi _{0},\rho _{0}}$ we
have $1+\left \vert \lambda \right \vert \geqslant 1+\rho _{0}\left \vert \mu
\right \vert ^{1/\varepsilon }\geqslant 1+\rho _{0}\left \vert \mu \right
\vert ^{2},$ thus we obtain $\left \Vert P_{\lambda ,\mu }\right \Vert _{\mathcal{L}(X)}\leqslant M.$

\item Since $Q_{\lambda }^{2}$ is closed then, from (\ref{Q-H stability bis}) and
the closed graph theorem, we obtain that $Q_{\lambda }^{2}\left( Q_{\lambda
}-H_{\mu }\right) ^{-1}Q_{\lambda }^{-1}\in \mathcal{L}(X)$. Moreover we have
\begin{eqnarray*}
\left \Vert Q_{\lambda }^{2}\left( Q_{\lambda }-H_{\mu }\right)
^{-1}Q_{\lambda }^{-1}\right \Vert _{\mathcal{L}(X)} &=&\left \Vert \left( -A+\lambda I\right) \left( Q_{\lambda }-H_{\mu
}\right) ^{-1}Q_{\lambda }^{-1}\right \Vert _{\mathcal{L}(X)} \\
&\leqslant &\left \Vert A\left( Q_{\lambda }-H_{\mu }\right) ^{-1}Q_{\lambda
}^{-1}\right \Vert _{\mathcal{L}(X)} \\
&& +\left \vert \lambda \right \vert \left
\Vert \left( Q_{\lambda }-H_{\mu }\right) ^{-1}\right \Vert _{\mathcal{L}%
(X)}\left \Vert Q_{\lambda }^{-1}\right \Vert _{\mathcal{L}(X)} \\
&\leqslant &M.
\end{eqnarray*}%
The last inequality is obtained, from statement 1, which gives%
\begin{eqnarray*}
\left \Vert -A\left( Q_{\lambda }-H_{\mu }\right) ^{-1}Q_{\lambda
}^{-1}\right \Vert _{\mathcal{L}(X)} &=&\left \Vert A\left( Q_{\lambda
_{1}}-H_{\mu _{1}}\right) ^{-1}Q^{-1}P_{\lambda ,\mu }\right \Vert _{%
\mathcal{L}(X)} \\
&\leqslant &\left \Vert -A\left( Q_{\lambda _{1}}-H_{\mu _{1}}\right)
^{-1}Q^{-1}\right \Vert _{\mathcal{L}(X)}\left \Vert P_{\lambda ,\mu }\right
\Vert _{\mathcal{L}(X)} \leqslant M,
\end{eqnarray*}%
and, from (\ref{est Q-H}), (\ref{EstimationQmoins1}), which furnishes%
\begin{equation*}
\left \vert \lambda \right \vert \left \Vert \left( Q_{\lambda }-H_{\mu
}\right) ^{-1}\right \Vert _{\mathcal{L}(X)}\left \Vert Q_{\lambda
}^{-1}\right \Vert _{\mathcal{L}(X)}\leqslant M.
\end{equation*}

\item We set $R_{\lambda ,\mu }=\left( Q_{\lambda }+H_{\mu }\right) \left(
Q_{\lambda }-H_{\mu }\right) ^{-1}\in \mathcal{L}(X)$ and write%
\begin{equation*}
\Lambda _{\lambda ,\mu }=\left( Q_{\lambda }-H_{\mu }\right) +e^{2Q_{\lambda
}}\left( Q_{\lambda }+H_{\mu }\right) =\left( I+e^{2Q_{\lambda }}R_{\lambda
,\mu }\right) \left( Q_{\lambda }-H_{\mu }\right) ;
\end{equation*}%
but $\Lambda _{\lambda ,\mu },\left( Q_{\lambda }-H_{\mu }\right) $ are
boundedly invertible, so $I+e^{2Q_{\lambda }}R_{\lambda ,\mu }$ is boundedly
invertible with%
\begin{equation*}
\left( I+e^{2Q_{\lambda }}R_{\lambda ,\mu }\right) ^{-1}=I-e^{2Q_{\lambda
}}R_{\lambda ,\mu }\left( I+e^{2Q_{\lambda }}R_{\lambda ,\mu }\right) ^{-1}.
\end{equation*}%
Now setting $W_{\lambda ,\mu }=R_{\lambda ,\mu }\left( I+e^{2Q_{\lambda
}}R_{\lambda ,\mu }\right) ^{-1}\in \mathcal{L}(X)$ we have%
\begin{equation*}
\Lambda _{\lambda ,\mu }^{-1}=\left( Q_{\lambda }-H_{\mu }\right)
^{-1}\left( I+e^{2Q_{\lambda }}R_{\lambda ,\mu }\right) ^{-1}=\left(
Q_{\lambda }-H_{\mu }\right) ^{-1}\left( I-e^{2Q_{\lambda }}W_{\lambda ,\mu
}\right) ,
\end{equation*}%
and, due to (\ref{est Q-H}), we have%
\begin{eqnarray*}
\left \Vert W_{\lambda ,\mu }\right \Vert _{\mathcal{L}(X)} &\leqslant
&\left \Vert R_{\lambda ,\mu }\right \Vert _{\mathcal{L}(X)}\left \Vert
\left( I+e^{2Q_{\lambda }}R_{\lambda ,\mu }\right) ^{-1}\right \Vert _{%
\mathcal{L}(X)} \\
&\leqslant &\frac{\left \Vert R_{\lambda ,\mu }\right \Vert _{\mathcal{L}(X)}%
}{1-\left \Vert e^{2Q_{\lambda }}R_{\lambda ,\mu }\right \Vert _{\mathcal{L}%
(X)}} \leqslant M.
\end{eqnarray*}

Finally%
\begin{eqnarray*}
\left \Vert Q_{\lambda }^{2}\Lambda _{\lambda ,\mu }^{-1}Q_{\lambda}^{-1}\right \Vert _{\mathcal{L}(X)} &=&\left \Vert Q_{\lambda }^{2}\left( Q_{\lambda }-H_{\mu }\right)
^{-1}\left( I-e^{2Q_{\lambda }}W_{\lambda ,\mu }\right) Q_{\lambda
}^{-1}\right \Vert _{\mathcal{L}(X)} \\
&\leqslant &\left \Vert Q_{\lambda }^{2}\left( Q_{\lambda }-H_{\mu }\right)
^{-1}Q_{\lambda }^{-1}\right \Vert _{\mathcal{L}(X)} \\
&&+\left \Vert Q_{\lambda }^{2}\left( Q_{\lambda }-H_{\mu }\right)
^{-1}Q_{\lambda }^{-1}\right \Vert _{\mathcal{L}(X)}\left \Vert
e^{2Q_{\lambda }}\right \Vert _{\mathcal{L}(X)}\left \Vert W_{\lambda ,\mu
}\right \Vert _{\mathcal{L}(X)} \\
&\leqslant &M.
\end{eqnarray*}
\end{enumerate}
\end{proof}

\subsection{Proofs to Theorem \ref{Main1Bis} and Theorem \ref{Main2bis}}

Let $\rho _{0}$ fixed as in Lemma \ref{Grand Lamba inversible 2}.

\subsubsection{Proof to Theorem \ref{Main1Bis}}

As in the proof to Theorem \ref{Main1}, we want to apply Proposition \ref%
{Prop without lambda}, with $A,H,Q,\Lambda $ replaced by $A-\lambda I,H+\mu
I,Q_{\lambda },\Lambda _{\lambda ,\mu }$. Assumptions $(H_{1})\sim(H_{4})$
are easily deduced from (\ref{NewH1})$\sim $(\ref{MoinsABip}), (\ref%
{D(Q)Principal}) and Lemma \ref{Grand Lamba inversible 2}. To obtain $%
(H_{5}) $, it is enough to prove $(H'_5)$ given by (\ref{Equiv Interp}). So, for $\xi \in \left( D\left( Q_{\lambda }\right)
,X\right) _{1/p,p}=\left( D\left( Q\right) ,X\right) _{1/p,p}$, we just have
to show that 
\begin{equation*}
\eta =Q_{\lambda }\left( Q_{\lambda }-H_{\mu }\right) ^{-1}\xi \in \left(
D\left( Q\right) ,X\right) _{1/p,p},
\end{equation*}%
but, from Lemma \ref{estim racine L}, statement 5. we have $Q_{\lambda
}=Q+\lambda \left( Q_{\lambda }+Q\right) ^{-1}$, thus%
\begin{equation*}
\left( Q_{\lambda }-H_{\mu }\right) Q_{\lambda }^{-1}=\left( Q-H\right)
Q_{\lambda }^{-1}+\lambda \left( Q_{\lambda }+Q\right) ^{-1}Q_{\lambda
}^{-1}-\mu Q_{\lambda }^{-1},
\end{equation*}%
so%
\begin{equation*}
\xi =\left( Q_{\lambda }-H_{\mu }\right) Q_{\lambda }^{-1}\eta =\left(
Q-H\right) Q_{\lambda }^{-1}\eta +\lambda \left( Q_{\lambda }+Q\right)
^{-1}Q_{\lambda }^{-1}\eta -\mu Q_{\lambda }^{-1}\eta ,
\end{equation*}%
and%
\begin{equation*}
\left( Q-H\right) Q_{\lambda }^{-1}\eta =\xi -\lambda \left( Q_{\lambda
}+Q\right) ^{-1}Q_{\lambda }^{-1}\eta +\mu Q_{\lambda }^{-1}\eta \in \left(
D\left( Q\right) ,X\right) _{1/p,p},
\end{equation*}%
which means that $Q_{\lambda }^{-1}\eta \in \left( Q-H\right) ^{-1}\left(
\left( D\left( Q\right) ,X\right) _{1/p,p}\right) $ and, from (\ref{inv Lamb
and interp}), we get%
\begin{equation*}
Q_{\lambda }^{-1}\eta \in Q^{-1}\left( \left( D\left( Q\right) ,X\right)
_{1/p,p}\right) ,
\end{equation*}%
and then $\eta \in \left( D\left( Q\right) ,X\right) _{1/p,p}$.

Here the condition $\left( Q_{\lambda }-H_{\mu }\right) ^{-1}d_{0}\in \left(
D\left( A\right) ,X\right) _{\frac{1}{2p},p}$ which is, from  Remark \ref{simplif}, equivalent to $\Lambda _{\lambda ,\mu }^{-1}d_{0}\in \left( D\left( A\right) ,X\right) _{\frac{1}{2p},p}$, appears naturally, since we have
not, as in Theorem \ref{Main1}, $\Lambda _{\lambda ,\mu }^{-1}\left(
X\right) \subset D\left( Q^{2}\right) .$

\subsubsection{Proof to Theorem \ref{Main2bis}}

Assume (\ref{NewH1})$\sim $(\ref{MoinsABip}) and (\ref{D(Q)Principal}), (\ref{NewH3bis}), (\ref{Q-H stability}). Let $\left( \lambda ,\mu \right) \in \Pi
_{\varphi _{0},\rho _{0}}$. From Lemma~\ref{Stab Op}, statement 2., we
have%
\begin{equation*}
\forall \xi \in D(Q_{\lambda }), \quad Q_{\lambda }\left( Q_{\lambda }-H_{\mu
}\right) ^{-1}\xi \in D(Q_{\lambda }),
\end{equation*}
then (\ref{inv Lamb and interp}) is satisfied, see Remark \ref{discuss H5}
statement 5, and we can apply Theorem \ref{Main1Bis}.

Again we adapt the proof of Theorem \ref{Main2} and write $%
u=k_{1}+k_{2}+k_{3}-h_{2}+h_{3}+h_{4}$ with 
\begin{equation*}
\left \{ 
\begin{array}{l}
k_{1}\left( x\right) =S_{\lambda }\left( x\right) \Lambda _{\lambda ,\mu
}^{-1}Q_{\lambda }^{-1}\left[ -Q_{\lambda }e^{2Q_{\lambda}}d_{0}+2Q_{\lambda }^{2}e^{Q_{\lambda }}u_{1}-2Q_{\lambda }e^{Q_{\lambda}}I_{\lambda ,f}\left( 1\right) \right] \vspace{0.1cm}\\ 
k_{2}\left( x\right) =2S_{\lambda }\left( x\right) \Lambda _{\lambda ,\mu
}^{-1}Q_{\lambda }J_{\lambda ,f}\left( 0\right)\vspace{0.1cm} \\ 
k_{3}\left( x\right) =S_{\lambda }\left( x\right) \Lambda _{\lambda ,\mu
}^{-1}d_{0} \vspace{0.1cm}\\ 
h_{2}\left( x\right) =S_{\lambda }\left( x\right) J_{\lambda ,f}\left(
0\right) +S_{\lambda }\left( 1-x\right) I_{\lambda ,f}\left( 1\right)\vspace{0.1cm} \\ 
h_{3}\left( x\right) =S_{\lambda }\left( 1-x\right) u_{1}\vspace{0.1cm} \\ 
h_{4}\left( x\right) =I_{\lambda ,f}\left( x\right) +J_{\lambda ,f}\left(
x\right) .
\end{array}
\right.
\end{equation*}

\begin{description}
\item[Estimate of $Q_{\protect\lambda }^{2}k_{1}.$] Due to (\ref{est Lambda}%
), we have for $\xi \in X$ and $x\in \lbrack 0,1]$%
\begin{eqnarray*}
\left \Vert Q_{\lambda }^{2}S_{\lambda }\left( x\right) \Lambda _{\lambda
,\mu }^{-1}Q_{\lambda }^{-1}\xi \right \Vert \\ \ecart
&\hspace{-4cm}=&\hspace{-2cm}\left \Vert \left( I-e^{2Q_{\lambda }}\right) ^{-1}\left( I-e^{2\left(
1-x\right) Q_{\lambda }}\right) e^{xQ_{\lambda }}Q_{\lambda }^{2}\Lambda
_{\lambda ,\mu }^{-1}Q_{\lambda }^{-1}\xi \right \Vert \\
&\hspace{-4cm}\leqslant &\hspace{-2cm}\left \Vert \left( I-e^{2Q_{\lambda }}\right) ^{-1}\left(
I-e^{2\left( 1-x\right) Q_{\lambda }}\right) \right \Vert _{\mathcal{L}%
(X)}\left \Vert e^{xQ_{\lambda }}\right \Vert _{\mathcal{L}(X)}\left \Vert
Q_{\lambda }^{2}\Lambda _{\lambda ,\mu }^{-1}Q_{\lambda }^{-1}\right \Vert _{%
\mathcal{L}(X)}\left \Vert \xi \right \Vert \\
&\hspace{-4cm}\leqslant &\hspace{-2cm} M\left \Vert \xi \right \Vert ,
\end{eqnarray*}%
then%
\begin{equation*}
\left \Vert Q_{\lambda }^{2}k_{1}\left( x\right) \right \Vert \leqslant
M\left( \left \Vert d_{0}\right \Vert +\left \Vert u_{1}\right \Vert +\left
\Vert f\right \Vert _{L^{p}(0,1;X)}\right) ,
\end{equation*}%
and%
\begin{equation*}
\left \Vert Q_{\lambda }^{2}k_{1}\right \Vert _{L^{p}(0,1;X)}\leqslant
M\left( \left \Vert d_{0}\right \Vert +\left \Vert u_{1}\right \Vert +\left
\Vert f\right \Vert _{L^{p}(0,1;X)}\right) .
\end{equation*}

\item[Estimate of $Q_{\protect\lambda }^{2}k_{2}.$] We write, for $x\in (0,1]$
\begin{eqnarray*}
Q_{\lambda }^{2}e^{xQ_{\lambda }}\Lambda _{\lambda ,\mu }^{-1}Q_{\lambda
}J_{\lambda ,f}\left( 0\right) &=&Q_{\lambda }^{2}e^{xQ_{\lambda }}\Lambda
_{\lambda ,\mu }^{-1}\int_{0}^{x}e^{sQ_{\lambda }}f(s)ds \\
&&+Q_{\lambda }^{2}e^{xQ_{\lambda }}\Lambda _{\lambda ,\mu
}^{-1}\int_{x}^{1}e^{sQ_{\lambda }}f(s)ds \\
&=&Q_{\lambda }\int_{0}^{x}e^{\left( x-s\right) Q_{\lambda }}e^{sQ_{\lambda
}}Q_{\lambda }\Lambda _{\lambda ,\mu }^{-1}e^{sQ_{\lambda }}f(s)ds \\
&&+e^{xQ_{\lambda }}Q_{\lambda }^{2}\Lambda _{\lambda ,\mu }^{-1}Q_{\lambda
}^{-1}e^{xQ_{\lambda }}Q_{\lambda }\int_{x}^{1}e^{\left( s-x\right)
Q_{\lambda }}f(s)ds \\
&=&Q_{\lambda }\int_{0}^{x}e^{\left( x-s\right) Q_{\lambda }}F_{\lambda
}\left( s\right) ds \\
&&+e^{xQ_{\lambda }}Q_{\lambda }^{2}\Lambda _{\lambda ,\mu }^{-1}Q_{\lambda
}^{-1}e^{xQ_{\lambda }}Q_{\lambda }\int_{x}^{1}e^{\left( s-x\right)
Q_{\lambda }}f(s)ds,
\end{eqnarray*}%
where $F_{\lambda }\left( s\right) =e^{sQ_{\lambda }}Q_{\lambda }\Lambda
_{\lambda ,\mu }^{-1}e^{sQ_{\lambda }}f(s)$. So%
\begin{eqnarray*}
\left \Vert Q_{\lambda }^{2}k_{2}\right \Vert _{L^{p}(0,1;X)} &\leqslant
&\dis M\left \Vert Q_{\lambda }\int_{0}^{\cdot }e^{\left( \cdot -s\right)
Q_{\lambda }}F_{\lambda }\left( s\right) ds\right \Vert _{_{L^{p}(0,1;X)}} \\
&&\dis +M\left \Vert Q_{\lambda }^{2}\Lambda _{\lambda ,\mu }^{-1}Q_{\lambda}^{-1}\right \Vert _{\mathcal{L}(X)}\left \Vert Q_{\lambda }\int_{\cdot
}^{1}e^{\left( s-\cdot \right) Q_{\lambda }}f\left( s\right) ds\right \Vert
_{_{L^{p}(0,1;X)}};
\end{eqnarray*}
but, from Lemma \ref{estim omega}, (\ref{estim exp xqlambda}), (\ref{est Q-H}%
), (\ref{est Q2Q-HQmoins1}) and (\ref{est Lambda}), we deduce%
\begin{equation*}
\left \{ 
\begin{array}{l}
\dis \left \Vert Q_{\lambda }\int_{0}^{\cdot }e^{\left( \cdot -s\right) Q_{\lambda }}F_{\lambda }\left( s\right) ds\right \Vert_{_{L^{p}(0,1;X)}}\leqslant M\left \Vert F_{\lambda }\right \Vert_{L^{p}(0,1;X)}\leqslant M\left \Vert f\right \Vert _{L^{p}(0,1;X)} \vspace{0.1cm}\\ 
\dis \left \Vert Q_{\lambda }^{2}\Lambda _{\lambda ,\mu }^{-1}Q_{\lambda}^{-1}\right \Vert_{\mathcal{L}(X)}\left \Vert Q_{\lambda }\int_{\cdot}^{1}e^{\left( s-\cdot \right) Q_{\lambda }}f\left( s\right) ds\right \Vert
\leqslant M\left \Vert f\right \Vert _{L^{p}(0,1;X)};
\end{array}\right.
\end{equation*}
therefore $\left \Vert Q_{\lambda }^{2}k_{2}\right \Vert _{L^{p}(0,1;X)}\leqslant
M\left \Vert f\right \Vert _{L^{p}(0,1;X)}.$

\item[Estimate of $Q_{\protect\lambda }^{2}k_{3}.$] Due to (\ref{Rel Lambda
Q-H}) we write $k_{3}=\widetilde{k_{3}}+\overline{k_{3}}$ with 
\begin{equation*}
\widetilde{k_{3}}\left( x\right) =S_{\lambda }\left( x\right) \left(
Q_{\lambda }-H_{\mu }\right) ^{-1}d_{0} \quad \text{and} \quad \overline{k_{3}}\left( x\right) =S_{\lambda }\left( x\right) \left(Q_{\lambda }-H_{\mu }\right) ^{-1}e^{2Q_{\lambda }}W_{\lambda ,\mu }d_{0}.
\end{equation*}%
Due to (\ref{Est Q2 SLambda}), we have for $x\in (0,1]$%
\begin{eqnarray*}
\left \Vert Q_{\lambda }^{2}\widetilde{k_{3}}\left( x\right) \right \Vert
&=&\left \Vert Q_{\lambda }^{2}\left( I-e^{2Q_{\lambda }}\right) ^{-1}\left(
I-e^{2\left( 1-x\right) Q_{\lambda }}\right) e^{xQ_{\lambda }}\left(
Q_{\lambda }-H_{\mu }\right) ^{-1}d_{0}\right \Vert \\
&\leqslant &M\left \Vert Q_{\lambda }^{2}e^{xQ_{\lambda }}\left( Q_{\lambda
}-H_{\mu }\right) ^{-1}d_{0}\right \Vert .
\end{eqnarray*}%
From Theorem 2.1 in \cite{Dore Yakubov}, since $\left( Q_{\lambda }-H_{\mu }\right) ^{-1}d_{0}\in (D(A),X)_{\frac{1}{2p},p},$ we get
\begin{eqnarray*}
\left \Vert Q_{\lambda }^{2}\widetilde{k_{3}}\right \Vert
_{_{L^{p}(0,1;X)}} &\leqslant &M\left( \left \Vert \left( Q_{\lambda }-H_{\mu }\right)
^{-1}d_{0}\right \Vert _{(D(A),X)_{\frac{1}{2p},p}}+\left \vert \lambda
\right \vert ^{1-\frac{1}{2p}}\left \Vert \left( Q_{\lambda }-H_{\mu
}\right) ^{-1}d_{0}\right \Vert \right) .
\end{eqnarray*}%
We have also, taking into account (\ref{estim exp xqlambda}), (\ref{est
Q2Q-HQmoins1}), (\ref{estim exp Q}) and (\ref{est Lambda}),
\begin{eqnarray*}
\left \Vert Q_{\lambda }^{2}\overline{k_{3}}\left( x\right) \right \Vert &=&\left \Vert Q_{\lambda }^{2}\left( I-e^{2Q_{\lambda }}\right) ^{-1}\left(
I-e^{2\left( 1-x\right) Q_{\lambda }}\right) e^{xQ_{\lambda }}\left(
Q_{\lambda }-H_{\mu }\right) ^{-1}e^{2Q_{\lambda }}W_{\lambda ,\mu
}d_{0}\right \Vert \\
&\leqslant &M\left \Vert e^{xQ_{\lambda }}Q_{\lambda }^{2}\left( Q_{\lambda
}-H_{\mu }\right) ^{-1}Q_{\lambda }^{-1}Q_{\lambda }e^{2Q_{\lambda
}}W_{\lambda ,\mu }d_{0}\right \Vert \\
&\leqslant &M\left \Vert e^{xQ_{\lambda }}\right \Vert \left \Vert
Q_{\lambda }^{2}\left( Q_{\lambda }-H_{\mu }\right) ^{-1}Q_{\lambda
}^{-1}\right \Vert \left \Vert Q_{\lambda }e^{2Q_{\lambda }}\right \Vert
\left \Vert W_{\lambda ,\mu }\right \Vert \left \Vert d_{0}\right \Vert \\
&\leqslant & M \left \Vert d_{0}\right \Vert .
\end{eqnarray*}%
Finally 
$$
\left \Vert Q_{\lambda }^{2}\widetilde{k_{3}}\left( x\right) \right \Vert
\leqslant M \left(\left \Vert \left( Q_{\lambda }-H_{\mu }\right)
^{-1}d_{0}\right \Vert _{(D(A),X)_{\frac{1}{2p},p}}\hspace{-0.31cm} + \left \vert \lambda \right \vert ^{1-\frac{1}{2p}}\left \Vert \left( Q_{\lambda }-H_{\mu }\right) ^{-1}d_{0}\right \Vert _{X}+\left \Vert d_{0}\right \Vert \right).
$$

\item[Estimates of $Q_{\protect\lambda }^{2}h_{2},Q_{\protect\lambda }^{2}h_{3},Q_{\protect\lambda }^{2}h_{4}\,.$] In these terms, $\Lambda_{\lambda ,\mu }^{-1}$ does not appear, so the estimates are the same as in Theorem \ref{Main2}.
\end{description}

\subsubsection{Proof to Theorem \ref{Gen second case}}

Assume that (\ref{NewH1})$\sim $(\ref{MoinsABip}) and (\ref{D(Q)Principal}),
(\ref{NewH3bis}), (\ref{Q-H stability}) hold. From Theorems \ref{Main1Bis} and Theorem \ref{Main2bis}%
, there exists $M\geqslant 0$ such that for any $\mu \in \mathbb{C}$

\begin{enumerate}
\item $\mathcal{L}_{A,H,\mu }$ is a closed linear operator on $Y.$

\item $S_{\varphi _{0}}\backslash B\left( 0,\rho _{\mu }\right) \subset \rho
\left( \mathcal{L}_{A,H,\mu }\right) $ where $\rho _{\mu }:=\max \left\{
\rho _{0},\rho _{0}\left\vert \mu \right\vert ^{1/\varepsilon }\right\} >0.$

\item $\forall \lambda \in S_{\varphi _{0}}\backslash B\left( 0,\rho _{\mu
}\right) ,~\forall f\in Y,~\forall x\in \lbrack 0,1]$%
\begin{eqnarray*}
\left( \mathcal{L}_{A,H,\mu }-\lambda I)^{-1}f\right) \left( x\right) &=&S_{\lambda }\left( x\right) \left[ 2\Lambda _{\lambda ,\mu }^{-1}Q_{\lambda }\left[ J_{\lambda ,f}\left( 0\right) -e^{Q_{\lambda}}I_{\lambda ,f}\left( 1\right) \right] -J_{\lambda ,f}\left( 0\right) \right] \vspace{0.1cm} \\
&&-S_{\lambda }\left( 1-x\right) I_{\lambda ,f}\left( 1\right) +I_{\lambda
,f}\left( x\right) +J_{\lambda ,f}\left( x\right) ,
\end{eqnarray*}%
where $S_{\lambda }\left( x\right) =\left( I-e^{2Q_{\lambda }}\right)
^{-1}\left( e^{xQ_{\lambda }}-e^{\left( 1-x\right) Q_{\lambda
}}e^{Q_{\lambda }}\right) .$

\item $\forall \lambda \in S_{\varphi _{0}}\backslash B\left( 0,\rho _{\mu
}\right) :\quad \left\Vert (\mathcal{L}_{A,H,\mu }-\lambda I)^{-1}\right\Vert _{%
\mathcal{L}\left( Y\right) }\leqslant \dfrac{M}{1+|\lambda |}.$
\end{enumerate}
The proof to Theorem \ref{Gen second case} is given by Statement $4$.

\section{Results for Dirichlet boundary conditions}

We can find, in \cite{Favini et al1} and \cite{Favin et al2}, the study of the following problem
\begin{equation}
\left \{ 
\begin{array}{l}
u^{\prime \prime }(x)+Au(x)=f(x),\quad x\in (0,1) \vspace{0.1cm} \\ 
u(0)=u_{0},\text{ }u(1)=u_{1}.%
\end{array}%
\right.  \label{Dirichlet Pb}
\end{equation}

A classical solution of this problem is a function $u\in W^{2,p}(0,1;X)\cap
L^{p}(0,1;D(A))$, satisfying (\ref{Dirichlet Pb}). 

\subsection{Proof to Theorem \ref{Main3}}

The authors obtain the following result (see Theorem 4, p. 200 in \cite{Favini
et al1} and Theorem 5 p.~173 (with $A=L=M$) in \cite{Favin et al2}).

\begin{proposition}[\cite{Favini et al1},\cite{Favin et al2}]
Let $f\in L^{p}(0,1;X)$ with $1<p<\mathbb{+}\infty $ and assume that (\ref{NewH1})$\sim $(\ref{MoinsABip}) are satisfied. Then the following assertions are
equivalent:

\begin{enumerate}
\item Problem (\ref{Dirichlet Pb}) admits a classical solution $u$.

\item $u_{1},u_{0}\in \left( D\left( A\right) ,X\right) _{\frac{1}{2p},p}$

Moreover in this case $u$ is unique and given by 
\begin{eqnarray}\label{repre Dirichlet}
u(x) &=&S\left( x\right) u_{0}+S\left( 1-x\right) u_{1}-S\left( x\right)
J\left( 0\right)   \vspace{0.1cm} \\
&&-S\left( 1-x\right) I\left( 1\right) +I\left( x\right) +J\left( x\right) ,%
\text{ \ }x\in (0,1).  \notag
\end{eqnarray}
\end{enumerate}
\end{proposition}

Note that $S(\cdot ),I\left( \cdot \right) ,J\left( \cdot \right) 
$ are precised in (\ref{def I et J sans parametre}) and (\ref{repre suite})
with $Q=-\sqrt{-A}$.

Now we are in a position to study, as in Sections \ref{Section 1er Cas} and %
\ref{Section 2eme Cas}, the corresponding spectral problem 
\begin{equation}
\left \{ 
\begin{array}{l}
u^{\prime \prime }(x)+Au(x)-\lambda u(x)=f(x),\text{ \ }x\in (0,1)  \vspace{0.1cm} \\ 
u(0)=u_{0},\text{ }u(1)=u_{1}.%
\end{array}%
\right.  \label{Dirichlet spectral}
\end{equation}

Applying the previous Proposition with $A$ replaced by $A-\lambda I$ we obtain Theorem \ref{Main3}.

\subsection{Proof to Theorem \ref{main Dirichlet}}

Let $\lambda \in S_{\varphi _{0}}$ and $f\in L^{p}\left( 0,1;X\right) $.
Taking into account (\ref{def I et J}) and (\ref{repre Dirichlet}) with $%
Q_{\lambda }$ replacing $Q$, for $x\in \left[ 0,1\right] $, we have
$$
u(x) = S_{\lambda }\left( x\right) u_{0}+S_{\lambda }\left( 1-x\right)
u_{1}-S\left( x\right) J_{\lambda ,f}\left( 0\right) -S\left( 1-x\right)
I_{\lambda ,f} + I_{\lambda ,f}\left( x\right) +J_{\lambda ,f}\left( x\right) .
$$
So we can write $u=-h_{2}+g_{3}+h_{3}+h_{4}$ with 
\begin{equation*}
\left \{ 
\begin{array}{l}
h_{2}\left( x\right) =S_{\lambda }\left( x\right) J_{\lambda ,f}\left(
0\right) +S_{\lambda }\left( 1-x\right) I_{\lambda ,f}\left( 1\right) \vspace{0.1cm} \\ 
g_{3}\left( x\right) =S_{\lambda }\left( x\right) u_{0} \vspace{0.1cm} \\ 
h_{3}\left( x\right) =S_{\lambda }\left( 1-x\right) u_{1} \vspace{0.1cm} \\ 
h_{4}\left( x\right) =I_{\lambda ,f}\left( x\right) +J_{\lambda ,f}\left(x\right).
\end{array}\right.
\end{equation*}%
As in the proof to Theorem \ref{Main2} we get
\begin{equation*}
\left \Vert Q_{\lambda }^{2}h_{2}\right \Vert _{L^{p}(0,1;X)}\leqslant
M\left \Vert f\right \Vert _{L^{p}(0,1;X)},\text{ \ }\left \Vert Q_{\lambda
}^{2}h_{4}\right \Vert _{L^{p}(0,1;X)}\leqslant M\left \Vert f\right \Vert
_{_{L^{p}(0,1;X)}},
\end{equation*}%
and also%
\begin{equation*}
\left \Vert Q_{\lambda }^{2}h_{3}\right \Vert _{_{L^{p}(0,1;X)}}\leqslant
M\left( \left \Vert u_{1}\right \Vert _{(D(A),X)_{\frac{1}{2p},p}}+\left
\vert \lambda \right \vert ^{1-\frac{1}{2p}}\left \Vert u_{1}\right \Vert
_{X}\right).
\end{equation*}
Moreover, $Q_{\lambda }^{2}g_{3}$ is treated like $Q_{\lambda }^{2}h_{3}$, so 
\begin{equation*}
\left \Vert Q_{\lambda }^{2}g_{3}\right \Vert _{_{L^{p}(0,1;X)}}\leqslant
M\left( \left \Vert u_{0}\right \Vert _{(D(A),X)_{\frac{1}{2p},p}}+\left
\vert \lambda \right \vert ^{1-\frac{1}{2p}}\left \Vert u_{0}\right \Vert
_{X}\right) .
\end{equation*}%
We finish as in the proof of Theorem \ref{Main2}.

\subsection{Proof to Theorem \ref{Gen Dirichlet}}

Assume $\eqref{NewH1}\sim \eqref{MoinsABip}$. From Theorems \ref{Main3} and Theorem \ref{main Dirichlet}, there exists $M\geqslant 0$ such that

\begin{enumerate}
\item $\mathcal{L}_{A}$ is a closed linear operator on $Y$ and $S_{\varphi _{0}}\subset \rho \left( \mathcal{L}_{A}\right) .$

\item $\forall \lambda \in S_{\varphi _{0}}$
$$\left( \mathcal{L}_{A}-\lambda I)^{-1}f\right) \left( x\right) = -S_{\lambda }\left( x\right) J_{\lambda ,f}\left( 0\right) -S_{\lambda}\left( 1-x\right) I_{\lambda ,f}\left( 1\right) + I_{\lambda ,f}\left( x\right) +J_{\lambda ,f}\left( x\right) ,$$

where $S_{\lambda }\left( x\right) =\left( I-e^{2Q_{\lambda }}\right)
^{-1}\left( e^{xQ_{\lambda }}-e^{\left( 1-x\right) Q_{\lambda
}}e^{Q_{\lambda }}\right) \in \mathcal{L}\left( X\right) .$

\item $\forall \lambda \in S_{\varphi _{0}}:\quad \left\Vert (\mathcal{L}%
_{A}-\lambda I)^{-1}\right\Vert _{\mathcal{L}\left( Y\right) }\leqslant 
\dfrac{M}{1+|\lambda |}.$
\end{enumerate}
The proof to Theorem \ref{Gen Dirichlet} is given by Statement $4$.

\section{Applications}

\subsection{A model example for the first case\label{applic first case}}

In view to illustrate the results obtained in this work, we will consider
the concrete problem of the heat equation in the square domain $\Omega
=(0,1)\times (0,1)$ with a dynamical-Wentzell condition in one of its
lateral boundaries

\begin{equation*}
(P)\text{ }\left\{ 
\begin{array}{llll}
\dfrac{\partial u}{\partial t}(t,x,y) &=& \Delta _{x,y} u(t,x,y), &(t,x,y)\in (0,+\infty )\times \Omega\medskip \\ 
\dfrac{\partial u}{\partial t}(t,0,y)&=&\dfrac{\partial u}{\partial x}(t,0,y)+\dfrac{\partial ^{2}u}{\partial y^{2}}(t,0,y), & (t,0,y)\in
(0,+\infty )\times \Gamma _{0}\medskip \medskip \\ 
u(t,1,y)&=&0, & (t,1,y)\in (0,+\infty )\times \Gamma _{1}\medskip \\ 
u(t,x,0)&=&u(t,x,1)=0, & x\in (0,1)\medskip \\ 
u(0,x,y)&=&u_{0}(x,y) &(x,y)\in (0,1)\times (0,1),%
\end{array}%
\right.
\end{equation*}%
where 
\begin{equation*}
\left\{ 
\begin{array}{ll}
\Gamma _{0}=\left\{ 0\right\} \times (0,1), & \Gamma _{1}=\left\{ 1\right\}
\times (0,1), \vspace{0.1cm}\\ 
\gamma _{0}=(0,1)\times \left\{ 0\right\} , & \gamma _{1}=(0,1)\times
\left\{ 1\right\} .%
\end{array}%
\right.
\end{equation*}%
Here $\dfrac{\partial ^{2}}{\partial y^{2}}$ is the Laplace-Beltrami
operator on $\Gamma _{0}$. Physically, $-\dfrac{\partial u}{\partial x}$ and 
$\dfrac{\partial u}{\partial x}$ represent the interaction between the
domain $\Omega $ and the lateral boundaries while $\dfrac{\partial ^{2}u}{%
\partial y^{2}}$ is the boundary diffusion.

Set $\mathcal{E}=L^{p}(\Omega )\times L^{p}(\Gamma _{0})$; this Banach space
is well defined and endowed with its natural norm. Define operator $\mathcal{%
P}$ by
$$\left\{\begin{array}{ccl}
D(\mathcal{P})&=&\left\{ w=(u,v_{0}):u,\Delta_{x,y} u\in L^{p}(\Omega ),v_{0}\in
W^{2,p}(\Gamma _{0}),u_{|\Gamma _{0}}=v_{0},\right. \\
&&\left. \hspace{1.5cm}\left( \Delta_{x,y} u\right) _{|\Gamma
_{0}}=\left( \dfrac{\partial u}{\partial x}\right) _{|\Gamma _{0}}+\dfrac{\partial
^{2}v_{0}}{\partial y^{2}}\text{~~and~~}u_{|\gamma _{0}\cup \gamma _{1}\cup
\Gamma _{1}}=0\right\}, \\ \\

\mathcal{P}w & = &\left( \Delta_{x,y} u,\left( \dfrac{\partial u}{\partial x}\right) _{|\Gamma
_{0}}+\dfrac{\partial ^{2}v_{0}}{\partial y^{2}}\right), \quad \text{ for } w=(u,v_{0})\in D(\mathcal{P}).
\end{array}\right.$$
The boundary conditions are defined in $L^{p}(\Gamma _{0})$ and $\mathcal{P}w\in \mathcal{E}$. On the other hand it is not difficult to see that this operator is closed in $\mathcal{E}$. When we integrate the time variable $t$, the following Cauchy problem%
\begin{equation*}
\left\{ 
\begin{array}{l}
\dfrac{\partial w}{\partial t}=\dfrac{\partial }{\partial t}(u,v_{0})=\left( 
\dfrac{\partial u}{\partial t},\dfrac{\partial v_{0}}{\partial t}\right) =%
\mathcal{P}w=\mathcal{P}(u,v_{0}) \\ \ecart
w(0)=\left( u(0,.),v_{0}(0,.)\right) \text{ given,}%
\end{array}%
\right.
\end{equation*}%
writes%
\begin{equation*}
\left\{ 
\begin{array}{l}
\dfrac{\partial u}{\partial t}=\Delta u\medskip \\ 
\dfrac{\partial v_{0}}{\partial t}=\left( \dfrac{\partial u}{\partial x}\right)
_{|\Gamma _{0}}+\dfrac{\partial ^{2}v_{0}}{\partial y^{2}}\medskip \\ 
u_{|\gamma _{0}\cup \gamma _{1}\cup \Gamma _{1}}=0 \\ \ecart 
\left( u(0,.),v_{0}(0,.)\right) \text{ given;}%
\end{array}%
\right.
\end{equation*}%
since $(u,v_{0})\in D(\mathcal{P})$ and $\dfrac{\partial u}{\partial t}%
=\Delta u,$ we obtain 
\begin{equation*}
\left( \Delta u\right) _{|\Gamma _{0}}=\left( \frac{\partial u}{\partial x}\right)
_{|\Gamma _{0}}+\left( \frac{\partial ^{2}v_{0}}{\partial y^{2}}\right) _{_{|\Gamma
_{0}}}=\left( \frac{\partial u}{\partial t}\right) _{_{|\Gamma _{0}}},
\end{equation*}%
and since $u_{|\Gamma _{0}}=v_{0}$, by using the tangential derivative, we
obtain 
\begin{equation*}
\left( \frac{\partial ^{2}v_{0}}{\partial y^{2}}\right) _{_{|\Gamma _{0}}}=\left(
\frac{\partial ^{2}u}{\partial y^{2}}\right) _{_{|\Gamma _{0}}};
\end{equation*}%
summarizing up, we deduce the same equation as in Problem $(P):$%
\begin{equation*}
\left\{ 
\begin{array}{l}
\dfrac{\partial u}{\partial t}=\Delta u \vspace{0.1cm}\\ 
\dis \left( \frac{\partial u}{\partial t}\right) _{_{|\Gamma _{0}}}=\left( \frac{\partial
u}{\partial x}\right) _{|\Gamma _{0}}+\left( \frac{\partial ^{2}u}{\partial
y^{2}}\right) _{_{|\Gamma _{0}}} \vspace{0.1cm} \\ 
u(0,.)\text{ is given} \vspace{0.1cm} \\ 
u_{|\gamma _{0}\cup \gamma _{1}\cup \Gamma _{1}}=0.%
\end{array}%
\right.
\end{equation*}%
The study of the evolution equation above is based on the study of the following spectral equation 
\begin{equation}
\left\{ 
\begin{array}{l}
\mathcal{P}(u,v_{0})-\lambda (u,v_{0})=(h,d_{0}) \vspace{0.1cm} \\ 
(u,v_{0})\in D(\mathcal{P}),(h,d_{0})\in \mathcal{E},%
\end{array}%
\right.  \label{Final Pb}
\end{equation}%
and since $u_{|\Gamma _{0}}=v_{0}$, (\ref{Final Pb}) is equivalent to%
\begin{equation}
\left\{ 
\begin{array}{l}
\Delta u-\lambda u=h \vspace{0.1cm}\\ 
\dis \left( \frac{\partial u}{\partial x}\right) _{|\Gamma _{0}}+\left( \frac{\partial
^{2}u}{\partial y^{2}}\right) _{_{|\Gamma _{0}}}-\lambda u_{|\Gamma _{0}}=d_{0}
\vspace{0.1cm}\\ 
u_{|\gamma _{0}\cup \gamma _{1}\cup \Gamma _{0}}=0,%
\end{array}%
\right.  \label{Final eq}
\end{equation}%
which is an elliptic partial differential equation with the same spectral
parameter in the equation and in the boundary condition on $\Gamma _{0}$. We will write (\ref{Final eq}) in an operational differential form. We
consider the Banach space $X=L^{p}(0,1)$ and identify $\mathcal{E}$ with $%
L^{p}(0,1;X)$ by writing as usual, for $g\in \mathcal{E}$, $g\left( x,y\right) =\left( g(x)\right) \left( y\right)$, $x,y\in
\left( 0,1\right)$. We define operator $A$ on $X$ by 
\begin{equation}
\left\{ 
\begin{array}{l}
D(A)=\left\{ \psi \in W^{2,p}(0,1):\psi \left( 0\right) =\psi \left(
1\right) =0\right\} \vspace{0.1cm}\\ 
A\psi (y)=\psi ^{\prime \prime }(y),%
\end{array}%
\right.  \label{def A}
\end{equation}%
and operator $H:=-A$. So, equation $\Delta u(x,y)-\lambda u(x,y)=h(x,y),$ takes the following form in space $X$%
\begin{equation*}
u^{\prime \prime }(x)+Au(x)-\lambda u(x)=h(x),\ x\in (0,1),
\end{equation*}%
while the boundary condition%
\begin{equation*}
\left( \dfrac{\partial u}{\partial x}\right) _{|\Gamma _{0}}+\left( \frac{\partial
^{2}u}{\partial y^{2}}\right) _{_{|\Gamma _{0}}}-\lambda u_{|\Gamma
_{0}}=d_{0},
\end{equation*}%
becomes $u'(0)-Hu(0)-\lambda u(0)=d_{0}$; the condition $u_{|\gamma _{0}\cup \gamma _{1}}=0$ (which means that $u(0,y)$
and $u(1,y)$ vanish in $y=0$ and $y=1$) is implicitly included in the fact
that $u(0):=u(0,.)$ and $u(1):=u(1,.)$ are in $D(H).$

Therefore (\ref{Final eq}) or equivalently (\ref{Final Pb}), takes the
following abstract form
\begin{equation}
\left\{ 
\begin{array}{l}
u^{\prime \prime }(x)+Au(x)-\lambda u(x)=h(x),\ x\in (0,1) \vspace{0.1cm}\\ 
u^{\prime }(0)-Hu(0)-\lambda u(0)=d_{0} \vspace{0.1cm}\\ 
u(1)=0,%
\end{array}%
\right.  \label{Abstract form}
\end{equation}%
where $\left( h,d_{0}\right) \in \mathcal{E}\equiv L^{p}(0,1;X)\times
L^{p}(X)$, and we are in the situation of Subsection \ref{espace produit}
with $\mu =0$.

Let $u$ be the classical solution of (\ref{Abstract form}); then $u\in
W^{2,p}(0,1;X)\cap L^{p}(0,1;D(A))$ and%
\begin{equation*}
(u,u\left( 0\right) )\in D\left( \mathcal{P}\right) ;
\end{equation*}%
so that $(u,u\left( 0\right) )=(\mathcal{P}-\lambda I)^{-1}\left(
h,d_{0}\right) .$

Taking into account the fact that, here, we can take $\varphi _{0}=\pi
-\varepsilon $ ($\varepsilon >0$ as close to $0$ as we want), we can use
Proposition \ref{Est L first case} and Theorem \ref{gen sg first case}, to
obtain : 
\begin{equation*}
\exists M>0,~\forall \lambda \in S_{\varphi _{0}}:\forall \left(
h,d_{0}\right) \in \mathcal{E}, \quad \left\Vert (\mathcal{P}-\lambda I)^{-1}\left(
h,d_{0}\right) \right\Vert _{\mathcal{E}}\leqslant \dfrac{M}{1+|\lambda|}%
\left\Vert \left( h,d_{0}\right) \right\Vert _{\mathcal{E}},
\end{equation*}%
and deduce that our operator $\mathcal{P}$ defined above generates an
analytic semigroup in $\mathcal{E}$.

This example can be extended to the following problem%
\begin{equation*}
\left\{ 
\begin{array}{l}
\Delta u-\lambda u=h \\ 
a_{0}\left( \dfrac{\partial u}{\partial x}\right) _{|\Gamma _{0}}+b_{0}\dfrac{%
\partial ^{2}v_{0}}{\partial y^{2}}-\lambda v_{0}=d_{0} \\ 
a_{1}\left( \dfrac{\partial u}{\partial x}\right) _{|\Gamma _{1}}+b_{1}\dfrac{%
\partial ^{2}v_{1}}{\partial y^{2}}-\lambda v_{1}=d_{1} \\ 
u_{|\gamma _{0}\cup \gamma _{1}}=0.%
\end{array}%
\right.
\end{equation*}

\subsection{Some concrete examples for the second case}

\subsubsection{Example 1}

Here, we set $\Omega = (0,1)\times(0,1)$. Our concrete spectral partial differential problem is 
\begin{equation*}
(P1)\left\{ 
\begin{array}{ll}
\dfrac{\partial ^{2}u}{\partial x^{2}}(x,y)+\dfrac{\partial ^{2}u}{\partial
y^{2}}(x,y)-\lambda u\left( x,y\right) =f\left( x,y\right) , & \left(
x,y\right) \in \Omega \vspace{0.1cm}\\ 
u\left( 1,y\right) =0, & y\in (0,1)\vspace{0.1cm} \\ 
\dis \dfrac{\partial u}{\partial x}\left( 0,y\right) - \int_{0}^{y}\phi \left(
y,\xi \right) u\left( 0,\xi \right) d\xi =0, & y\in (0,1) \vspace{0.1cm}\\ 
u(x,0)=u(x,1)=0, & x\in (0,1),%
\end{array}%
\right.
\end{equation*}%
where we can take $\lambda \in S_{\varphi _{0}}$with $\varphi _{0}$ fixed in 
$(\pi /2,\pi )$.

Define operator $A$ on $X:=L^{p}(0,1)$, with $1<p<+\infty $, as in (\ref{def A});
then the square root of the negative of this operator is well defined and 
\begin{equation*}
W_{0}^{1,p}(0,1)\subset D((-A)^{1/2})\subset W^{1,p}(0,1) \quad \text{and} \quad \left\Vert (-A)^{1/2}\psi \right\Vert \approx \left\Vert \psi ^{\prime}\right\Vert _{L^{p}(0,1)}+\left\Vert \psi \right\Vert _{L^{p}(0,1)},
\end{equation*}%
see \cite{Auscher}. We know also that $Q=-\sqrt{-A}$ generates an analytic
semigroup in $X$; on the other hand $Q_{\lambda }=-\sqrt{-A+\lambda I}$ is
well defined and generates an analytic semigroup in $X$ for all $\lambda \in
S_{\varphi _{0}}$.

Now let us define operator $H$ by%
\begin{equation}
H\psi (y)=\int_{0}^{y}\phi (y,\xi )\psi (\xi )d\xi ,\quad \psi \in X,
\label{def H}
\end{equation}%
with an appropiate function $\phi $ having the following properties. Let $%
q\in (1,+\infty )$ such that $1/q+1/p=1$. We then assume that

\begin{equation}
\left\{ 
\begin{array}{l}
\phi \left( y,\cdot \right) ,\dfrac{\partial \phi }{\partial y}\left(
y,\cdot \right) \in L^{q}(0,1)\text{, for a.e. }y\in (0,1) \vspace{0.1cm}\\ 
\phi \left( 1,\cdot \right) =0 \vspace{0.1cm}\\ 
\Phi _{j}:y\longmapsto \dfrac{\partial ^{j}\phi }{\partial y^{j}}\left(
y,\cdot \right) \in L^{p}(0,1;L^{q}(0,1))\text{, for }j=0,1\vspace{0.1cm} \\ 
\phi _{1}:y\longmapsto \phi \left( y,y\right) \in L^{p}(0,1).%
\end{array}%
\right.  \label{cond phi}
\end{equation}%
We can build a simple example of a function $\phi $ satisfying (\ref{cond phi}%
), setting, for a fixed $n\in \mathbb{N}\backslash \left\{ 0\right\} $%
\begin{equation*}
\phi \left( y,\xi \right) =\left( 1-y\right) ^{n}\widetilde{\psi }\left( \xi
\right) ,~\xi ,y\in (0,1),
\end{equation*}%
where $\widetilde{\psi }\in W^{1,q}(0,1)\cap W^{1,p}(0,1)$. We have%
\begin{eqnarray*}
\left\Vert H(\psi )\right\Vert _{X} &=&\left( \int_{0}^{1}\left\vert
\int_{0}^{y}\phi \left( y,\xi \right) \psi (\xi )d\xi \right\vert
^{p}dy\right) ^{1/p} \\
&\leqslant &\left( \int_{0}^{1}\left[ \left( \int_{0}^{1}\left\vert \phi
\left( y,\xi \right) \right\vert ^{q}d\xi \right) ^{1/q}\left(
\int_{0}^{1}\left\vert \psi (\xi )\right\vert ^{p}d\xi \right) ^{1/p}\right]
^{p}dy\right) ^{1/p} \\
&\leqslant &\left( \int_{0}^{1}\left\Vert \phi \left( y,\cdot \right)
\right\Vert _{L^{q}(0,1)}^{p}dy\right) ^{1/p}\left\Vert \psi \right\Vert _{X}
\\
&\leqslant &\left\Vert \Phi \right\Vert _{L^{p}(0,1;L^{q}(0,1))}\times
\left\Vert \psi \right\Vert _{X},
\end{eqnarray*}

so $H\in \mathcal{L}(X)$.

Our concrete problem $(P1)$ writes in the following abstract form%
\begin{equation*}
\left\{ 
\begin{array}{l}
u^{\prime \prime }\left( x\right) +Au\left( x\right) -\lambda u\left(
x\right) =f\left( x\right) ,\text{~~a.e. }x\in \left( 0,1\right) \vspace{0.1cm} \\ 
u\left( 1\right) =0\text{, \ }u^{\prime }(0)-Hu(0)=0.
\end{array}\right.
\end{equation*}

The following assumptions are satisfied:

\begin{enumerate}
\item $X$ is a UMD space and operator $A$ verifies 
\begin{equation*}
\left\{ 
\begin{array}{l}
\exists \text{ }\varphi _{0}\in \left( 0,\pi \right) :\text{ }S_{\varphi
_{0}}\subset \rho \left( A\right) \text{ and }\exists C_{A}>0: \vspace{0.1cm} \\ 
\forall \lambda \in S_{\varphi _{0}},\quad \left\Vert \left( A-\lambda
I\right) ^{-1}\right\Vert _{\mathcal{L}(X)}\leqslant \dfrac{C_{A}}{%
1+\left\vert \lambda \right\vert },%
\end{array}%
\right.
\end{equation*}%
\begin{equation*}
\left\{ 
\begin{array}{l}
\forall s\in \mathbb{R},\text{ }\left( -A\right) ^{is}\in \mathcal{L}\left( X\right) ,\text{ }\exists \theta _{A}\in \left( 0,\pi \right) \text{:} \vspace{0.1cm} \\ 
\underset{s\in \mathbb{R}}{\sup }\left\Vert e^{-\theta _{A}\left\vert s\right\vert
}(-A)^{is}\right\Vert _{\mathcal{L}\left( X\right) }<+\infty.
\end{array}\right.
\end{equation*}
This last property is proved explicitely in \cite{Labbas Moussaoui}.

\item Since $H$ is bounded, from Remark \ref{second case} statement 1, we
get $D(Q)\subset D\left( H\right) $ and
\begin{equation*}
\exists \,C_{H,Q}>0,\quad \underset{t\in \lbrack 0,+\infty )}{\sup }\left(
1+t\right) ^{1/2}\parallel HQ_{t}^{-1}\parallel _{\mathcal{L}\left( X\right)
}\leqslant C_{H,Q}.
\end{equation*}

\item We verifiy that $\left( Q-H\right) ^{-1}\left( D\left( Q\right) \right)
\subset D\left( Q^{2}\right)$.

Let $\psi \in D\left( Q\right) $ such that $\left( Q-H\right) \left( \psi
\right) \in D\left( Q\right) $; then $Q\psi -H\psi =g\in D\left( Q\right) ,$ with%
\begin{equation*}
W_{0}^{1,p}(0,1)\subset D\left( Q\right) \subset W^{1,p}(0,1).
\end{equation*}%
To obtain $\psi \in D\left( Q^{2}\right) $, it suffices to have $H\psi \in
W_{0}^{1,p}(0,1)$ for $\psi \in D\left( Q\right) \subset W^{1,p}(0,1)$. We
have
\begin{equation*}
H\psi (y)=\int_{0}^{y}\phi \left( y,\xi \right) \psi (\xi )d\xi;
\end{equation*}
then $H\psi (0)=0,$ and $H\psi (1)=0$ due to (\ref{cond phi}) and 
\begin{equation*}
\left( H\psi \right) ^{\prime }(y)=\phi \left( y,y\right) \psi
(y)+\int_{0}^{y}\dfrac{\partial \phi }{\partial y}\left( y,\xi \right) \psi
(\xi )d\xi .
\end{equation*}%
In virtue of the assumptions verified by $\phi $, we then get $H\psi \in
W_{0}^{1,p}(0,1)$. Therefore $\psi \in D\left( Q^{2}\right) .$
\end{enumerate}

Now, we set $Y=L^{p}(0,1;X)=L^{p}(\Omega )$ and considering $A$, $H$ defined
by (\ref{def A}) and (\ref{def H}), we build, as in (\ref{Def L A H Mu}) 
\begin{equation*}
\begin{array}{llll}
\mathcal{L}_{A,H,0}: & D\left( \mathcal{L}_{A,H,0}\right) \subset Y & 
\longrightarrow & Y \\ 
& u & \longmapsto & u^{\prime \prime }+A(u(.)).
\end{array}
\end{equation*}%
Note that in this example, in general, operators $Q$ and $H$ do not commute. We can apply Theorem \ref{Gen second case} (with $\mu =0$), to obtain that $\mathcal{L}_{A,H,0}$ is the infinitesimal generator of an analytic semigroup. This result allows us to consider and solve the corresponding Cauchy problem
with respect to $(P1).$

\subsubsection{Example 2}

Here, we are considering a quasi-elliptic problem under an oblique derivative
boundary condition. Let $\Omega=(0,1)^2$ and consider the following spectral problem%
\begin{equation*}
(P2)\left\{ 
\begin{array}{ll}
\dfrac{\partial ^{2}u}{\partial x^{2}}(x,y)-\dfrac{\partial ^{4}u}{\partial
y^{4}}(x,y)-\lambda u\left( x,y\right) =f\left( x,y\right) , & \left(
x,y\right) \in \Omega\smallskip \\ 
u\left( 1,y\right) =0, & y\in (0,1) \vspace{0.1cm}\\ 
\dfrac{\partial u}{\partial x}\left( 0,y\right) +c(y)\dfrac{\partial u}{%
\partial y}\left( 0,y\right) =0, & y\in (0,1) \vspace{0.1cm}\\ 
u(x,0)=u(x,1)=\dfrac{\partial ^{2}u}{\partial y^{2}}(x,0)=\dfrac{\partial
^{2}u}{\partial y^{2}}(x,0)=0, & x\in (0,1).%
\end{array}%
\right.
\end{equation*}%
We will assume that $c\in \mathcal{C}^{2}[0,1]$ and $c\left( 0\right) =c\left( 1\right) =0.$ Here the boundary condition on $\Gamma =\left\{ 0\right\} \times (0,1)$
\begin{equation*}
\dfrac{\partial u}{\partial x}\left( 0,y\right) +c(y)\dfrac{\partial u}{%
\partial y}\left( 0,y\right) =0,
\end{equation*}%
can be written as%
\begin{equation}
\nabla u(\sigma )\cdot \alpha (\sigma )=0\text{ in }\Gamma ,
\label{OblicDerivative}
\end{equation}%
with $\alpha (\sigma )$ a vector on $\Gamma $ equal to $\left( 1,c(y)\right)$ which is pointing inwardly of $\Omega$. It is known that (\ref{OblicDerivative}) is called oblique derivative boundary condition on $\Gamma $. We set, in space $X=L^{p}(0,1),$ as above 
\begin{equation} \label{New def A}
\left\{ \begin{array}{l}
D(A)=\left\{ \psi \in W^{4,p}(0,1):\psi \left( 0\right) =\psi \left(
1\right) =\psi ^{\prime \prime }\left( 0\right) =\psi ^{\prime \prime
}\left( 1\right) =0\right\} \vspace{0.1cm}\\ 
A\psi (y)=-\psi ^{(4)}(y);
\end{array}\right. 
\end{equation}
so, as we have seen 
\begin{equation*}
\left\{ 
\begin{array}{l}
D(\sqrt{-A})=\left\{ \psi \in W^{2,p}(0,1):\psi \left( 0\right) =\psi \left(
1\right) =0\right\} \vspace{0.1cm}\\ 
\sqrt{-A}\psi (y)=-\psi ^{\prime \prime }(y),%
\end{array}%
\right.
\end{equation*}%
and clearly $Q=-\sqrt{-A}$ and $Q_{\lambda }=-\sqrt{-A+\lambda I}$, for all $%
\lambda \in S_{\varphi }$ generate analytic semigroups in $X$. We note also
that $\sqrt{-Q}=(-A)^{1/4}$ is well defined and%
\begin{equation*}
W_{0}^{1,p}(0,1)\subset D((-A)^{1/4})\subset W^{1,p}(0,1)\quad \text{and} \quad \left\Vert (-A)^{1/4}\psi \right\Vert \approx \left\Vert \psi ^{\prime}\right\Vert _{L^{p}(0,1)}+\left\Vert \psi \right\Vert _{L^{p}(0,1)},
\end{equation*}%
see \cite{Auscher}. Now, define operator $H$ by setting  
\begin{equation}
\left\{ 
\begin{array}{l}
D(H)=W^{1,p}(0,1) \vspace{0.1cm}\\ 
\left[ H\psi \right] (y)=-c(y)\psi ^{\prime }(y).%
\end{array}%
\right.  \label{New def H}
\end{equation}%
We then have $D(\left( -A\right) ^{1/4})\subset D\left( H\right)$;
therefore, see Remark \ref{second case}, statement 1, with $\omega=1/4$, there exists $C>0$ such that, for $t\geqslant 0$, we have
\begin{equation*}
\parallel HQ_{t}^{-1}\parallel _{\mathcal{L}\left( X\right) } \leqslant \dfrac{C}{\left( 1+t\right) ^{1/4}}.
\end{equation*}%
Now, we will prove that $\left( Q-H\right) ^{-1}\left( D\left( Q\right) \right)
\subset D\left( Q^{2}\right) $. To this end, let $\psi \in D\left( Q\right) $ such that $%
\left( Q-H\right) \left( \psi \right) \in D\left( Q\right) $; then 
\begin{equation*}
\psi ^{\prime \prime }-c\psi ^{\prime }=g\in D\left( Q\right)
=W^{2,p}(0,1)\cap W_{0}^{1,p}(0,1);
\end{equation*}
so $\psi \in W^{4,p}\left( 0,1\right) $. We have $\psi \in D\left( Q\right) $; then $\psi \left( 0\right) =\psi \left( 1\right) =0$. But $g\in D\left(
Q\right) $ thus $g(0)=g\left( 1\right) =0$ and 
\begin{equation*}
\psi ^{\prime \prime }\left( j\right) =\left( c\psi ^{\prime }\right) \left(
j\right) +g\left( j\right) =0,\quad j=0,1,
\end{equation*}%
that is $\psi ^{\prime \prime }\left( 0\right) =\psi ^{\prime \prime }\left(
1\right) =0$, therefore $\psi \in D\left( Q^{2}\right) $. Note that in this example $Q-H$ is boundedly invertible and from equation 
$Q\psi -H\psi =g,$ it follows that%
\begin{equation*}
\left\{ 
\begin{array}{l}
\psi ^{\prime \prime }(y)-c(y)\psi ^{\prime }(y)=g(y) \vspace{0.1cm}\\ 
\psi (0)=\psi (1)=0.%
\end{array}%
\right.
\end{equation*}%
Let $\psi _{1}$ and $\psi _{2}$ two linearly independent solutions to equation $\psi ^{\prime \prime }(y)-c(y)\psi ^{\prime }(y)=0,$ such that $\psi _{1}(0)=0$ and $\psi _{2}(1)=0$. Then we have 
$$\psi (y) = -\psi _{2}(y)\int_{0}^{y}\frac{\psi _{1}(s)}{W(s)}g(s)ds-\psi
_{1}(y)\int_{y}^{1}\frac{\psi _{2}(s)}{W(s)}g(s)ds = \left[ \left( Q-H\right) ^{-1}g\right] (y),$$ 
where the wronskian $W$ is given by
\begin{equation*}
W(s)=\psi _{1}(s)\psi _{2}^{\prime }(s)-\psi _{2}(s)\psi _{1}^{\prime }(s).
\end{equation*}%
We have%
\begin{equation*}
\psi ^{\prime }(y)=-\psi _{2}^{\prime }(y)\int_{0}^{y}\frac{\psi _{1}(s)}{%
W(s)}g(s)ds-\psi _{1}^{\prime }(y)\int_{y}^{1}\frac{\psi _{2}(s)}{W(s)}g(s)ds,
\end{equation*}%
and 
$$
\psi ^{\prime \prime }(y) = -\psi _{2}^{\prime \prime }(y)\int_{0}^{y}\frac{%
\psi _{1}(s)}{W(s)}g(s)ds-\psi _{1}^{\prime \prime }(y)\int_{y}^{1}\frac{%
\psi _{2}(s)}{W(s)}g(s)ds +g(y).
$$
If $g\in D\left( Q\right) =W^{2,p}(0,1)\cap W_{0}^{1,p}(0,1)$, it is clear
that $\psi \in W^{4,p}(0,1)$ and 
$$\psi''(0) = g(0)-\psi _{1}^{\prime \prime }(0)\int_{0}^{1}%
\frac{\psi _{2}(s)}{W(s)}g(s)ds = 0-\left[ c(0)\psi _{1}^{\prime }(0)\right] \int_{0}^{1}\frac{\psi _{2}(s)}{W(s)}g(s)ds = 0;$$
similarly we obtain $\psi ^{\prime \prime }(1)=0.$

Again, our concrete problem $(P3)$ writes in the abstract form%
\begin{equation*}
\left\{ 
\begin{array}{l}
u^{\prime \prime }\left( x\right) +Au\left( x\right) -\lambda u\left(
x\right) =f\left( x\right) ,\quad \text{for a.e. }x\in \left( 0,1\right) \vspace{0.1cm}\\ 
u\left( 1\right) =0\text{, \ }u^{\prime }(0)-Hu(0)=0,
\end{array}\right.
\end{equation*}%
with $A$ and $H$ defined by (\ref{New def A}), (\ref{New def H}) and setting 
\begin{equation*}
\begin{array}{llll}
\mathcal{L}_{A,H,0}: & D\left( \mathcal{L}_{A,H,0}\right) \subset Y & 
\longrightarrow & Y \\ 
& u & \longmapsto & u^{\prime \prime }+A(u(.)).
\end{array}
\end{equation*}%
We can apply Theorem \ref{Gen second case} (with $\mu =0$),{\Large \ }to
obtain that $\mathcal{L}_{A,H,0}$ is the infinitesimal generator of an
analytic semigroup.

\subsubsection{Example 3}

In \cite{Krasnoschok} the authors have considered and studied the following problem 
\begin{equation*}
\left \{ 
\begin{array}{ll}
\dfrac{\partial ^{2}u}{\partial x^{2}}(x,y,t)+\dfrac{\partial ^{2}u}{\partial y^{2}}(x,y,t)=f\left( x,y,t\right) , & (x,y,t) \in \mathbb{R}_{+}\times \mathbb{R}\times (0,T)\smallskip \\ 
u\left( 0,y,0\right) =f_{1}\left( y\right) , & y\in \mathbb{R}\vspace{0.1cm}\\ 
\dfrac{\partial u}{\partial x}\left( 0,y,t\right) -D_{t}^{\nu
}u(0,y,t)=f_{2}\left( y,t\right), & (y,t)\in \mathbb{R}\times (0,T)\smallskip ,%
\end{array}
\right.
\end{equation*}%
where $D_{t}^{\nu }$, for $\nu \in (0,1)$, denotes the fractional time
derivative (or Caputo Derivative) defined, for instance, by
\begin{equation*}
D_{t}^{\nu }g(.,t)=\frac{1}{\Gamma (1-\nu )}\int_{0}^{t}\frac{1}{(t-\tau
)^{\nu }}\frac{\partial g}{\partial \tau }(.,\tau )d\tau ,
\end{equation*}%
for functions $g$ of classe $C^{1}$ with respect to the second variable; for
this derivative, see for instance \cite{Chen}. This derivative has been
extended to functions in $L_{loc}^{1}\left( \mathbb{R}\right) $ verifying some integrability condition, see \cite{Lei}.

Analysis of the above problem is useful to study the free boundary problem
for the Laplace equation in the case of subdiffusion as illustrated by the
fractional derivative, see \cite{Voller}. We recall that this subdiffusion
expressed by this Caputo Derivative means that the square displacement of
the diffusing species has a behaviour as $t^{\nu }$ for some real number $%
\nu $. When $\nu \in (0,1)$, we are in the presence of a subdiffusion.

Our objective is not to study this problem, but it helps us to consider a
class of similar problems illustrating our theory of the second case. So, setting $\Omega_T= (0,1)\times (0,1)\times (0,T)$, we will take inspiration from this example to consider the following spectral
elliptic problem:
\begin{equation*}
(P4)\left\{ 
\begin{array}{ll}
\dfrac{\partial ^{2}u}{\partial x^{2}}(x,y,t)+\dfrac{\partial ^{2}u}{\partial y^{2}}(x,y,t)-\lambda u\left( x,y,t\right) =f\left( x,y,t\right), & (x,y,t) \in \Omega_T \smallskip \\ 
u\left( 1,y,0\right) =f_{1}\left( y\right) , & y\in (0,1) \vspace{0.1cm} \\ 
\dfrac{\partial u}{\partial x}\left( 0,y,t\right) -D_{t}^{\nu}u(0,y,t)=f_{2}\left( y,t\right) , & (y,t)\in (0,1)\times(0,T),%
\end{array}%
\right.
\end{equation*}%
for $\lambda \in S_{\varphi _{0}}$with $\varphi _{0}\in (\pi /2,\pi )$.

In view to write this problem in an abstract form, we will hide the variable 
$(y,t)$ by considering the following anisotropic Sobolev Banach space $X=W_{p}^{0,1}((0,1)\times (0,T)),$ consisting of all functions $(y,t)\longmapsto w(y,t)$ which are in $L^{p}((0,1)\times (0,T))$ such that we have $\dfrac{\partial w}{\partial t}\in
L^{p}((0,1)\times (0,T));$ it is endowed with the following natural norm%
\begin{equation*}
\left\Vert w\right\Vert _{X}=\left\Vert w\right\Vert _{L^{p}((0,1)\times
(0,T))}+\left\Vert \dfrac{\partial w}{\partial t}\right\Vert
_{L^{p}((0,1)\times (0,T))}.
\end{equation*}
Now, define operator $A$ in $X$ by%
\begin{equation*}
\left\{ 
\begin{array}{l}
D(A)=\left\{ w\in X:\dfrac{\partial w}{\partial y},\dfrac{\partial ^{2}w}{%
\partial y^{2}}\in L^{p}(\mathbb{R}\times (0,T))\text{ and } w(0,t)=w(1,t)=0\text{ for }t\in (0,T)\right\} \\ 
\left[ Aw\right] (y,t)=\dfrac{\partial ^{2}w}{\partial y^{2}}(y,t).
\end{array}
\right.
\end{equation*}%
We also define $H$ by%
\begin{equation*}
\left\{ 
\begin{array}{l}
D(H)=W_{p}^{0,1}(\mathbb{R}\times (0,T))=X \vspace{0.1cm}\\ 
\left[ Hw\right] (y,t)=D_{t}^{\nu }w(y,t).
\end{array}
\right.
\end{equation*}%
This problem can be written in the following abstract form:%
\begin{equation*}
\left\{ 
\begin{array}{l}
u^{\prime \prime }\left( x\right) +Au\left( x\right) -\lambda u\left(x\right) =f\left( x\right) ,\quad \text{for a.e. }x\in \left( 0,1\right) \vspace{0.1cm} \\ 
u\left( 1\right) =f_{1} \vspace{0.1cm} \\ 
u^{\prime }(0)-Hu(0)=f_{2},%
\end{array}%
\right.
\end{equation*}%
where we have used the usual writting $u(x,y,t)=u(x)(y,t)$ and $f(x,y,t)=f\left( x\right) (y,t)$. Now we must verify the following statements.

\begin{enumerate}
\item $X$ has the UMD property.

In fact, consider the application%
\begin{equation*}
\begin{array}{ccccc}
\mathcal{T} & : & W_{p}^{0,1}((0,1)\times (0,T)) & \longrightarrow & Z= 
\left[ L^{p}((0,1)\times (0,T))\right] ^{2}\vspace{0.1cm} \\ 
&  & w & \longmapsto & \left( w,\dfrac{\partial w}{\partial t}\right) ,%
\end{array}%
\end{equation*}%
then $\mathcal{T}\left( W_{p}^{0,1}((0,1)\times (0,T))\right) $ is a closed
subspace of $Z$ and thus has a UMD property. Since it is isometric to $X$ ,
we deduce that $X$ is a UMD space.

\item Operator $A$ verifies 
\begin{equation*}
\left\{ 
\begin{array}{l}
S_{\varphi _{0}}\subset \rho \left( A\right) \text{ and }\exists C_{A}>0: \vspace{0.1cm} \\ 
\forall \lambda \in S_{\varphi _{0}},\text{ }\left\Vert \left( A-\lambda
I\right) ^{-1}\right\Vert _{\mathcal{L}(X)}\leqslant \dfrac{C_{A}}{%
1+\left\vert \lambda \right\vert },%
\end{array}%
\right.
\end{equation*}
and
\begin{equation*}
\left\{ 
\begin{array}{l}
\forall s\in \mathbb{R},\text{ }\left( -A\right) ^{is}\in \mathcal{L}\left( X\right) ,\text{ }%
\exists \theta _{A}\in \left( 0,\pi \right) \text{:} \vspace{0.1cm} \\ 
\underset{s\in \mathbb{R}}{\sup }\left\Vert e^{-\theta _{A}\left\vert s\right\vert
}(-A)^{is}\right\Vert _{\mathcal{L}\left( X\right) }<+\infty .
\end{array}\right.
\end{equation*}
For the first property we note that the spectral properties of operator $A$
are based on the equation%
\begin{equation*}
\left\{ 
\begin{array}{l}
\dfrac{\partial ^{2}w}{\partial y^{2}}(y,t)-\lambda w(y,t)=h(y,t) \vspace{0.1cm} \\ 
w(0,t)=w(1,t)=0\text{ for }t\in (0,T),%
\end{array}%
\right.
\end{equation*}%
where $h\in W_{p}^{0,1}((0,1)\times (0,T))$. Then, for all $\lambda \in
S_{\varphi _{0}}$, we have 
\begin{equation*}
\forall (y,t)\in (0,1)\times (0,T),\text{ \ }w(y,t)=\int_{0}^{1}K_{\sqrt{%
\lambda }}(y,s)h(s,t)ds,
\end{equation*}%
where the kernel $K_{\sqrt{\lambda }}(y,s)$ is well known. Using the Schur
Lemma, for all $t\in (0,1)$, we obtain
\begin{equation*}
\int_{0}^{1}\left\vert w(y,t)\right\vert ^{p}dy\leqslant \left[ \frac{C}{1+\left\vert \lambda \right\vert }\right] ^{p}\int_{0}^{1}\left\vert
h(s,t)\right\vert ^{p}ds;
\end{equation*}%
then%
\begin{equation*}
\int_{0}^{T}\int_{0}^{1}\left\vert w(y,t)\right\vert ^{p}dydt\leqslant \left[
\frac{C}{1+\left\vert \lambda \right\vert }\right] ^{p}\int_{0}^{T}%
\int_{0}^{1}\left\vert h(s,t)\right\vert ^{p}dsdt,
\end{equation*}%
that is 
\begin{equation*}
\left\Vert w\right\Vert _{L^{p}((0,1)\times (0,T))}\leqslant \frac{C}{%
1+\left\vert \lambda \right\vert }\left\Vert h\right\Vert
_{L^{p}((0,1)\times (0,T))}.
\end{equation*}%
Since we have%
\begin{equation*}
\forall (y,t)\in (0,1)\times (0,T),\quad \dfrac{\partial w}{\partial t}%
(y,t)=\int_{0}^{1}K_{\sqrt{\lambda }}(y,s)\dfrac{\partial h}{\partial t}%
(s,t)ds,
\end{equation*}%
we deduce%
\begin{equation*}
\left\Vert \dfrac{\partial w}{\partial t}\right\Vert _{L^{p}((0,1)\times
(0,T))}\leqslant \frac{C}{1+\left\vert \lambda \right\vert }\left\Vert 
\dfrac{\partial h}{\partial t}\right\Vert _{L^{p}((0,1)\times (0,T))},
\end{equation*}%
and then%
\begin{equation*}
\left\Vert w\right\Vert _{X}\leqslant \frac{C}{1+\left\vert \lambda
\right\vert }\left\Vert h\right\Vert _{X}.
\end{equation*}%
The second property is proved explicitely in \cite{Labbas Moussaoui}.

\item Since $H$ is bounded then from Remark \ref{second case}, statement 1, $%
D(Q)\subset D\left( H\right) $ and
\begin{equation*}
\exists C_{H,Q}>0, \quad \underset{t\in \lbrack 0,+\infty )}{\sup }\left(
1+t\right) ^{1/2}\parallel HQ_{t}^{-1}\parallel _{\mathcal{L}\left( X\right)
}\leqslant C_{H,Q}.
\end{equation*}

\item Now, we must verifiy that $\left( Q-H\right) ^{-1}\left( D\left( Q\right)
\right) \subset D\left( Q^{2}\right)$. It is enough to verify that $D_{t}^{\nu }A^{-1}=A^{-1}D_{t}^{\nu }$ on $X$.
We have%
\begin{equation*}
\forall (y,t)\in (0,1)\times (0,T),\quad \left[ A^{-1}w\right]
(y,t)=\int_{0}^{1}G(y,s)w(s,t)ds,
\end{equation*}%
where the kernel $G$ is well known. So, for any $(y,t)\in (0,1)\times (0,T)$ 
\begin{equation*}
\left[ D_{t}^{\nu }A^{-1}w\right] (y,t)=\int_{0}^{1}G(y,s)D_{t}^{\nu
}w(s,t)ds=\left[ A^{-1}D_{t}^{\nu }\right] w(y,t).
\end{equation*}
\end{enumerate}

Again, as in the previous examples , we get that $\mathcal{L}_{A,H,0}$ is
the infinitesimal generator of an analytic semigroup.

\begin{remark}
We can generalize the above examples by considering operator $A$ defined in an open bounded regular set $\omega$ of $\mathbb{R}^{n-1}$. 
\end{remark}

%

\section*{Declarations}

\textbf{Ethical Approval:} Not applicable.
 
\noindent\textbf{Competing interests:} Not applicable.

\noindent\textbf{Authors' contributions:} All the authors have written and reviewed this manuscript.

\noindent\textbf{Funding:} Not applicable.

\noindent\textbf{Availability of data and materials:} Not applicable.

\end{document}